\documentclass{amsart}
 \usepackage[foot]{amsaddr}
\let\oldtocsection=\tocsection
 
\let\oldtocsubsection=\tocsubsection
 
\let\oldtocsubsubsection=\tocsubsubsection
 
\renewcommand{\tocsection}[2]{\vspace{0.5em}\hspace{0em}\oldtocsection{#1}{#2}}
\renewcommand{\tocsubsection}[2]{\vspace{0.5em}\hspace{1em}\oldtocsubsection{#1}{#2}}
\renewcommand{\tocsubsubsection}[2]{\vspace{0.5em}\hspace{2em}\oldtocsubsubsection{#1}{#2}}
\usepackage{graphicx,cancel,xcolor,comment,graphicx,geometry}
 
\usepackage[utf8]{inputenc}
\usepackage[T1]{fontenc}

\usepackage{graphicx,url,etoolbox}

\usepackage{lipsum}
 
 \usepackage{caption}
\usepackage{subcaption}
   \DeclareMathOperator{\sign}{sign}

\usepackage{comment}

\usepackage{tikz}
\usepackage{tkz-tab}
\usepackage{caption}
\usepackage{latexsym}
\usepackage{amssymb}
\usepackage{amsmath}
\usepackage{subcaption}  
 \makeatletter
\patchcmd{\@settitle}{center}{flushleft}{}{}
\patchcmd{\@settitle}{center}{flushleft}{}{}
\patchcmd{\@setauthors}{\centering}{\raggedright}{}{}
\patchcmd{\abstract}{3pc}{0pt}{}{} % remove indentation
\makeatother

\makeatletter
\renewcommand*\@maketitle{%
  \normalfont\normalsize
  \@adminfootnotes
  \@mkboth{\@nx\shortauthors}{\@nx\shorttitle}%
  \global\topskip42\p@\relax % 5.5pc   "   "   "     "     "
  \@settitle
  \ifx\@empty\authors \else \@setauthors \fi
  \ifx\@empty\@date \else {\vskip 1em \vtop{\centering\large\@date\@@par}}\fi% MY CHANGE
  \ifx\@empty\@dedicatory
  \else
    \baselineskip18\p@
    \vtop{\centering{\footnotesize\itshape\@dedicatory\@@par}%
      \global\dimen@i\prevdepth}\prevdepth\dimen@i
  \fi
  \@setabstract
  \normalsize
  \if@titlepage
    \newpage
  \else
    \dimen@34\p@ \advance\dimen@-\baselineskip
    \vskip\dimen@\relax
  \fi
} % end \@maketitle
\renewcommand*\@adminfootnotes{%
  \let\@makefnmark\relax  \let\@thefnmark\relax
  \ifx\@empty\@subjclass\else \@footnotetext{\@setsubjclass}\fi
  \ifx\@empty\@keywords\else \@footnotetext{\@setkeywords}\fi
  \ifx\@empty\thankses\else \@footnotetext{%
    \def\par{\let\par\@par}\@setthanks}%
  \fi
}
\makeatother
 
\usepackage{fancyhdr}
\def\sf{\sin\left(\frac{c}{4}\right)}
\def\cf{\cos\left(\frac{c}{4}\right)}
\def\st{\sin\left(\frac{c}{2}\right)}
\def\ct{\cos\left(\frac{c}{2}\right)}

\def\shLt{\sinh\left(\frac{\lambda}{2}\right)}
\def\chLt{\cosh\left(\frac{\lambda}{2}\right)}
\def\shLtt{\sinh\left(\frac{3\lambda}{2}\right)}
\def\chLtt{\cosh\left(\frac{3\lambda}{2}\right)}

\usepackage{lastpage,comment}

\def\nline{\\ \noalign{\medskip}}
\newcounter{dummy} 
\numberwithin{dummy}{section}
\newtheorem{thm}[dummy]{Theorem}

\newtheorem{cor}[dummy]{Corollary}

\newtheorem{defi}[dummy]{Definition}

\newtheorem{lem}[dummy]{Lemma}

\newtheorem{prop}[dummy]{Proposition}
\newtheorem{rk}[dummy]{Remark}

\numberwithin{equation}{section}
 \renewenvironment{proof}{{\bfseries \noindent Proof.}}{\demo}
\newcommand\xqed[1]{%
  \leavevmode\unskip\penalty9999 \hbox{}\nobreak\hfill
  \quad\hbox{#1}}
\newcommand\demo{\xqed{$\square$}}
\RequirePackage[bookmarks,%
                colorlinks,%
                urlcolor=blue,%
                citecolor=blue,%
                linkcolor=red,%
                hyperfigures,%
                pagebackref=true,%
                pdfcreator=LaTeX,%
                breaklinks=true,%
                pdfpagelayout=SinglePage,%
                bookmarksopen=true,%
                bookmarksopenlevel=2]{hyperref}
\title[{Stabilization of the Timoshenko system with Kelvin-Voigt damping}]{A transmission problem for the  Timoshenko system with one local Kelvin-Voigt damping and non-smooth coefficient at the interface}        

\author{Mouhammad Ghader and Ali Wehbe  \vspace{0.58cm}\\
Lebanese University, Faculty of sciences 1, Khawarizmi Laboratory of Mathematics and Applications-KALMA, Hadath-Beirut. \\ \vspace{0.2cm}
 \\ \vspace{0.2cm}
Emails:  mhammadghader@hotmail.com and ali.wehbe@ul.edu.lb.
}
\begin{document}
\begin{abstract}
In this paper, we study the indirect stability of  Timoshenko system with  local or global  Kelvin–Voigt damping, under fully Dirichlet or mixed  boundary conditions.  Unlike \cite{Zhao--2004} and  \cite{Tian-Zhang-2017}, in this paper, we consider the Timoshenko system with only one locally or globally distributed Kelvin-Voigt damping $D$ (see System \eqref{E--(1.1)}). Indeed, we prove that the energy of the system decays  polynomially of type $t^{-1}$ and that this decay rate is in some sense optimal. The method is based on the frequency domain approach combining with multiplier method.\\[0.1in]

\vspace{-0.5cm}
\end{abstract}
\maketitle
%%%%%%%%%%%%%%%%%%%%%%%%%%%%%%%%%%%%%%%%%%%%%%%%%%
%%%%%%%%%%%%%%%%%%%%%%%%%%%%%%%%%%%%%%%%%%%%%%%%%%
%%%%%%%%%%%%%%%%%%%%%%%%%%%%%%%%%%%%%%%%%%%%%%%%%%
 % Section 1: Introduction %
%%%%%%%%%%%%%%%%%%%%%%%%%%%%%%%%%%%%%%%%%%%%%%%%%%
%%%%%%%%%%%%%%%%%%%%%%%%%%%%%%%%%%%%%%%%%%%%%%%%%%	
%%%%%%%%%%%%%%%%%%%%%%%%%%%%%%%%%%%%%%%%%%%%%%%%%%
\section{Introduction}\label{Section--1}
\noindent In this paper, we study the indirect stability of a one-dimensional Timoshenko system with only one local or global  Kelvin-Voigt damping. This system consists of two coupled hyperbolic equations:
%%%%%%%%%%%%%%%%%%%%Equation%%%%%%%%%%%%%%%%%%%%%%
\begin{equation} \label{E--(1.1)}
\begin{array}{lll}
\displaystyle{\rho_1 u_{tt}-k_1 \left(u_x+y\right)_x=0,}
&\displaystyle{(x,t)\in\left(0,L\right)\times \mathbb{R}_{+},}

         \nline
         
\displaystyle{ \rho_2y_{tt}-\left(k_2y_{x}+Dy_{xt}\right)_x+k_1\left(u_x+y\right)=0,}  &\displaystyle{(x,t)\in\left(0,L\right)\times \mathbb{R}_{+}.}
\end{array}
\end{equation}
System \eqref{E--(1.1)} is subject to the following initial conditions:
%%%%%%%%%%%%%%%%%%%%Equation%%%%%%%%%%%%%%%%%%%%%%
\begin{equation} \label{E--(1.2)}
\begin{array}{lll}
u(x,0)=u_0(x),& u_t(x,0)=u_1(x),&x\in   (0,L),\nline
y(x,0)=y_0(x),& y_t(x,0)=y_1(x),&x\in(0,L),
\end{array}
\end{equation}
in addition to  the following  boundary conditions:
%%%%%%%%%%%%%%%%%%%%Equation%%%%%%%%%%%%%%%%%%%%%%
\begin{equation}\label{E--(1.3)}
u(0,t)=y(0,t)=u(L,t)=y(L,t)=0, \quad \displaystyle{t\in \mathbb{R}_+, }
\end{equation}
or 
%%%%%%%%%%%%%%%%%%%%Equation%%%%%%%%%%%%%%%%%%%%%%
\begin{equation}\label{E--(1.4)}
u(0,t)=y_x(0,t)=u(L,t)=y_x(L,t)=0, \quad \displaystyle{t\in \mathbb{R}_+. }
\end{equation}
Here the coefficients $\rho_1,\ \rho_2,\ k_1$, and $k_2$  are strictly  positive constant numbers. The function $D\in L^{\infty}(0,L)$, such that $D(x)\geq0, \ \forall x\in [0, L]$.   We assume that there exist $D_0>0$, $\alpha,\ \beta\in \mathbb{R}$,  $0\leq \alpha < \beta\leq L,$ such that
%%%%%%%%%%%%%%%%%%%%Equation%%%%%%%%%%%%%%%%%%%%%%
\begin{equation}\tag{$\rm{H}$}
D\in  C\left([\alpha,\beta]\right)\ \ \ \text{and}\ \ \ D(x)\geq D_0>0\ \ \  \forall\ x\in (\alpha, \beta).
\end{equation}
The hypothesis {\rm (H)} means that the control $D$ can be locally near the boundary (see Figures \ref{Fig-1} and \ref{Fig-2}), or locally internal (see Figure \ref{Fig-3}), or globally  (see Figure \ref{Fig-4}). Indeed, in the case when $D$ is local damping (i.e., $\alpha\neq0$ or $\beta\neq L$), we see that $D$ is not necessary   continuous  over $(0,L)$ (see Figures \ref{Fig-1}, \ref{Fig-2}, and \ref{Fig-3}).$\\$
%%%%%%%%%%%%%%%%%%%%Figure%%%%%%%%%%%%%%%%%%%%%%%%
\begin{figure}[!h]
\centering
    \begin{subfigure}[b]{0.32\textwidth}
        \centering
        \resizebox{\linewidth}{!}{
            
           \begin{tikzpicture}[domain=0:4]
   
    \draw[->] (0,0) -- (4,0) node[right] {$x$};
    \draw[->] (0,0) -- (0,3) ;
    \coordinate [label=left:\textcolor{blue}{$0$}] (n1) at (0,0);              
 \coordinate [label=below:\textcolor{blue}{$\alpha$}] (n2) at (1,0) ;
  \coordinate [label=below:\textcolor{blue}{$\beta=L$}]  (n3) at (3,0)  ;   
     \draw[red] (n1)--(n2) ;
    \coordinate [label=above:\textcolor{red}{$D(x)$}](a1)  at (2,1)  ;
    \coordinate [label=above:\textcolor{red}{$D(x)$}](a2)  at (0.6,0.1)  ;
    \draw[color=red, domain=0:1]    plot (\x,0)  ;
    \draw[color=red, domain=1:3]    plot (\x,{sin(\x r)}) ;

\end{tikzpicture}

        }
    \caption{Figure 1a} 
\label{Fig-1}
    \end{subfigure}\hspace{1cm}
    \begin{subfigure}[b]{0.32\textwidth}
    \centering
        \resizebox{\linewidth}{!}{
       \begin{tikzpicture}[domain=0:4]
   
    \draw[->] (0,0) -- (4,0) node[right] {$x$};
    \draw[->] (0,0) -- (0,3);
    \coordinate [label=left:\textcolor{blue}{$0=\alpha$}] (n1) at (0,0);              
 \coordinate [label=below:\textcolor{blue}{$\beta$}] (n2) at (2,0) ;
  \coordinate [label=below:\textcolor{blue}{$L$}]  (n3) at (3,0)  ;  
  \coordinate [label=above:\textcolor{red}{$D(x)$}](a1)  at (1,0.5)  ;
    \coordinate [label=above:\textcolor{red}{$D(x)$}](a2)  at (2.6,0.1)  ; 
     \draw[red] (n2)--(n3) ;
    
    \draw[color=red, domain=2:3]    plot (\x,0)  ;
    \draw[color=red, domain=0:2]    plot (\x,{0.1*exp( \x )}) ;

\end{tikzpicture}	     
        }
   \caption{Figure 1b} 
\label{Fig-2}   
        
    \end{subfigure}
   
\end{figure}

\begin{figure}[!h]
\centering
    \begin{subfigure}[b]{0.32\textwidth}
        \centering
        \resizebox{\linewidth}{!}{
           \begin{tikzpicture}[domain=0:4]
   
    \draw[->] (0,0) -- (4,0) node[right] {$x$};
    \draw[->] (0,0) -- (0,3) ;
    \coordinate [label=left:\textcolor{blue}{$0$}] (n1) at (0,0);  
   \coordinate [label=below:\textcolor{blue}{$\alpha$}] (n2) at (1,0) ;             
 \coordinate [label=below:\textcolor{blue}{$\beta$}] (n2) at (2,0) ;
  \coordinate [label=below:\textcolor{blue}{$L$}]  (n3) at (3,0)  ;   
     \draw[red] (n2)--(n3) ;
    \draw[color=red, domain=0:1]    plot (\x,0)  ;
    \draw[color=red, domain=2:3]    plot (\x,0)  ;
    \draw[color=red, domain=1:2]    plot (\x,{0.3*exp( \x )}) ;
   
    \coordinate [label=above:\textcolor{red}{$D(x)$}](a1)  at (2.6,0.1)  ;
  \coordinate [label=above:\textcolor{red}{$D(x)$}](a3)  at (2,1)  ;
    \coordinate [label=above:\textcolor{red}{$D(x)$}](a2)  at (0.6,0.1)  ;
\end{tikzpicture}

        }
    \caption{Figure 2a} 
\label{Fig-3}
    \end{subfigure}\hspace{1cm}
    \begin{subfigure}[b]{0.32\textwidth}
    \centering
        \resizebox{\linewidth}{!}{
         \begin{tikzpicture}[domain=0:4]
   
    \draw[->] (0,0) -- (4,0) node[right] {$x$};
    \draw[->] (0,0) -- (0,3) ;
    \coordinate [label=left:\textcolor{blue}{$0=\alpha$}] (n1) at (0,0);  

  \coordinate [label=below:\textcolor{blue}{$L=\beta$}]  (n3) at (3,0)  ;   
     \coordinate [label=above:\textcolor{red}{$D(x)$}](a3)  at (2,1)  ;

    \draw[color=red, domain=0:3]    plot (\x,{0.1*exp( \x )}) ;

\end{tikzpicture}	  
        }
   \caption{Figure 2b} 
\label{Fig-4}   
        
    \end{subfigure}
 \end{figure}
\noindent The  Timoshenko system is usually considered in describing the transverse vibration of a beam and ignoring damping effects of any nature. Indeed, we have the following model, see in \cite{Timoshenko01},
%%%%%%%%%%%%%%%%%%%%Equation%%%%%%%%%%%%%%%%%%%%%%
\begin{equation*}
\left\{
\begin{array}{lll}

\displaystyle{\rho\varphi_{tt}= \left(K\left(\varphi_x-\psi\right)\right)_x }

\nline

\displaystyle{I_{\rho} \psi_{tt}=\left( EI\psi_x\right)_{x}-K\left(\varphi_x-\psi\right),}

\end{array}
\right.
\end{equation*}
where $\varphi$ is the transverse displacement of the beam and $\psi$ is the rotation angle of the filament of the beam. The coefficients $\rho,\ I_\rho,\ E,\ I,$ and $K$ are respectively the density (the mass per unit length), the polar moment of inertia of a cross section, Young's
modulus of elasticity, the moment of inertia of a cross section and the shear modulus respectively. \\[0.1in] 
%%%%%%%%%%%%%%%%%%%%%%%%%%%%%%%%%%%%%%%%%%%%%%%%%%
The stabilization of the Timoshenko system with different kinds of damping has been studied in  number of publications.  For the internal stabilization, Raposo and al. in \cite{Santos03} showed that the Timoshenko system with two internal distributed dissipation is exponentially stable. Messaoudi and Mustafa in \cite{Messaoud01} extended the results to nonlinear feedback laws. Soufyane and Wehbe in \cite{WhebeSoufyane-2003}    showed that  Timoshenko system with one internal distributed dissipation law is exponentially stable if and only if the wave propagation speeds are equal (i.e., $\frac{k_1}{\rho_1}=\frac{\rho_2}{k_2}$), otherwise,  only the strong  stability holds. Indeed, Rivera and Racke  in \cite{Racke01}  they improved the results of \cite{WhebeSoufyane-2003}, where an exponential decay of the solution of the system has been established, allowing the coefficient of the feedback to be with an indefinite sign. Wehbe and Youssef in \cite{Wehbe07} proved that the  Timoshenko system with one locally distributed viscous feedback is exponentially stable if and only if the wave propagation speeds are equal (i.e., $\frac{k_1}{\rho_1}=\frac{\rho_2}{k_2}$),  otherwise, only the polynomial stability holds.  Tebou in \cite{Tebou-2015} showed that the   Timoshenko beam with same feedback control in both equations  is exponentially stable. The stability of the Timoshenko system with thermoelastic dissipation  has been studied in \cite{HugoRacke01}, \cite{AlmeidaJnior2013}, \cite{Fatori01}, and \cite{Hao-2018}. The stability of Timoshenko system with memory type has been studied in  \cite{Racke03}, \cite{HugoRacke01}, \cite{MessaoudGuesmia-2009}, \cite{Messaoud02},  and \cite{Wehbe08}.  For the  boundary stabilization of the Timoshenko beam. Kim and Renardy in \cite{Kim01} showed that the Timoshenko beam under two boundary controls is exponentially stable. Ammar-Khodja and al. in \cite{Soufyane01}  studied the decay rate of the energy of the nonuniform Timoshenko beam with two boundary controls acting in the rotation-angle equation. In fact, under the equal speed wave propagation condition, they established exponential decay results up to an unknown finite dimensional space of initial data. In addition, they showed that the equal speed wave propagation condition is necessary for the exponential stability. However, in the case of non-equal speed, no decay rate has been discussed.  This result has been recently improved by Wehbe and al. in \cite{Wehbe06}; i.e., the authors in \cite{Wehbe06},   proved nonuniform stability and an optimal polynomial energy decay rate of the Timoshenko system with only one dissipation law on the boundary. For the stabilization of the Timoshenko beam with nonlinear term,  we mention  \cite{Racke02}, \cite{Alabau03}, \cite{ZuazuaAraruna}, \cite{Messaoud01}, \cite{Cavalcanti-2013},  and \cite{Hao-2018}.\\[0.1in]
%%%%%%%%%%%%%%%%%%%%%%%%%%%%%%%%%%%%%%%%%%%%%%%%%%
 \noindent Kelvin-Voigt material is a viscoelastic structure having properties of both elasticity and viscosity. There are a number of publications concerning the stabilization of wave equation with global or local  Kelvin-Voigt damping. For the global case, the authors  in \cite{Huang-1988, ChenLiuLiu-1998}, proved the analyticity and the exponential stability of the semigroup. When the Kelvin-Voigt damping is localized on an interval of the string, the regularity and stability of the solution depend on the properties of the damping coefficient. Notably, the system is more effectively controlled by the local Kelvin-Voigt damping when the coefficient changes more smoothly near the interface (see \cite{LiuLiu-1998,Renardy-2004,Zhang-2010, LiuZhang-2016, Liu--2017}).\\[0.1in]
%%%%%%%%%%%%%%%%%%%%%%%%%%%%%%%%%%%%%%%%%%%%%%%%%%
 Last but not least, in addition to the previously cited papers, the stability of the Timoshenko system with Kelvin-Voigt damping has been studied in few papers. Zhao and al. in  \cite{Zhao--2004} they considered the Timoshenko system with local distributed Kelvin–Voigt damping:
%%%%%%%%%%%%%%%%%%%%Equation%%%%%%%%%%%%%%%%%%%%%%
 \begin{equation} \label{E--(1.5)}
\begin{array}{lll}

\displaystyle{\rho_1 u_{tt}-\left[k_1 \left(u_x+y\right)_x+D_1 (u_{xt}-y_t)\right]_x=0,}
&\displaystyle{(x,t)\in\left(0,L\right)\times \mathbb{R}_{+},}

         \nline
         
\displaystyle{ \rho_2y_{tt}-\left(k_2y_{x}+D_2y_{xt}\right)_x+k_1 \left(u_x+y\right)_x+D_1 (u_{xt}-y_t)=0,}  &\displaystyle{(x,t)\in\left(0,L\right)\times \mathbb{R}_{+}.}
\end{array}
\end{equation}
They proved that the energy  of the System \eqref{E--(1.5)} subject to Dirichlet-Neumann boundary conditions has an exponential decay rate when coefficient functions $D_1,\ D_2 \in  C^{1,1}([0, L])$ and satisfy $D_1 \leq c D_2 (c > 0).$ Tian and Zhang in \cite{Tian-Zhang-2017} considered the Timoshenko System \eqref{E--(1.5)} under fully Dirichlet boundary conditions   with locally or globally distributed Kelvin-Voigt damping when coefficient functions $D_1,\ D_2 \in  C([0, L])$. First, when the Kelvin-Voigt damping is globally distributed, they showed that the Timoshenko System \eqref{E--(1.5)} under fully Dirichlet boundary conditions  is analytic. Next,  for their  system with local Kelvin-Voigt damping, they analyzed the exponential and polynomial stability according to the properties of coefficient functions $D_1,\ D_2.$ Unlike \cite{Zhao--2004} and  \cite{Tian-Zhang-2017}, in this paper, we consider the Timoshenko system with only one locally or globally distributed Kelvin-Voigt damping $D$ (see System \eqref{E--(1.1)}). Indeed, in this paper,  under hypothesis {\rm (H)}, we show that the energy of the Timoshenko System \eqref{E--(1.1)} subject to initial state \eqref{E--(1.2)} to
either the boundary conditions \eqref{E--(1.3)} or \eqref{E--(1.4)} has a polynomial decay rate of type     $t^{-1}$ and that this decay rate is in some sense optimal.\\[0.1in]
%%%%%%%%%%%%%%%%%%%%%%%%%%%%%%%%%%%%%%%%%%%%%%%%%%
This paper is organized as follows: In Section \ref{Section--2}, first, we show that the Timoshenko System \eqref{E--(1.1)} subject to initial state \eqref{E--(1.2)} to
either the boundary conditions \eqref{E--(1.3)} or \eqref{E--(1.4)} can  reformulate into an evolution equation and we deduce the well-posedness property of the problem by the semigroup approach. Second, using a criteria  of Arendt-Batty \cite{Arendt01}, we show that our system  is strongly stable. In Section \ref{Section--3}, we show that the Timoshenko System \eqref{E--(1.1)}-\eqref{E--(1.2)} with the boundary conditions  \eqref{E--(1.4)}  is not uniformly exponentially stable. In Section \ref{Section--4}, we prove the polynomial energy decay rate of type $t^{-1}$ for the System \eqref{E--(1.1)}-\eqref{E--(1.2)} to either the boundary conditions \eqref{E--(1.3)} or \eqref{E--(1.4)}. Moreover, we prove that this decay rate is in some sense optimal.
%%%%%%%%%%%%%%%%%%%%%%%%%%%%%%%%%%%%%%%%%%%%%%%%%%
%%%%%%%%%%%%%%%%%%%%%%%%%%%%%%%%%%%%%%%%%%%%%%%%%%
%%%%%%%%%%%%%%%%%%%%%%%%%%%%%%%%%%%%%%%%%%%%%%%%%%
 % Section 2: Well-Posedness and Strong Stability %
%%%%%%%%%%%%%%%%%%%%%%%%%%%%%%%%%%%%%%%%%%%%%%%%%%
%%%%%%%%%%%%%%%%%%%%%%%%%%%%%%%%%%%%%%%%%%%%%%%%%%	
%%%%%%%%%%%%%%%%%%%%%%%%%%%%%%%%%%%%%%%%%%%%%%%%%%
\section{Well-Posedness and Strong Stability}\label{Section--2}
%%%%%%%%%%%%%%%%%%%%%%%%%%%%%%%%%%%%%%%%%%%%%%%%%%
%%%%%%%%%%%%%%%%%%%%%%%%%%%%%%%%%%%%%%%%%%%%%%%%%%
 % Section 2.1: Well-Posedness  %
%%%%%%%%%%%%%%%%%%%%%%%%%%%%%%%%%%%%%%%%%%%%%%%%%%
%%%%%%%%%%%%%%%%%%%%%%%%%%%%%%%%%%%%%%%%%%%%%%%%%%
\subsection{Well-posedness of the problem.} \label{Section--2.1}
\noindent In this part, under condition {\rm (H)}, using a semigroup approach, we establish well-posedness result for the Timoshenko System \eqref{E--(1.1)}-\eqref{E--(1.2)} to either the boundary conditions \eqref{E--(1.3)} or \eqref{E--(1.4)}. The energy of solutions of the System \eqref{E--(1.1)} subject to initial state \eqref{E--(1.2)} to either the boundary conditions \eqref{E--(1.3)} or \eqref{E--(1.4)} is defined by:
%%%%%%%%%%%%%%%%%%%%Equation%%%%%%%%%%%%%%%%%%%%%%
\begin{equation*}
E\left(t\right)= \frac{1}{2}\int_0^L\left(\rho_1\left|u_t\right|^2+\rho_2\left|y_t\right|^2+k_1\left|u_x+y\right|^2+k_2\left|y_x\right|^2\right)dx.
\end{equation*}
 Let $\left(u,y\right)$ be a regular solution for the System  \eqref{E--(1.1)}. Multiplying the first and the second  equation of  \eqref{E--(1.1)}   by $u_t$ and $y_t,$  respectively, then  using  the boundary conditions \eqref{E--(1.3)} or \eqref{E--(1.4)}, we get
%%%%%%%%%%%%%%%%%%%%Equation%%%%%%%%%%%%%%%%%%%%%%
\begin{equation*}
E'\left(t\right)=-\int_{0}^LD(x)\left|y_{xt}\right|^2 dx\leq 0.
\end{equation*}
Thus    System \eqref{E--(1.1)} subject to initial state \eqref{E--(1.2)} to
either the boundary conditions \eqref{E--(1.3)} or \eqref{E--(1.4)}  is dissipative in the sense that its energy is non increasing with respect to the time $t$. Let us define the energy spaces  $\mathcal{H}_1$ and $\mathcal{H}_2$ by:
%%%%%%%%%%%%%%%%%%%%Equation%%%%%%%%%%%%%%%%%%%%%%
\begin{equation*}
	\mathcal{H}_1=H_0^1\left(0,L\right)\times L^2\left(0,L\right)\times H_0^1\left(0,L\right)\times L^2\left(0,L\right)
\end{equation*}
and
%%%%%%%%%%%%%%%%%%%%Equation%%%%%%%%%%%%%%%%%%%%%%
\begin{equation*}
	\mathcal{H}_2=H_0^1\left(0,L\right)\times L^2\left(0,L\right)\times H_*^1\left(0,L\right)\times L^2\left(0,L\right),
\end{equation*}
such that
%%%%%%%%%%%%%%%%%%%%Equation%%%%%%%%%%%%%%%%%%%%%%
\begin{equation*}
H_*^1(0,L)=\left\{f\in H^1(0,L)\ |\ \int_0^Lfdx=0\right\}.
\end{equation*}
It is easy to check that the space  $H_*^1$ is  Hilbert spaces over $\mathbb{C}$ equipped  with the norm
$$\left\|u\right\|^2_{H_*^1\left(0,L\right)}=\left\|u_x\right\|^2,$$
where  $\|\cdot\|$  denotes  the usual norm of $L^2\left(0,L\right)$. Both energy spaces $\mathcal{H}_1$ and $\mathcal{H}_2$ are equipped with the inner product defined by:
%%%%%%%%%%%%%%%%%%%%Equation%%%%%%%%%%%%%%%%%%%%%%
\begin{equation*}
\left<U,U_1\right>_{\mathcal{H}_j}=\rho_1 \int_0^Lv\overline{v}_1dx+\rho_2 \int_0^Lz\overline{z}_1dx+k_1\int_0^L\left(u_x+y\right)\overline{\left((u_1)_x+y_1\right)}dx+k_2\int_0^Ly_x\overline{(y_1)_x}dx\end{equation*}
for  all  $U=\left(u,v,y,z\right)$ and $U_1=\left(u_1,v_1,y_1,z_1\right)$ in $\mathcal{H}_j$,\, $j=1,2$. We use  $\|U\|_{\mathcal{H}_j}$ to denote the corresponding norms.
We now define the  following unbounded linear operators $\mathcal{A}_j$ in $\mathcal{H}_j$ by
%%%%%%%%%%%%%%%%%%%%Equation%%%%%%%%%%%%%%%%%%%%%%
\begin{equation*}
D\left(\mathcal{A}_1\right)=\left\{U=(u,v,y,z)\in  \mathcal{H}_1\ |\ v,\ z \in H^1_0(0,L),\ u\in H^2\left(0,L\right),\ \left(k_2y_{x}+Dz_{x}\right)_x\in  L^2\left(0,L\right)\right\},
\end{equation*}
%%%%%%%%%%%%%%%%%%%%Equation%%%%%%%%%%%%%%%%%%%%%%
\begin{equation*}
\begin{array}{ll}
\displaystyle{D\left(\mathcal{A}_2\right)=\bigg\{U=(u,v,y,z)\in  \mathcal{H}_2\ |\ v \in H^1_0(0,L),\ z\in H^1_*(0,L),\  u\in H^2\left(0,L\right),}\nline \hspace{5.2cm} \displaystyle{\left(k_2y_{x}+Dz_{x}\right)_x\in  L^2\left(0,L\right), \ y_x(0)=y_x(L)=0\bigg\}}
\end{array}
\end{equation*}
and for $j=1,2,$
%%%%%%%%%%%%%%%%%%%%Equation%%%%%%%%%%%%%%%%%%%%%%
\begin{equation*}
	\mathcal{A}_jU=\left(v,\frac{k_1}{\rho_1}(u_x+y)_x,z,\frac{1}{\rho_2}\left(k_2y_{x}+Dz_x\right)_x-\frac{k_1}{\rho_2}(u_x+y)\right),\ \ \forall\ U=\left(u,v,y,z\right)\in D\left(\mathcal{A}_j\right).
\end{equation*}
If $U=\left(u,u_t,y,y_t\right)$ is the state of System \eqref{E--(1.1)}-\eqref{E--(1.2)} to
either the boundary conditions \eqref{E--(1.3)} or \eqref{E--(1.4)}, then the Timoshenko system is transformed into a first order evolution equation on the Hilbert space $\mathcal{H}_j$:
%%%%%%%%%%%%%%%%%%%%Equation%%%%%%%%%%%%%%%%%%%%%%
\begin{equation}\label{E--(2.1)}
	\left\{
	\begin{array}{c}
	U_t(x,t)=\mathcal{A}_jU(x,t),\\  \noalign{\medskip}
	U\left(x,0\right)=U_0(x),
	\end{array}
	\right.
\end{equation}
where
%%%%%%%%%%%%%%%%%%%%Equation%%%%%%%%%%%%%%%%%%%%%%
\begin{equation*}
U_0\left(x\right)=\left(u_0(x),u_1(x),y_0(x),y_1(x)\right).
	\end{equation*}
%%%%%%%%%%%%%%%%%%%%%%%%%%%%%%%%%%%%%%%%%%%%%%%%%%
                % Proposition %
%%%%%%%%%%%%%%%%%%%%%%%%%%%%%%%%%%%%%%%%%%%%%%%%%%
\begin{prop}\label{Proposition--2.1}
\rm{	Under hypothesis {\rm (H)}, for $j=1,2,$ the unbounded linear operator $\mathcal{A}_j$ is m-dissipative in the energy space $\mathcal{H}_j$.}
\end{prop}
%%%%%%%%%%%%%%%%%%%%%%%%%%%%%%%%%%%%%%%%%%%%%%%%%%
                % Proof of  Proposition  %
%%%%%%%%%%%%%%%%%%%%%%%%%%%%%%%%%%%%%%%%%%%%%%%%%%
		\begin{proof} Let $j=1,2$, for $U=(u,v,y,z)\in D\left(\mathcal{A}_j\right)$, one has
%%%%%%%%%%%%%%%%%%%%Equation%%%%%%%%%%%%%%%%%%%%%%
\begin{equation*}
		\Re\left<\mathcal{A}_jU,U\right>_{\mathcal{H}_j}=	-\int_{0}^LD(x)\left|z_{x}\right|^2 dx\leq 0,
\end{equation*}
which implies that $\mathcal{A}_j$ is dissipative under hypothesis {\rm (H)}. Here $\Re$ is used to denote the real part of a complex number. We next prove the maximality of $\mathcal{A}_j$.
For $F=(f_1,f_2,f_3,f_4)\in\mathcal{H}_j$, we prove the existence of $U=(u,v,y,z)\in D(\mathcal{A}_j)$, unique solution of the equation 
%%%%%%%%%%%%%%%%%%%%Equation%%%%%%%%%%%%%%%%%%%%%%
			\begin{equation*}
						-\mathcal{A}_jU=F.
			\end{equation*}
Equivalently, one must consider the  system given by
%%%%%%%%%%%%%%%%%%%%Equation%%%%%%%%%%%%%%%%%%%%%%
			\begin{eqnarray}
			-v&=&f_1,\label{E--(2.2)}\\
			-{k_1}(u_x+y)_x&=&{\rho_1}f_2,\label{E--(2.3)}\\
			-z&=&f_3,\label{E--(2.4)}\\
			-\left({k_2}y_{x}+Dz_x\right)_x+{k_1}(u_x+y)&=&{\rho_2}f_4,\label{E--(2.5)}
			\end{eqnarray}
with the boundary conditions 	
%%%%%%%%%%%%%%%%%%%%Equation%%%%%%%%%%%%%%%%%%%%%%
\begin{equation}\label{E--(2.6)}
u(0)=u(L)=v(0)=v(L)=0\ \ \ \text{and}\ \ \ \left\{\begin{array}{ll}
 y(0)=y(L)=z(0)=z(L)=0,\quad \text{for }j=1,\nline
  y_x(0)=y_x(L)=0,\qquad \qquad \qquad\,\,\,   \text{for }j=2.
 \end{array}\right.
\end{equation}
Let $(\varphi,\psi)\in \mathcal{V}_j(0,L)$, where  $\mathcal{V}_1(0,L)=H_0^1(0,L)\times H_0^1(0,L)$ and $\mathcal{V}_2(0,L)=H_0^1(0,L)\times H_*^1(0,L)$. Multiplying Equations \eqref{E--(2.3)} and \eqref{E--(2.5)}  by $\overline{\varphi}$ and $\overline{\psi}$ respectively, integrating in $(0,L)$, taking the sum, then  using Equation \eqref{E--(2.4)} and the boundary condition \eqref{E--(2.6)}, we get
%%%%%%%%%%%%%%%%%%%%Equation%%%%%%%%%%%%%%%%%%%%%%
\begin{equation}\label{E--(2.7)}
\int_0^L \left(k_1 \left(u_x+y \right)\overline{\left( \varphi_x+\psi \right)}+k_2y_x\overline{\psi_{x}}\right)dx=\int_0^L\left( \rho_1 f_1 \bar{\varphi}+\rho_2f_4\bar{ \psi}+D\left(f_3\right)_x\overline{\psi_x}\right) dx,\quad \forall\ (\varphi,\psi)\in  \mathcal{V}_j(0,L).
\end{equation}
The left hand side of  \eqref{E--(2.7)} is a bilinear continuous coercive form on $\mathcal{V}_j(0,L)\times \mathcal{V}_j(0,L)$,  and  the right  hand side of  \eqref{E--(2.7)} is a linear continuous form on  $\mathcal{V}_j(0,L)$. Then, using Lax-Milligram theorem (see in \cite{Pazy01}), we deduce that there exists $(u,y)\in \mathcal{V}_j(0,L)$ unique solution of the variational Problem \eqref{E--(2.7)}. Thus, using \eqref{E--(2.2)}, \eqref{E--(2.4)}, and classical regularity arguments, we conclude that $-\mathcal{A}_jU=F$ admits a unique solution $U\in D\left(\mathcal{A}_j\right)$ and consequently $0\in \rho(\mathcal{A}_j)$,  where $\rho\left(\mathcal{A}_j\right)$ denotes the resolvent set of $\mathcal{A}_j$. Then, $\mathcal{A}_j$ is closed and consequently $\rho\left(\mathcal{A}_j\right)$ is open set of $\mathbb{C}$ (see Theorem 6.7 in \cite{Kato01}). Hence,  we easily get $\lambda\in\rho\left(\mathcal{A}_j\right)$  for sufficiently small $\lambda>0 $. This, together with the dissipativeness of $\mathcal{A}_j$, imply that   $D\left(\mathcal{A}_j\right)$ is dense in $\mathcal{H}_j$   and that $\mathcal{A}_j$ is m-dissipative in $\mathcal{H}_j$ (see Theorems 4.5, 4.6 in  \cite{Pazy01}). Thus, the proof is complete.	 
\end{proof}$\\[0.1in]$
%%%%%%%%%%%%%%%%%%%%%%%%%%%%%%%%%%%%%%%%%%%%%%%%%%
\noindent	Thanks to  Lumer-Phillips theorem (see \cite{LiuZheng01, Pazy01}), we deduce that $\mathcal{A}_j$ generates a  $C_0$-semigroup of contraction $e^{t\mathcal{A}_j}$ in $\mathcal{H}_j$ and therefore  Problem \eqref{E--(2.1)} is well-posed. Then, we have the following result.
%%%%%%%%%%%%%%%%%%%%%%%%%%%%%%%%%%%%%%%%%%%%%%%%%%
                % Theorem %
%%%%%%%%%%%%%%%%%%%%%%%%%%%%%%%%%%%%%%%%%%%%%%%%%%
\begin{thm}\label{Theorem--2.2}
\rm{Under hypothesis {\rm (H)}, for $j=1,2,$ for any $U_0\in\mathcal{H}_j$, the Problem \eqref{E--(2.1)}  admits a unique weak solution $U(x,t)=e^{t\mathcal{A}_j}U_0(x)$, such that
%%%%%%%%%%%%%%%%%%%%Equation%%%%%%%%%%%%%%%%%%%%%%
\begin{equation*}
U\in C\left(\mathbb{R}_{+};\mathcal{H}_j\right).
\end{equation*}
Moreover, if $U_0\in D\left(\mathcal{A}_j\right),$ then
%%%%%%%%%%%%%%%%%%%%Equation%%%%%%%%%%%%%%%%%%%%%%
\begin{equation*}
U\in C\left(\mathbb{R}_{+};D\left(\mathcal{A}_j\right)\right) \cap C^1\left(\mathbb{R}_{+};\mathcal{H}_j\right).
\end{equation*}
\xqed{$\square$}}\end{thm}
%%%%%%%%%%%%%%%%%%%%%%%%%%%%%%%%%%%%%%%%%%%%%%%%%%
\noindent Before starting the main results of this work, we introduce here the notions of stability that we encounter in this work.
%%%%%%%%%%%%%%%%%%%%%%%%%%%%%%%%%%%%%%%%%%%%%%%%%%
                % Definition %
%%%%%%%%%%%%%%%%%%%%%%%%%%%%%%%%%%%%%%%%%%%%%%%%%%
\begin{defi}\label{Definition--2.3}
\rm{Let $A:D(A)\subset H\to H $  generate a C$_0-$semigroup of contractions $\left(e^{t A}\right)_{t\geq0}$  on $H$. The  $C_0$-semigroup $\left(e^{t A}\right)_{t\geq0}$  is said to be
%%%%%%%%%%%%%%%%%%Enumerate%%%%%%%%%%%%%%%%%%%%%%%
\begin{enumerate}
%%%%%%%%%%%%%%%%%%%%%%item%%%%%%%%%%%%%%%%%%%%%%%%
\item[1.]  strongly stable if 
%%%%%%%%%%%%%%%%%%%%Equation%%%%%%%%%%%%%%%%%%%%%%
\begin{equation*}
\lim_{t\to +\infty} \|e^{t A}x_0\|_{H}=0, \quad\forall \ x_0\in H;
\end{equation*}
%%%%%%%%%%%%%%%%%%%%%%item%%%%%%%%%%%%%%%%%%%%%%%%
\item[2.]  exponentially (or uniformly) stable if there exist two positive constants $M$ and $\epsilon$ such that
%%%%%%%%%%%%%%%%%%%%Equation%%%%%%%%%%%%%%%%%%%%%%
\begin{equation*}
\|e^{t A}x_0\|_{H} \leq Me^{-\epsilon t}\|x_0\|_{H}, \quad
\forall\  t>0,  \ \forall \ x_0\in {H};
\end{equation*}
%%%%%%%%%%%%%%%%%%%%%%item%%%%%%%%%%%%%%%%%%%%%%%%
\item[3.] polynomially stable if there exists two positive constants $C$ and $\alpha$ such that
%%%%%%%%%%%%%%%%%%%%Equation%%%%%%%%%%%%%%%%%%%%%%
\begin{equation*}
 \|e^{t A}x_0\|_{H}\leq C t^{-\alpha}\|A x_0\|_{H},  \quad\forall\ 
t>0,  \ \forall \ x_0\in D\left(A\right).
\end{equation*}
In that case, one says that solutions of \eqref{E--(2.1)} decay at a rate $t^{-\alpha}$.
\noindent The  $C_0$-semigroup $\left(e^{t A}\right)_{t\geq0}$  is said to be  polynomially stable with optimal decay rate $t^{-\alpha}$ (with $\alpha>0$) if it is polynomially stable with decay rate $t^{-\alpha}$ and, for any $\varepsilon>0$ small enough, there exists solutions of \eqref{E--(2.1)} which do not decay at a rate $t^{-(\alpha+\varepsilon)}$.
\end{enumerate}}
\xqed{$\square$}\end{defi}
%%%%%%%%%%%%%%%%%%%%%%%%%%%%%%%%%%%%%%%%%%%%%%%%%%
\noindent We now look  for necessary conditions to show the strong stability of the $C_0$-semigroup $\left(e^{t A}\right)_{t\geq0}$. We will rely on the following result obtained by Arendt and Batty in \cite{Arendt01}. 
%%%%%%%%%%%%%%%%%%%%%%%%%%%%%%%%%%%%%%%%%%%%%%%%%%
                % Theorem %
%%%%%%%%%%%%%%%%%%%%%%%%%%%%%%%%%%%%%%%%%%%%%%%%%%
\begin{thm}[Arendt and Batty in \cite{Arendt01}]\label{Theorem--2.4}
\rm{ Let $A:D(A)\subset H\to H $  generate a C$_0-$semigroup of contractions $\left(e^{tA}\right)_{t\geq0}$  on $H$. If
%%%%%%%%%%%%%%%%%%Enumerate%%%%%%%%%%%%%%%%%%%%%%%
\begin{enumerate}
%%%%%%%%%%%%%%%%%%%%%%item%%%%%%%%%%%%%%%%%%%%%%%%
 \item[1.]  $A$ has no pure imaginary eigenvalues,
%%%%%%%%%%%%%%%%%%%%%%item%%%%%%%%%%%%%%%%%%%%%%%%
  \item[2.]  $\sigma\left(A\right)\cap i\mathbb{R}$ is countable,
 \end{enumerate}
where $\sigma\left(A\right)$ denotes the spectrum of $A$, then the $C_0$-semigroup $\left(e^{tA}\right)_{t\geq0}$  is strongly stable.}\xqed{$\square$}\end{thm}
%%%%%%%%%%%%%%%%%%%%%%%%%%%%%%%%%%%%%%%%%%%%%%%%%%
\noindent Our subsequent findings on  polynomial  stability  will rely on the following result from \cite{Borichev01, RaoLiu01, Batty01}, which gives necessary and sufficient conditions for a semigroup to be polynomially stable. For this aim, we recall the following standard result  (see \cite{Borichev01, RaoLiu01, Batty01} for part (i) and \cite{Huang01,pruss01} for part (ii)).
%%%%%%%%%%%%%%%%%%%%%%%%%%%%%%%%%%%%%%%%%%%%%%%%%%
                % Theorem %
%%%%%%%%%%%%%%%%%%%%%%%%%%%%%%%%%%%%%%%%%%%%%%%%%%
\begin{thm}\label{Theorem--2.5}
\rm{Let $A:D(A)\subset H\to H $  generate a C$_0-$semigroup of contractions $\left(e^{t A}\right)_{t\geq0}$  on $H$.  Assume that $i\lambda \in \rho(A),\ \forall \ \lambda\in \mathbb{R}$. Then, 
the  $C_0$-semigroup $\left(e^{t A}\right)_{t\geq0}$  is
%%%%%%%%%%%%%%%%%%Enumerate%%%%%%%%%%%%%%%%%%%%%%%
\begin{enumerate}
%%%%%%%%%%%%%%%%%%%%%%item%%%%%%%%%%%%%%%%%%%%%%%%
\item[(i)] Polynomially stable of order $\frac{1}{\ell}\, (\ell>0)$ if and only if 
%%%%%%%%%%%%%%%%%%%%Equation%%%%%%%%%%%%%%%%%%%%%%
\begin{equation*}
\lim \sup_{\lambda\in \mathbb{R},\ |\lambda|\to \infty} |\lambda|^{-\ell}\left\|\left(i\lambda I-A\right)^{-1}\right\|_{\mathcal{L}\left(H\right)}<+\infty.
\end{equation*}
%%%%%%%%%%%%%%%%%%%%%%item%%%%%%%%%%%%%%%%%%%%%%%%
\item[(ii)]  Exponentially stable if and only if
%%%%%%%%%%%%%%%%%%%%Equation%%%%%%%%%%%%%%%%%%%%%%
\begin{equation*}
\lim \sup_{\lambda\in \mathbb{R},\ |\lambda|\to \infty}\left\|\left(i\lambda I-A\right)^{-1}\right\|_{\mathcal{L}\left(H\right)}<+\infty.
\end{equation*}
\end{enumerate}}\xqed{$\square$}\end{thm}
%%%%%%%%%%%%%%%%%%%%%%%%%%%%%%%%%%%%%%%%%%%%%%%%%%
%%%%%%%%%%%%%%%%%%%%%%%%%%%%%%%%%%%%%%%%%%%%%%%%%%
 % Section 2.2: Strong Stability %
%%%%%%%%%%%%%%%%%%%%%%%%%%%%%%%%%%%%%%%%%%%%%%%%%%
%%%%%%%%%%%%%%%%%%%%%%%%%%%%%%%%%%%%%%%%%%%%%%%%%%	
\subsection{Strong stability}\label{Section--2.2}
In this part, we use  general criteria of Arendt-Batty in \cite{Arendt01} (see Theorem \ref{Theorem--2.4})   to show the strong stability of the $C_0$-semigroup $e^{t\mathcal{A}_j}$ associated to the Timoshenko System \eqref{E--(2.1)}. Our main result is the following theorem.
%%%%%%%%%%%%%%%%%%%%%%%%%%%%%%%%%%%%%%%%%%%%%%%%%%
                % Theorem %
%%%%%%%%%%%%%%%%%%%%%%%%%%%%%%%%%%%%%%%%%%%%%%%%%%
\begin{thm}\label{Theorem--2.6}
\rm{Assume that {\rm (H)} is true. Then, for $j=1,2,$ the $C_0-$semigroup $e^{t\mathcal{A}_j}$ is strongly stable in $\mathcal{H}_j$; i.e., for all $U_0\in\mathcal{H}_j$, the solution of \eqref{E--(2.1)} satisfies
%%%%%%%%%%%%%%%%%%%%Equation%%%%%%%%%%%%%%%%%%%%%%
\begin{equation*}
\lim_{t\to+\infty}\left\|e^{t\mathcal{A}_j} U_0\right\|_{\mathcal{H}_j}=0.
\end{equation*}
}\end{thm}
%%%%%%%%%%%%%%%%%%%%%%%%%%%%%%%%%%%%%%%%%%%%%%%%%%
\noindent The argument for Theorem \ref{Theorem--2.6} relies on the subsequent lemmas. 
%%%%%%%%%%%%%%%%%%%%%%%%%%%%%%%%%%%%%%%%%%%%%%%%%%
                % Lemma %
%%%%%%%%%%%%%%%%%%%%%%%%%%%%%%%%%%%%%%%%%%%%%%%%%%
\begin{lem}\label{Lemma--2.7}
\rm{ Under hypothesis {\rm (H)}, for $j=1,2,$  one has
%%%%%%%%%%%%%%%%%%%%Equation%%%%%%%%%%%%%%%%%%%%%%
\begin{equation*}
\ker\left(i\lambda I-{\mathcal{A}_j}\right)=\{0\},\ \  \forall \lambda\in \mathbb{R}.
\end{equation*}
}\end{lem}
%%%%%%%%%%%%%%%%%%%%%%%%%%%%%%%%%%%%%%%%%%%%%%%%%%
                % Proof of  Lemma  %
%%%%%%%%%%%%%%%%%%%%%%%%%%%%%%%%%%%%%%%%%%%%%%%%%%
\begin{proof} For $j=1,2$, from  Proposition \ref{Proposition--2.1}, we deduce that $0\in\rho\left(\mathcal{A}_j\right)$.  We still need to show the result for 
$\lambda\in\mathbb{R^*}$. Suppose that  there exists a real number $\lambda\neq 0$  and
$ U=\left(u,v,y,z\right)\in D\left(\mathcal{A}_j\right)$ such that
%%%%%%%%%%%%%%%%%%%%Equation%%%%%%%%%%%%%%%%%%%%%%
\begin{equation*}
{\mathcal{A}_j} U=i\lambda U.
\end{equation*}
Equivalently, we have 
%%%%%%%%%%%%%%%%%%%%Equation%%%%%%%%%%%%%%%%%%%%%%
\begin{equation}\label{E--(2.8)}
\left\{
\begin{array}{ll}

\displaystyle{v} =\displaystyle{i\lambda u},\nline

\displaystyle{{k_1}(u_x+y)_x}=\displaystyle{i{\rho_1}\lambda v},\nline

\displaystyle{z}=\displaystyle{i\lambda y},\nline

\displaystyle{\left({k_2} y_{x}+Dz_x\right)_x-{k_1}(u_x+y)}=\displaystyle{i{\rho_2}\lambda z}.

\end{array}
\right.
\end{equation}
First, a straightforward computation gives 
%%%%%%%%%%%%%%%%%%%%Equation%%%%%%%%%%%%%%%%%%%%%%
\begin{equation*}
0=\Re\left<i\lambda U,U\right>_{{\mathcal{H}}_j}=\Re\left<{\mathcal{A}_j} U,U\right>_{{\mathcal{H}_j}}=-\int_{0}^LD(x)\left|z_{x}\right|^2 dx,
\end{equation*}
using hypothesis {\rm (H)},  we deduce that
%%%%%%%%%%%%%%%%%%%%Equation%%%%%%%%%%%%%%%%%%%%%%
\begin{equation}\label{E--(2.9)}
Dz_x=0 \quad \text{ over }\ (0,L) \ \ \ \text{and}\ \ \ z_x=0 \quad  \text{ over }\ (\alpha,\beta).
\end{equation}
Inserting \eqref{E--(2.9)} in \eqref{E--(2.8)}, we get
%%%%%%%%%%%%%%%%%%%%Equation%%%%%%%%%%%%%%%%%%%%%%
\begin{eqnarray}
u=y_x=0,\quad \text{over }\  (\alpha,\beta),\label{E--(2.10)}\nline
{k_1} u_{xx}+{\rho_1}\lambda^2 u+{k_1} y_x=0,\quad \text{over }\  (0,L),\label{E--(2.11)}\nline
-k_1 u_x +{k_2} y_{xx}+\left({\rho_2}\lambda^2-{k_1}\right) y=0,\ \text{over }\  (0,L),\label{E--(2.12)}
\end{eqnarray}
with the following boundary conditions
%%%%%%%%%%%%%%%%%%%%Equation%%%%%%%%%%%%%%%%%%%%%%
\begin{equation}\label{E--(2.13)}
u(0)=u(L)=y(0)=y(L)=0, \text{ if  }\ j=1\ \ \ \text{or}\ \ \ u(0)=u(L)=y_x(0)=y_x(L)=0, \text{ if  }j=2.
\end{equation}
In fact, System \eqref{E--(2.11)}-\eqref{E--(2.13)} admits a unique solution  $(u,y)\in C^2((0,L))$. From \eqref{E--(2.10)} and by the uniqueness of solutions, we get 
%%%%%%%%%%%%%%%%%%%%Equation%%%%%%%%%%%%%%%%%%%%%%
\begin{equation}\label{E--(2.14)}
u=y_x=0,\quad \text{over }\  (0,L).
\end{equation}
%%%%%%%%%%%%%%%%%%Enumerate%%%%%%%%%%%%%%%%%%%%%%%
\begin{enumerate}
%%%%%%%%%%%%%%%%%%%%%%item%%%%%%%%%%%%%%%%%%%%%%%%
\item[1.] If $j=1$, from \eqref{E--(2.14)} and the fact that $y(0)=0,$ we get
%%%%%%%%%%%%%%%%%%%%Equation%%%%%%%%%%%%%%%%%%%%%%
\begin{equation*}
u=y=0,\quad \text{over }\  (0,L),
\end{equation*}
hence,  $U=0$. In this case the proof is complete.
%%%%%%%%%%%%%%%%%%%%%%item%%%%%%%%%%%%%%%%%%%%%%%%
\item[2.] If $j=2$, from \eqref{E--(2.14)} and the fact that $y\in H^1_*(0,L)\, (i.e.,\ \int_0^L y dx=0),$ we get
%%%%%%%%%%%%%%%%%%%%Equation%%%%%%%%%%%%%%%%%%%%%%
\begin{equation*}
u=y=0,\quad \text{over }\  (0,L),
\end{equation*}
therefore,  $U=0$,  also in this case the proof is complete.
\end{enumerate}
 \end{proof}
%%%%%%%%%%%%%%%%%%%%%%%%%%%%%%%%%%%%%%%%%%%%%%%%%%
                % Lemma %
%%%%%%%%%%%%%%%%%%%%%%%%%%%%%%%%%%%%%%%%%%%%%%%%%%
\begin{lem}\label{Lemma--2.8}
\rm{Under hypothesis {\rm (H)}, for $j=1,2,$  for all $\lambda\in\mathbb{R}$, then $i\lambda I-\mathcal{A}_j$ is surjective.}
\end{lem}
%%%%%%%%%%%%%%%%%%%%%%%%%%%%%%%%%%%%%%%%%%%%%%%%%%
                % Proof of  Lemma  %
%%%%%%%%%%%%%%%%%%%%%%%%%%%%%%%%%%%%%%%%%%%%%%%%%%
\begin{proof}
Let $F=(f_1,f_2,f_3,f_4) \in \mathcal{H}_j$, we look for $U=(u,v,y,z) \in D(\mathcal{A}_j)$ solution of
%%%%%%%%%%%%%%%%%%%%Equation%%%%%%%%%%%%%%%%%%%%%%
			\begin{equation*}
			(i\lambda U-\mathcal{A}_j)U=F.
			\end{equation*}
			Equivalently, we have 
%%%%%%%%%%%%%%%%%%%%Equation%%%%%%%%%%%%%%%%%%%%%%	
			\begin{eqnarray}
			v=i\lambda u-f_1,\label{E--(2.15)}\\ \noalign{\medskip}
			z=i\lambda y-f_3,\label{E--(2.16)}\\ \noalign{\medskip}
			\lambda^2 u+\frac{k_1}{\rho_1}(u_x+y)_x=F_1,\label{E--(2.17)}\\ \noalign{\medskip}
			\lambda^2 y+{\rho_2}^{-1}\left[\left({k_2}+i\lambda D\right)y_{x}\right]_x-\frac{k_1}{\rho_2}(u_x+y)=F_2,\label{E--(2.18)}
						\end{eqnarray}		
	with the boundary conditions 
%%%%%%%%%%%%%%%%%%%%Equation%%%%%%%%%%%%%%%%%%%%%%	
\begin{equation}\label{E--(2.19)}
u(0)=u(L)=v(0)=v(L)=0\ \ \ \text{and}\ \ \ \left\{\begin{array}{ll}
 y(0)=y(L)=z(0)=z(L)=0,\quad \text{for }j=1,\nline
  y_x(0)=y_x(L)=0,\qquad \qquad \qquad\,\,\,  \text{for }j=2.
 \end{array}\right.
\end{equation}					
Such that	
%%%%%%%%%%%%%%%%%%%%Equation%%%%%%%%%%%%%%%%%%%%%%
	\begin{equation*}
	\left\{
	\begin{array}{lll}
	\displaystyle{F_1=-f_2-i\lambda f_1\in L^2(0,L)}	,\\  \noalign{\medskip}
	\displaystyle{F_2=-f_4-i\lambda f_3+{\rho_2}^{-1}\left(D\left(f_3\right)_x\right)_{x}\in H^{-1}(0,L)}.
	\end{array}
	\right.
\end{equation*}	
We define the operator    $\mathcal{L}_j $  by
%%%%%%%%%%%%%%%%%%%%Equation%%%%%%%%%%%%%%%%%%%%%%
\begin{equation*}
\mathcal{L}_j\mathcal{U}=\left(-\frac{k_1}{\rho_1}(u_x+y)_{x}, -
    \rho_2^{-1}\left[\left({k_2}+i\lambda D\right)y_{x}\right]_x+\frac{k_1}{\rho_2}(u_x+y) \right),\qquad \forall\    \mathcal{U}=(u,y)\in \mathcal{V}_j(0,L),
\end{equation*}
where
%%%%%%%%%%%%%%%%%%%%Equation%%%%%%%%%%%%%%%%%%%%%%
\begin{equation*}
\mathcal{V}_1(0,L)=H_0^1(0,L)\times H_0^1(0,L)\ \ \ \text{and}\ \ \ \mathcal{V}_2(0,L)=H_0^1(0,L)\times H_*^1(0,L).
\end{equation*}       
Using Lax-Milgram theorem, it is easy to show that $\mathcal{L}_j$ is an isomorphism from $\mathcal{V}_j(0,L)$ onto $(H^{-1}\left(0,L\right))^2$. Let $\mathcal{U}=\left(u,y\right)$ and  $F=\left(-F_1,-F_2\right)$, then we transform System \eqref{E--(2.17)}-\eqref{E--(2.18)} into the following form
%%%%%%%%%%%%%%%%%%%%Equation%%%%%%%%%%%%%%%%%%%%%%
\begin{equation}\label{E--(2.20)}
\mathcal{U}-\lambda^2\mathcal{L}^{-1}_j\mathcal{U}=\mathcal{L}^{-1}F.
\end{equation}
Using the compactness embeddings from $L^2(0,L)$ into $H^{-1}(0,L)$ and from $H^1_0(0,L)$ into $L^{2}(0,L)$, and from $H^1_L(0,L)$ into $L^{2}(0,L)$, we deduce that the operator $\mathcal{L}_j^{-1}$ is compact from $L^2(0,L)\times L^2(0,L)$ into $L^2(0,L)\times L^2(0,L)$. Consequently, by Fredholm alternative, proving the existence of $\mathcal{U}$ solution of \eqref{E--(2.20)} reduces to proving  
$\ker\left(I-\lambda^2\mathcal{L}^{-1}_j\right)=0$. Indeed, if $\left(\varphi,\psi\right)\in\ker(I-\lambda^2\mathcal{L}_j^{-1})$, then we have $\lambda^2 \left(\varphi,\psi\right)-  \mathcal{L}_j \left(\varphi,\psi\right)=0$. It follows that
%%%%%%%%%%%%%%%%%%%%Equation%%%%%%%%%%%%%%%%%%%%%% 
\begin{equation}\label{E--(2.21)}
		\left\{
		\begin{array}{ll}
		\displaystyle{\lambda^2 \varphi+\frac{k_1}{\rho_1}(\varphi_x+\psi)_x
			=0,}
		\\  \noalign{\medskip}
		\displaystyle{\lambda^2 \psi+{\rho_2}^{-1}\left[\left({k_2}+i\lambda D\right)\psi_{x}\right]_x-\frac{k_1}{\rho_2}(\varphi_x+\psi)=0,}
		
		\end{array}
		\right.
\end{equation}
with the following boundary conditions
%%%%%%%%%%%%%%%%%%%%Equation%%%%%%%%%%%%%%%%%%%%%%
\begin{equation}\label{E--(2.22)}
\varphi(0)=\varphi(L)=\psi(0)=\psi(L)=0, \text{ if  }j=1\ \ \ \text{or}\ \ \ \varphi(0)=\varphi(L)=\psi_x(0)=\psi_x(L)=0, \text{ if  }j=2.
\end{equation}
It is now easy to see that if $(\varphi,\psi)$ is a solution of System \eqref{E--(2.21)}-\eqref{E--(2.22)}, then the vector $V$ defined by
%%%%%%%%%%%%%%%%%%%%Equation%%%%%%%%%%%%%%%%%%%%%%
\begin{equation*}
V= \left(\varphi,i\lambda \varphi,\psi,i\lambda \psi\right)
\end{equation*}
belongs to $D(\mathcal{A}_j)$ and  $i\lambda V-\mathcal{A}_jV=0.$ Therefore,  ${V}\in\ker\left(i\lambda I-{\mathcal{A}_j}\right) $. Using   Lemma \ref{Lemma--2.7}, we get $V=0$, and so 
%%%%%%%%%%%%%%%%%%%%Equation%%%%%%%%%%%%%%%%%%%%%%
\begin{equation*}
\ker(I-\lambda^2\mathcal{L}^{-1}_j)=\{0\}.
\end{equation*}
  Thanks to Fredholm alternative, the Equation  \eqref{E--(2.20)} admits a unique solution $(u,v) \in \mathcal{V}_j(0,L)$. Thus, using \eqref{E--(2.15)}, \eqref{E--(2.17)} and a classical regularity arguments, we conclude that $\left(\ i\lambda -\mathcal{A}_j\right)U=F$ admits a unique solution $U\in D\left(\mathcal{A}_j\right)$. Thus, the proof is complete.
\end{proof}$\\[0.1in]$
%%%%%%%%%%%%%%%%%%%%%%%%%%%%%%%%%%%%%%%%%%%%%%%%%%
We are now in a position to conclude the proof of Theorem \ref{Theorem--2.6}.\\[0.1in]
%%%%%%%%%%%%%%%%%%%%%%%%%%%%%%%%%%%%%%%%%%%%%%%%%%
                % Proof of  Theorem  %
%%%%%%%%%%%%%%%%%%%%%%%%%%%%%%%%%%%%%%%%%%%%%%%%%%
\noindent \textbf{Proof of Theorem \ref{Theorem--2.6}.} Using Lemma \ref{Lemma--2.7}, we directly deduce that ${\mathcal{A}}_j$ ha non pure imaginary eigenvalues. According to Lemmas \ref{Lemma--2.7}, \ref{Lemma--2.8} and with the help of the closed graph theorem of Banach, we deduce that $\sigma({\mathcal{A}}_j)\cap i\mathbb{R}=\{\emptyset\}$.  Thus, we get the conclusion by applying Theorem \ref{Theorem--2.4} of Arendt and Batty. \xqed{$\square$}
%%%%%%%%%%%%%%%%%%%%%%%%%%%%%%%%%%%%%%%%%%%%%%%%%%
%%%%%%%%%%%%%%%%%%%%%%%%%%%%%%%%%%%%%%%%%%%%%%%%%%
%%%%%%%%%%%%%%%%%%%%%%%%%%%%%%%%%%%%%%%%%%%%%%%%%%
 % Section 3: Lack of exponential  stability %
%%%%%%%%%%%%%%%%%%%%%%%%%%%%%%%%%%%%%%%%%%%%%%%%%%
%%%%%%%%%%%%%%%%%%%%%%%%%%%%%%%%%%%%%%%%%%%%%%%%%%	
%%%%%%%%%%%%%%%%%%%%%%%%%%%%%%%%%%%%%%%%%%%%%%%%%%
\section{Lack of exponential  stability of $\mathcal{A}_2$}\label{Section--3}
\noindent In this section, our goal is to show that the  Timoshenko System  \eqref{E--(1.1)}-\eqref{E--(1.2)} with Dirichlet-Neumann boundary conditions \eqref{E--(1.4)}  is not exponentially stable. 
%%%%%%%%%%%%%%%%%%%%%%%%%%%%%%%%%%%%%%%%%%%%%%%%%%
%%%%%%%%%%%%%%%%%%%%%%%%%%%%%%%%%%%%%%%%%%%%%%%%%%
 % Section 3.1: Lack of exponential  stability %
%%%%%%%%%%%%%%%%%%%%%%%%%%%%%%%%%%%%%%%%%%%%%%%%%%
%%%%%%%%%%%%%%%%%%%%%%%%%%%%%%%%%%%%%%%%%%%%%%%%%%	
\subsection{Lack of exponential  stability of $\mathcal{A}_2$ with  global Kelvin–Voigt damping}\label{Section--3.1}
In this part, assume that 
%%%%%%%%%%%%%%%%%%%%Equation%%%%%%%%%%%%%%%%%%%%%% 
\begin{equation}\label{E--(3.1)}
D(x)=D_0>0,\ \forall x\in (0,L),
\end{equation}
where $D_0\in \mathbb{R}^+_{*}$. We prove the following theorem.
%%%%%%%%%%%%%%%%%%%%%%%%%%%%%%%%%%%%%%%%%%%%%%%%%%
                % Theorem %
%%%%%%%%%%%%%%%%%%%%%%%%%%%%%%%%%%%%%%%%%%%%%%%%%%	
\begin{thm}\label{Theorem--3.1}
\rm{Under hypothesis \eqref{E--(3.1)}, for $\epsilon>0\left(\text{small enough}\right)$,  we cannot expect the energy decay rate $t^{-\frac{2}{{2-\epsilon}}}$ for all initial data $U_0\in D\left(\mathcal{A}_2\right)$ and for all $t>0.$  } 
\end{thm}
%%%%%%%%%%%%%%%%%%%%%%%%%%%%%%%%%%%%%%%%%%%%%%%%%%
                % Proof of  Theorem  %
%%%%%%%%%%%%%%%%%%%%%%%%%%%%%%%%%%%%%%%%%%%%%%%%%%
\begin{proof} Following to  Borichev \cite{Borichev01} (see  Theorem \ref{Theorem--2.4} part (i)), it suffices to show the existence of sequences
$\left(\lambda_n\right)_n\subset\mathbb{R}$ with $\lambda_n\to+\infty$,
$\left(U_n\right)_n\subset D\left(\mathcal{A}_2\right)$, and $\left(F_n\right)_n\subset \mathcal{H}_2$
 such that $\left(i\lambda_n I-\mathcal{A}_2\right)U_n=F_n$ is bounded in $\mathcal{H}_2$ and $\lambda_n^{-2+\epsilon}\| U_n\|\to+\infty$. Set
%%%%%%%%%%%%%%%%%%%%Equation%%%%%%%%%%%%%%%%%%%%%% 
 \begin{equation*}
 F_n =\left(0,\sin\left(\frac{n\pi x}{L}\right),0,0\right),\ 
 U_n=\left(A_n\sin\left(\frac{n\pi x}{L}\right),i\lambda_n
 A_n\sin\left(\frac{n\pi x}{L}\right),B_n\cos\left(\frac{n\pi x}{L}\right),i\lambda_n B_n\cos\left(\frac{n\pi x}{L}\right)\right)
 \end{equation*}
and
%%%%%%%%%%%%%%%%%%%%Equation%%%%%%%%%%%%%%%%%%%%%%
\begin{equation}\label{E--(3.2)}
\lambda_n=\frac{n\pi}{L}\sqrt{\frac{k_1}{\rho_1}} ,\quad A_n=-\frac{in\pi D_0}{k_1\, L}\sqrt{\frac{\rho_1}{k_1}}+\frac{k_2}{k_1}\left(\frac{\rho_2}{k_2}-\frac{\rho_1}{k_1}\right)-\frac{\rho_1 L^2}{k_1 \pi^2 n^2},\quad B_n=\frac{\rho_1 L}{k_1 n\pi}.
\end{equation}
Clearly that  $U_n\in D\left(\mathcal{A}_2\right),$ and $F_n$ is bounded in $\mathcal{H}_2$. Let us show that
%%%%%%%%%%%%%%%%%%%%Equation%%%%%%%%%%%%%%%%%%%%%%
\begin{equation*}
\left(i\lambda_n I-\mathcal{A}_2\right)U_n=F_n.
\end{equation*}
Detailing $\left(i\lambda_n I-\mathcal{A}_2\right)U_n$, we get
%%%%%%%%%%%%%%%%%%%%Equation%%%%%%%%%%%%%%%%%%%%%%
\begin{equation*}
\left(i\lambda_n I-\mathcal{A}_2\right)U_n=\left(0,C_{1,n}\sin\left(\frac{n\pi x}{L}\right),0,C_{2,n} \cos\left(\frac{n\pi x}{L}\right)\right),
\end{equation*}
where
%%%%%%%%%%%%%%%%%%%%Equation%%%%%%%%%%%%%%%%%%%%%% 
\begin{equation}\label{E--(3.3)}
C_{1,n}=\left(\frac{k_1}{\rho_1} \left(\frac{n\pi}{L}\right)^2-\lambda^2_n \right) A_n+\frac{k_1 n\pi}{\rho_1 L}B_n,\ C_{2,n}=\frac{n\pi k_1}{\rho_2 L} A_n+\left(-\lambda^2_n+\frac{k_1}{\rho_2}+\frac{{k_2}+i \lambda_n D_0 }{\rho_2} \left(\frac{n\pi}{L}\right)^2\right)B_n.
\end{equation}
Inserting \eqref{E--(3.2)} in \eqref{E--(3.3)}, we get
%%%%%%%%%%%%%%%%%%%%Equation%%%%%%%%%%%%%%%%%%%%%% 
\begin{equation*}
C_{1,n}=1\ \ \ \text{and}\ \ \ C_{2,n}=0,
\end{equation*}
hence we obtain 
%%%%%%%%%%%%%%%%%%%%Equation%%%%%%%%%%%%%%%%%%%%%%
\begin{equation*}
\left(i\lambda_n I-\mathcal{A}_2\right)U_n=\left(0,\sin\left(\frac{n\pi x}{L}\right),0,0\right)=F_n.
\end{equation*}
Now, we have
%%%%%%%%%%%%%%%%%%%%Equation%%%%%%%%%%%%%%%%%%%%%%
\begin{equation*}
\left\|U_n\right\|^2_{\mathcal{H}_2}\geq \rho_1\int^L_0\left|i\lambda_n
 A_n\sin\left(\frac{n\pi x}{L}\right)\right|^2dx=\frac{\rho_1\,  L\, \lambda_n^2 }{2}\left|A_n\right|^2\sim \lambda^4_n.
\end{equation*} 
Therefore, for $\epsilon>0\left(\text{small enough}\right)$, we have 
%%%%%%%%%%%%%%%%%%%%Equation%%%%%%%%%%%%%%%%%%%%%%
\begin{equation*}
\lambda_n^{-2+\epsilon}\left\|U_n\right\|_{\mathcal{H}_2}\sim\lambda_n^\epsilon\to +\infty.
\end{equation*} 
Finally, following to  Borichev \cite{Borichev01} (see  Theorem \ref{Theorem--2.4} part (i)) we cannot expect the energy decay rate $t^{-\frac{2}{{2-\epsilon}}}$. 
\end{proof}$\\[0.1in]$
%%%%%%%%%%%%%%%%%%%%%%%%%%%%%%%%%%%%%%%%%%%%%%%%%%
\noindent Note that Theorem \ref{Theorem--3.1} also implies that our system is non-uniformly stable.
%%%%%%%%%%%%%%%%%%%%%%%%%%%%%%%%%%%%%%%%%%%%%%%%%%
%%%%%%%%%%%%%%%%%%%%%%%%%%%%%%%%%%%%%%%%%%%%%%%%%%
 % Section 3.2: Lack of exponential  stability %
%%%%%%%%%%%%%%%%%%%%%%%%%%%%%%%%%%%%%%%%%%%%%%%%%%
%%%%%%%%%%%%%%%%%%%%%%%%%%%%%%%%%%%%%%%%%%%%%%%%%%	
\subsection{Lack of exponential  stability of $\mathcal{A}_2$ with  local Kelvin–Voigt damping}\label{Section--3.2}
In this part,  under the equal speed wave propagation condition (i.e., $\frac{\rho_1}{k_1}=\frac{\rho_2}{k_2}$), we use the classical method developed by Littman and Markus in \cite{Littman1988} (see also \cite{CurtainZwart01}), to show that the Timoshenko System \eqref{E--(1.1)}-\eqref{E--(1.2)} with  local Kelvin–Voigt damping, and with Dirichlet-Neumann boundary conditions \eqref{E--(1.4)}  is not exponentially stable. For this aim, assume that
%%%%%%%%%%%%%%%%%%%%Equation%%%%%%%%%%%%%%%%%%%%%%
\begin{equation}\label{E--(3.4)}
\frac{\rho_1}{k_1}=\frac{\rho_2}{k_2}\ \  \text{and}\ \ \ D(x)=\left\{\begin{array}{ll}
0,& 0<x\leq \alpha,\nline 
D_0& \alpha<x\leq L,
\end{array}\right.
\end{equation}
where $D_0\in \mathbb{R}^+_{*}$ and $\alpha \in (0,L)$. For simplicity and without loss of generality, in this part, we  take $\frac{\rho_1}{k_1}=1$, $D_0=k_2$, $L=1$, and $\alpha=\frac{1}{2}$,  then hypothesis \eqref{E--(3.4)} becomes
%%%%%%%%%%%%%%%%%%%%Equation%%%%%%%%%%%%%%%%%%%%%%
\begin{equation}\label{E--(3.5)}
\frac{\rho_1}{k_1}=\frac{\rho_2}{k_2}=1\ \  \text{and}\ \ \ D(x)=\left\{\begin{array}{ll}
0,& 0<x\leq \frac{1}{2},\nline 
k_2& \frac{1}{2}<x\leq 1.
\end{array}\right.
\end{equation} 
Our main result in this part is following theorem.
%%%%%%%%%%%%%%%%%%%%%%%%%%%%%%%%%%%%%%%%%%%%%%%%%%
                % Theorem %
%%%%%%%%%%%%%%%%%%%%%%%%%%%%%%%%%%%%%%%%%%%%%%%%%%
\begin{thm}\label{Theorem--3.2}
\rm{Under hypothesis \eqref{E--(3.5)}. The semigroup generated by the operator $\mathcal{A}_2$ is not exponentially stable in the energy space $\mathcal{H}_2.$}
\end{thm}
%%%%%%%%%%%%%%%%%%%%%%%%%%%%%%%%%%%%%%%%%%%%%%%%%%
\noindent For the proof of Theorem  \ref{Theorem--3.2}, we recall the following definitions: the growth bound 
$\omega_0\left(\mathcal{A}_2\right)$ and the 
the spectral bound $s\left(\mathcal{A}_2\right)$ of $\mathcal{A}_2$
are defined respectively as 
%%%%%%%%%%%%%%%%%%%%Equation%%%%%%%%%%%%%%%%%%%%%%
\begin{equation*}
\omega(\mathcal{A}_2)=\lim_{t\to \infty}\frac{\log \left\|e^{t\mathcal{A}_2}\right\|_{\mathcal{L}(\mathcal{H}_2)}}{t}\ \ \ \text{and}\ \ \ s\left(\mathcal{A}_2\right)=\sup\left\{\Re\left(\lambda\right):\ \lambda\in \sigma\left(\mathcal{A}_2\right) \right\}.
\end{equation*}
From the Hille-Yoside theorem (see also Theorem 2.1.6 and Lemma 2.1.11 in \cite{CurtainZwart01}), one has that
%%%%%%%%%%%%%%%%%%%%Equation%%%%%%%%%%%%%%%%%%%%%%
\begin{equation*}
s\left(\mathcal{A}_2\right)\leq \omega_0\left(\mathcal{A}_2\right).
\end{equation*} 
By the previous results, one clearly has that $s\left(\mathcal{A}_2\right)\leq 0$ and the theorem would follow if equality holds in the previous inequality. It therefore amounts to 
show the existence of a sequence of eigenvalues of $\mathcal{A}_2$ whose real parts tend to zero.\\[0.1in]
%%%%%%%%%%%%%%%%%%%%%%%%%%%%%%%%%%%%%%%%%%%%%%%%%%
Since $\mathcal{A}_2$ is dissipative, we fix $\alpha_0>0$ small enough and we study the asymptotic behavior of the eigenvalues $\lambda$ of 
$\mathcal{A}_2$ in the strip 
%%%%%%%%%%%%%%%%%%%%Equation%%%%%%%%%%%%%%%%%%%%%%
\begin{equation*}
S=\left\{\lambda \in \mathbb{C}:-\alpha_0\leq \Re(\lambda)\leq 0\right\}.
\end{equation*}
First, we determine the characteristic equation satisfied by the eigenvalues of $\mathcal{A}_2$. For this aim, let $\lambda\in\mathbb{C}^*$ be an eigenvalue of $\mathcal{A}_2$ and let $U=\left(u,\lambda u,y,\lambda y,\omega\right)\in D(\mathcal{A}_2)$  be an associated eigenvector. Then  the  eigenvalue problem is given by
%%%%%%%%%%%%%%%%%%%%Equation%%%%%%%%%%%%%%%%%%%%%%
\begin{eqnarray}
\lambda^2u-u_{xx}-y_x=0,& x\in (0,1),\label{E--(3.6)}\\ \noalign{\medskip}
c^2u_x+ \left(\lambda^2+c^2\right) y-\left(1+\frac{D}{k_2}\lambda\right) y_{xx}=0,& x\in (0,1),\label{E--(3.7)}
\end{eqnarray}
with the boundary conditions
%%%%%%%%%%%%%%%%%%%%Equation%%%%%%%%%%%%%%%%%%%%%%
\begin{equation}\label{E--(3.8)}
u(0)=y_x(0)=u(1)=y_x(1)=0,
\end{equation}
where $c=\sqrt{k_1 k_2^{-1}}$. We define 
\begin{equation*}
\left\{
\begin{array}{ll}
\displaystyle{u^-(x):=u(x)}, \ \displaystyle{y^-(x):=y(x)},\quad \displaystyle{x\in (0,1/2)},  \nline
\displaystyle{u^+(x):=u(x)},\  \displaystyle{y^+(x):=y(x)},\quad \displaystyle{x\in [1/2,1)},
\end{array}
\right.
\end{equation*}
then System \eqref{E--(3.6)}-\eqref{E--(3.8)} becomes
%%%%%%%%%%%%%%%%%%%%Equation%%%%%%%%%%%%%%%%%%%%%%
\begin{eqnarray}
\lambda^2u^--u^-_{xx}-y^-_x=0,& x\in (0,1/2),\label{E--(3.9)}\\ \noalign{\medskip}
c^2u^-_x+ \left(\lambda^2+c^2\right) y^--y^-_{xx}=0,& x\in (0,1/2),\label{E--(3.10)}\\ \noalign{\medskip}
\lambda^2u^+-u^+_{xx}-y^+_x=0,& x\in [1/2,1),\label{E--(3.11)}\\ \noalign{\medskip}
c^2u^+_x+ \left(\lambda^2+c^2\right) y^+-\left(1+\lambda\right) y^+_{xx}=0,& x\in [1/2,1),\label{E--(3.12)}
\end{eqnarray}
with the boundary conditions
%%%%%%%%%%%%%%%%%%%%Equation%%%%%%%%%%%%%%%%%%%%%%
\begin{eqnarray}
u^-(0)=y^-_x(0)=0,\label{E--(3.13)}\\ \noalign{\medskip}
u^+(1)=y^+_x(1)=0,\label{E--(3.14)}
\end{eqnarray} 
and the continuity conditions
%%%%%%%%%%%%%%%%%%%%Equation%%%%%%%%%%%%%%%%%%%%%%
\begin{eqnarray}
u^-(1/2)=u^+(1/2),&\label{E--(3.15)}\\ \noalign{\medskip}
u_x^-(1/2)= u^+_{x}(1/2),\label{E--(3.16)}\\ \noalign{\medskip}
y^-(1/2)=y^+(1/2),\label{E--(3.17)}\\ \noalign{\medskip}
y^-_x(1/2)=\left(1+\lambda\right)y^+_x(1/2).\label{E--(3.18)}
\end{eqnarray}
In order to proceed, we set the following notation. Here and below, in the case where $z$ is a non zero non-real number, we define (and denote) by  $\sqrt{z}$ the square root of $z$; i.e., the unique complex number with positive real part whose square is equal to $z$. Our aim is to study the asymptotic behavior of the large eigenvalues $\lambda $ of ${\mathcal{A}_2}$ in $S$.  By taking $\lambda$ large enough,  the general solution of System \eqref{E--(3.9)}-\eqref{E--(3.10)} with boundary condition \eqref{E--(3.13)} is given by
%%%%%%%%%%%%%%%%%%%%Equation%%%%%%%%%%%%%%%%%%%%%% 
			\begin{equation*}\left\{
			\begin{array}{ll}
				\displaystyle{		u^-(x)=\alpha_1 \sinh(r_1 x)+\alpha_2\sinh(r_2 x),}
					\\ \noalign{\medskip}
	\displaystyle{		y^-(x)=\alpha_1\frac{\lambda^2-r_1^2}{r_1} \cosh(r_1 x)+\alpha_2\frac{\lambda^2-r_2^2}{r_2}\cosh(r_2 x),}

			\end{array}\right.
			\end{equation*}		
and the general solution of Equation \eqref{E--(3.11)}-\eqref{E--(3.12)} with boundary condition \eqref{E--(3.14)} is given by
%%%%%%%%%%%%%%%%%%%%Equation%%%%%%%%%%%%%%%%%%%%%% 
\begin{equation*}\left\{
			\begin{array}{ll}
				\displaystyle{		u^+(x)=\beta_1 \sinh(s_1 (1-x))+\beta_2\sinh(s_2 (1-x)),}
					\\ \noalign{\medskip}
	\displaystyle{		y^+(x)=-\beta_1\frac{\lambda^2-s_1^2}{s_1} \cosh(s_1 (1-x))-\beta_2\frac{\lambda^2-s_2^2}{s_2}\cosh(s_2 (1-x)),}

			\end{array}\right.
			\end{equation*}						
where $\alpha_1,\ \alpha_2,\ \beta_1,\ \beta_2\in\mathbb{C}$, 
%%%%%%%%%%%%%%%%%%%%Equation%%%%%%%%%%%%%%%%%%%%%%
\begin{equation}\label{E--(3.19)}
r_1=\lambda\sqrt{1+\frac{i c}{\lambda}},\qquad r_2=\lambda\sqrt{1-\frac{i c}{\lambda}}
\end{equation}
and
%%%%%%%%%%%%%%%%%%%%Equation%%%%%%%%%%%%%%%%%%%%%% 
\begin{equation}\label{E--(3.20)}
s_1=\sqrt{\frac{\lambda+ \frac{\lambda^2}{2}\left(1 +\sqrt{1 -\frac{4c^2}{ \lambda^3}-\frac{4 c^2}{ \lambda^4}}\right)}{ 1+\frac{1}{\lambda} }},
\qquad s_2=\sqrt{\frac{\lambda+ \frac{\lambda^2}{2}\left(1 -\sqrt{1 -\frac{4c^2}{ \lambda^3}-\frac{4 c^2}{ \lambda^4}}\right)}{ 1+\frac{1}{\lambda} }}.
\end{equation}
The boundary conditions in  \eqref{E--(3.15)}-\eqref{E--(3.18)} can be expressed by 
\begin{equation*}
M\begin{pmatrix}
	\alpha_1\\ \alpha_2\\ \beta_1\\ \beta_2
\end{pmatrix}=0,
\end{equation*} 
where
%%%%%%%%%%%%%%%%%%%%Equation%%%%%%%%%%%%%%%%%%%%%%
\begin{equation*}
	M=\begin{pmatrix}
	\sinh\left(\frac{r_1}{2} \right)&\sinh\left(\frac{r_2}{2} \right)&-\sinh\left(\frac{s_1}{2} \right)&-\sinh\left(\frac{s_2}{2} \right)
	\\ \noalign{\medskip}
\frac{r_1}{i \, c\, \lambda^2}\cosh\left(\frac{r_1}{2} \right)&\frac{r_2}{i \, c\, \lambda^2}\cosh\left(\frac{r_2}{2} \right)&\frac{s_1}{i \, c\, \lambda^2}\cosh\left(\frac{s_1}{2} \right)&\frac{s_2}{i \, c\, \lambda^2}\cosh\left(\frac{s_2}{2} \right)	
	\\ \noalign{\medskip}
r_1^2\sinh\left(\frac{r_1}{2} \right)&r_2^2\sinh\left(\frac{r_2}{2} \right)&\left(\lambda^3-(\lambda+1) s_1^2\right)\sinh\left(\frac{s_1}{2} \right)&\left(\lambda^3-(\lambda+1) s_2^2\right)\sinh\left(\frac{s_2}{2} \right)		
	\\ \noalign{\medskip}
r_1^{-1}\cosh\left(\frac{r_1}{2} \right)&r_2^{-1}\cosh\left(\frac{r_2}{2} \right) &s_1^{-1}\cosh\left(\frac{s_1}{2} \right)&s_2^{-1}\cosh\left(\frac{s_2}{2} \right)
\end{pmatrix}.	
\end{equation*} 
Denoting the determinant of a matrix $M$ by $det(M)$, consequently, System \eqref{E--(3.9)}-\eqref{E--(3.18)}  admits a non trivial solution if and only if $\displaystyle{det\left(M\right)}=0$.  Using Gaussian elimination,   $\displaystyle{det\left(M\right)}=0$ is equivalent to  $\displaystyle{det\left(\tilde{M}\right)}=0$, where $\tilde{M}$ is given by
%%%%%%%%%%%%%%%%%%%%Equation%%%%%%%%%%%%%%%%%%%%%%
\begin{equation*}
	\tilde{M}=\begin{pmatrix}
	\sinh\left(\frac{r_1}{2} \right)&\sinh\left(\frac{r_2}{2} \right)&-\sinh\left(\frac{s_1}{2} \right)&-1-e^{-s_2}
	\\ \noalign{\medskip}
\frac{r_1}{i \, c\, \lambda^2}\cosh\left(\frac{r_1}{2} \right)&\frac{r_2}{i \, c\, \lambda^2}\cosh\left(\frac{r_2}{2} \right)&\frac{s_1}{i \, c\, \lambda^2}\cosh\left(\frac{s_1}{2} \right)&\frac{s_2}{i \, c\, \lambda^2}\left(1+ e^{-s_2}\right)	
	\\ \noalign{\medskip}
r_1^2\sinh\left(\frac{r_1}{2} \right)&r_2^2\sinh\left(\frac{r_2}{2} \right)&\left(\lambda^3-(\lambda+1) s_1^2\right)\sinh\left(\frac{s_1}{2} \right)&\left(\lambda^3-(\lambda+1) s_2^2\right)\left(1- e^{-s_2}\right)	\\ \noalign{\medskip}
r_1^{-1}\cosh\left(\frac{r_1}{2} \right)&r_2^{-1}\cosh\left(\frac{r_2}{2} \right) &s_1^{-1}\cosh\left(\frac{s_1}{2} \right)&s_2^{-1}\left(1+ e^{-s_2}\right)
\end{pmatrix}	.
\end{equation*}
One gets that
%%%%%%%%%%%%%%%%%%%%Equation%%%%%%%%%%%%%%%%%%%%%%
\begin{equation}\label{E--(3.21)}
\begin{array}{lll}
\displaystyle{det\left(\tilde{M}\right)=g_1 \cosh\left(\frac{r_1}{2} \right)
\cosh\left(\frac{r_2}{2} \right)
\sinh\left(\frac{s_1}{2} \right)+g_2 \sinh\left(\frac{r_1}{2} \right)
\cosh\left(\frac{r_2}{2} \right)
\cosh\left(\frac{s_1}{2} \right)}\\ \noalign{\medskip}
\hspace{1.7cm}
\displaystyle{+g_3 \cosh\left(\frac{r_1}{2} \right)
\sinh\left(\frac{r_2}{2} \right)
\cosh\left(\frac{s_1}{2} \right) +g_4\sinh\left(\frac{r_1}{2} \right)
\sinh\left(\frac{r_2}{2} \right)
\cosh\left(\frac{s_1}{2} \right)}\\ \noalign{\medskip}
\hspace{1.7cm}
\displaystyle{+g_5 \cosh\left(\frac{r_1}{2} \right)
\sinh\left(\frac{r_2}{2} \right)
\sinh\left(\frac{s_1}{2} \right)+g_6 \sinh\left(\frac{r_1}{2} \right)
\cosh\left(\frac{r_2}{2} \right)
\sinh\left(\frac{s_1}{2} \right)
 }\\ \noalign{\medskip}
\hspace{1.7cm}
\displaystyle{\bigg(-g_1 \cosh\left(\frac{r_1}{2} \right)
\cosh\left(\frac{r_2}{2} \right)
\sinh\left(\frac{s_1}{2} \right)-g_2 \sinh\left(\frac{r_1}{2} \right)
\cosh\left(\frac{r_2}{2} \right)
\cosh\left(\frac{s_1}{2} \right)}\\ \noalign{\medskip}
\hspace{1.8cm}
\displaystyle{-g_3 \cosh\left(\frac{r_1}{2} \right)
\sinh\left(\frac{r_2}{2} \right)
\cosh\left(\frac{s_1}{2} \right) +g_4\sinh\left(\frac{r_1}{2} \right)
\sinh\left(\frac{r_2}{2} \right)
\cosh\left(\frac{s_1}{2} \right)}\\ \noalign{\medskip}
\hspace{1.9cm}
\displaystyle{+g_5 \cosh\left(\frac{r_1}{2} \right)
\sinh\left(\frac{r_2}{2} \right)
\sinh\left(\frac{s_1}{2} \right)+g_6 \sinh\left(\frac{r_1}{2} \right)
\cosh\left(\frac{r_2}{2} \right)
\sinh\left(\frac{s_1}{2} \right)\bigg) e^{-s_2}
 },
\end{array}
\end{equation}
where
%%%%%%%%%%%%%%%%%%%%Equation%%%%%%%%%%%%%%%%%%%%%%
\begin{equation}\label{E--(3.22)}
\left\{\begin{array}{lll}
\displaystyle{g_1=\frac{\left(\lambda+1\right)\left(r_1^2-r_2^2\right)\left(s_1^2-s_2^2\right)}{i\, c\, r_1 r_2 \lambda^2},\quad g_2=\frac{\left(r_2^2-s_1^2\right)\left(\left(\lambda+1\right) s_2^2-\lambda^3-r_1^2\right)}{i\, c\, s_1 r_2 \lambda^2}},\\ \noalign{\medskip}

\displaystyle{g_3=-\frac{\left(r_1^2-s_1^2\right)\left(\left(\lambda+1\right) s_2^2-\lambda^3-r_2^2\right)}{i\, c\, r_1 s_1 \lambda^2}},\quad
\displaystyle{g_4=\frac{\left(r_1^2-r_2^2\right)\left(s_1^2-s_2^2\right)}{i\, c\, s_1 s_2 \lambda^2}},\\ \noalign{\medskip}

\displaystyle{g_5=\frac{\left(r_1^2-s_2^2\right)\left(\left(\lambda+1\right) s_1^2-\lambda^3-r_2^2\right)}{i\, c\, s_2 r_1 \lambda^2}},\quad

\displaystyle{g_6=-\frac{\left(r_2^2-s_2^2\right)\left(\left(\lambda+1\right) s_1^2-\lambda^3-r_1^2\right)}{i\, c\, r_2 s_2 \lambda^2}}.
\end{array}\right.
\end{equation}
%%%%%%%%%%%%%%%%%%%%%%%%%%%%%%%%%%%%%%%%%%%%%%%%%%
                % Proposition %
%%%%%%%%%%%%%%%%%%%%%%%%%%%%%%%%%%%%%%%%%%%%%%%%%%
\begin{prop} \label{Proposition--3.3}
\rm{Under hypothesis \eqref{E--(3.5)}, there exist $n_0\in \mathbb{N}$ sufficiently large  and two sequences $\left(\lambda_{1,n}\right)_{ |n|\geq n_0} $ and $\left(\lambda_{2,n}\right)_{ |n|\geq n_0} $ of simple roots of $\det(\tilde{M})$ (that are also simple eigenvalues of $\mathcal{A}_2$) satisfying the following asymptotic behavior:\\[0.1in]
%%%%%%%%%%%%%%%%%%%%%%Case%%%%%%%%%%%%%%%%%%%%%%%%
\textbf{Case 1.} If there exist no integers $\kappa\in\mathbb{N}$ such that  $c= 2\kappa\pi$ (i.e., $\sf\neq 0$ and $\cf\neq 0$), then
%%%%%%%%%%%%%%%%%%%%Equation%%%%%%%%%%%%%%%%%%%%%%
\begin{eqnarray}
\lambda_{1,n}=2i n\pi -\frac{2\left(1-i\sign(n)\right)\sf^2}{\left(3+\ct\right) \sqrt{\pi |n|}}+O\left(n^{-1}\right), \label{E--(3.23)}\\ \noalign{\medskip}
\lambda_{2,n}=2i n\pi+\pi i+i\arccos\left(\cf\right)-\frac{(1-i\sign(n))\cf^2}{\left(1+\cf^2\right)\sqrt{\pi |n|}}+O\left(n^{-1}\right).\label{E--(3.24)}
\end{eqnarray}
%%%%%%%%%%%%%%%%%%%%%%Case%%%%%%%%%%%%%%%%%%%%%%%%
\textbf{Case 2.} If there exists  $\kappa_0\in\mathbb{N}$ such that  $c=2\left(2\kappa_0+1\right) \pi$, (i.e., $\cf= 0$), then
%%%%%%%%%%%%%%%%%%%%Equation%%%%%%%%%%%%%%%%%%%%%%
\begin{eqnarray}
\lambda_{1,n}=2i n\pi -\frac{1-i\sign(n)}{ \sqrt{\pi |n|}}+O\left(n^{-1}\right),\label{E--(3.25)} \\ \noalign{\medskip}
\lambda_{2,n}=2i n\pi+\frac{3\pi i}{2}+\frac{i \, c^2}{32\pi n}-\frac{\left(8+i (3\pi-2)\right)\, c^2}{128\pi^2 n^2}+O\left(|n|^{-5/2}\right).\label{E--(3.26)}
\end{eqnarray} 
%%%%%%%%%%%%%%%%%%%%%%Case%%%%%%%%%%%%%%%%%%%%%%%%
\textbf{Case 3.} If there exists  $\kappa_1\in\mathbb{N}$ such that  $c=4\kappa_1 \pi$, (i.e., $\sf= 0$), then
%%%%%%%%%%%%%%%%%%%%Equation%%%%%%%%%%%%%%%%%%%%%%
\begin{eqnarray}
\lambda_{1,n}=2i n\pi +\frac{i \, c^2}{32\pi n}-\frac{c^2}{16\pi^2 n^2}+O\left(|n|^{-5/2}\right),\label{E--(3.27)} \\ \noalign{\medskip}
\lambda_{2,n}=2i n\pi+\pi i+\frac{i \, c^2}{32\pi n}-\frac{\left(4+i\pi\right)\, c^2}{64\pi^2 n^2}+O\left(|n|^{-5/2}\right).\label{E--(3.28)}
\end{eqnarray}
Here $\sign$ is used to denote the sign function or signum function.}
\end{prop} 
%%%%%%%%%%%%%%%%%%%%%%%%%%%%%%%%%%%%%%%%%%%%%%%%%%
\noindent The argument for Proposition \ref{Proposition--3.3} relies on the subsequent lemmas.
%%%%%%%%%%%%%%%%%%%%%%%%%%%%%%%%%%%%%%%%%%%%%%%%%%
                % Lemma %
%%%%%%%%%%%%%%%%%%%%%%%%%%%%%%%%%%%%%%%%%%%%%%%%%%
\begin{lem}\label{Lemma--3.4}
\rm{Under hypothesis \eqref{E--(3.5)}, let $\lambda$ be a large eigenvalue of $\mathcal{A}_2$, then $\lambda$ is large root of the following asymptotic equation:
%%%%%%%%%%%%%%%%%%%%Equation%%%%%%%%%%%%%%%%%%%%%%  
\begin{equation}\label{E--(3.29)}
F(\lambda):=f_0(\lambda) +\frac{f_1(\lambda)}{\lambda^{1/2}}+\frac{ f_2(\lambda)}{8\lambda}+\frac{f_3(\lambda)}{8\lambda^{3/2}}+\frac{ f_4(\lambda)}{128\lambda^{2}}+\frac{f_5(\lambda)}{128\lambda^{5/2}}+O\left(\lambda^{-3}\right)=0,
\end{equation} 
where
%%%%%%%%%%%%%%%%%%%%Equation%%%%%%%%%%%%%%%%%%%%%%
\begin{equation}\label{E--(3.30)}
\left\{
\begin{array}{ll}
\displaystyle{f_0(\lambda) =\shLtt+\shLt \ct},
\\ \noalign{\medskip}
\displaystyle{f_1(\lambda) =\chLtt-\chLt \ct},
\\ \noalign{\medskip}
\displaystyle{f_2(\lambda) =c^2\, \chLtt-4c\, \chLt \st},
\\ \noalign{\medskip}
\displaystyle{f_3(\lambda) =c^2\shLtt-4\chLtt+12c\,\shLt \st +4\chLt \ct },

\\ \noalign{\medskip}
\displaystyle{f_4(\lambda) =c^2\left(c^2-56\right)\shLtt-32c^2\,\chLtt+8c^2\left(c\st-8\ct+1\right) \shLt}\\ \noalign{\medskip}
\displaystyle{\hspace{1cm}  -32c\, \left( 8\st+c \ct\right) \ct },

\\ \noalign{\medskip}
\displaystyle{f_5(\lambda) =-40c^2\shLtt+\left(c^4-88c^2+48\right)\chLtt+32c\left(5\st+c\ct\right) \shLt}\\ \noalign{\medskip}
\displaystyle{\hspace{1cm}  -\left(8 c^3\st -16(4c^2-3) \ct-24c^2\right) \ct }.

\end{array}
\right.
\end{equation}}
\end{lem}
%%%%%%%%%%%%%%%%%%%%%%%%%%%%%%%%%%%%%%%%%%%%%%%%%%
                % Proof of  Lemma  %
%%%%%%%%%%%%%%%%%%%%%%%%%%%%%%%%%%%%%%%%%%%%%%%%%%
\begin{proof} Let $\lambda$ be a large eigenvalue of $\mathcal{A}_2$, then $\lambda$ is root of $det\left(\tilde{M}\right)$. In this lemma, we furnish an asymptotic development of the function $det\left(\tilde{M}\right)$ for large $\lambda$.   First, using the asymptotic expansion in  \eqref{E--(3.19)} and \eqref{E--(3.20)}, we get
%%%%%%%%%%%%%%%%%%%%Equation%%%%%%%%%%%%%%%%%%%%%%
\begin{equation}\label{E--(3.31)}
\left\{
\begin{array}{ll}
\displaystyle{r_1=\lambda+\frac{i\, c}{2}+\frac{c^2}{8\lambda}-\frac{i\, c^3}{16\lambda^2}+O\left(\lambda^{-3}\right),\  r_2=\lambda-\frac{i\, c}{2}+\frac{c^2}{8\lambda}+\frac{i\, c^3}{16\lambda^2}+O\left(\lambda^{-3}\right),}\\ \noalign{\medskip}
\displaystyle{
s_1=\lambda-\frac{c^2}{2\lambda^2} +O\left(\lambda^{-5}\right),\ s_2=\lambda^{1/2}-\frac{1}{2\lambda^{1/2}}+\frac{4 c^2+3}{8\lambda^{3/2}}+O\left(\lambda^{-5/2}\right) .}

\end{array}
\right.
\end{equation}
Inserting \eqref{E--(3.31)} in \eqref{E--(3.22)}, we get
%%%%%%%%%%%%%%%%%%%%Equation%%%%%%%%%%%%%%%%%%%%%%
\begin{equation}\label{E--(3.32)}
\left\{
\begin{array}{lll}
\displaystyle{g_1=2-\frac{c^2}{\lambda^2}+O\left(\lambda^{-3}\right),\quad g_2=1+\frac{i\, c}{2\lambda}-\frac{(3c-16i)\, c}{8\lambda^2}+O\left(\lambda^{-3}\right)},

\\ \noalign{\medskip}

\displaystyle{g_3=1-\frac{i\, c}{2\lambda}-\frac{(3c+16i)\, c}{8\lambda^2}+O\left(\lambda^{-3}\right),}\quad
\displaystyle{g_4= 2\lambda^{1/2}-\frac{1}{\lambda^{3/2}}-\frac{4c^2-3}{4\lambda^{5/2}}+O\left(\lambda^{-7/2}\right)},

\\ \noalign{\medskip}

\displaystyle{g_5=\lambda^{1/2}-\frac{1-3i\, c}{2\lambda^{3/2}}-\frac{7c^2-3-10i\, c}{8\lambda^{5/2}}+O\left(\lambda^{-7/2}\right)},\quad
\\ \noalign{\medskip}
\displaystyle{g_6=\lambda^{1/2}-\frac{1+3i\, c}{2\lambda^{3/2}}-\frac{7c^2-3+10i\, c}{8\lambda^{5/2}}+O\left(\lambda^{-7/2}\right)}.
\end{array}\right.
\end{equation}
Inserting \eqref{E--(3.32)} in \eqref{E--(3.21)}, then using the fact that real $\lambda$ is bounded  in S, we get
%%%%%%%%%%%%%%%%%%%%Equation%%%%%%%%%%%%%%%%%%%%%%
\begin{equation}\label{E--(3.33)}
\begin{array}{lll}

\displaystyle{det\left(\tilde{M}\right)}=

\displaystyle{\sinh\left(L_1\right)+\sinh\left(L_2\right)\cosh\left(L_3\right)}

+\displaystyle{\frac{\cosh\left(L_1\right)-\cosh\left(L_2\right)\cosh\left(L_3\right)}{\lambda^{1/2}}}
\\ \noalign{\medskip}
+\displaystyle{\frac{i\, c\,  \cosh\left(L_2\right)\sinh\left(L_3\right)}{2\lambda}
}

-\displaystyle{\frac{\cosh\left(L_1\right)-\cosh\left(L_2\right)\cosh\left(L_3\right)+3i\, c\, \sinh\left(L_2\right)\sinh\left(L_3\right) }{2\lambda^{3/2}}}
\\ \noalign{\medskip}

-\displaystyle{\frac{7c^2\, \sinh\left(L_1\right)+8c^2\,\sinh\left(L_2\right)\cosh\left(L_3\right)-32i\, c\, \cosh\left(L_2\right)\sinh\left(L_3\right)-c^2\sinh\left(L_4\right)  }{16\lambda^{3/2}}}

 \\ \noalign{\medskip}

\displaystyle{-\frac{(11c^2-6) \cosh\left(L_1\right)-(8c^2-6)\,\cosh\left(L_2\right)\cosh\left(L_3\right)+20i\, c\, \sinh\left(L_2\right)\sinh\left(L_3\right)-3c^2\,\cosh\left(L_4\right)  }{16\lambda^{5/2}}}
 \\ \noalign{\medskip}
 
\displaystyle{+\left(\sinh\left(L_1\right)+\sinh\left(L_2\right)\cosh\left(L_3\right)+O\left(\lambda^{-1/2}\right)\right)e^{-s_2}+O\left(\lambda^{-3}\right)},
\end{array}
\end{equation}
where
%%%%%%%%%%%%%%%%%%%%Equation%%%%%%%%%%%%%%%%%%%%%%
\begin{equation*}
L_1=\frac{r_1+r_2+s_1}{2},\ L_2=\frac{s_1}{2},\ L_3=\frac{r_1-r_2}{2},\ L_4=\frac{r_1+r_2-s_1}{2}
.
\end{equation*}
Next, from \eqref{E--(3.31)} and using the fact that real $\lambda$ is bounded S, we get
%%%%%%%%%%%%%%%%%%%%Equation%%%%%%%%%%%%%%%%%%%%%%
\begin{equation}\label{E--(3.34)}
\left\{
\begin{array}{ll}
\displaystyle{
\sinh\left(L_1 \right)=\shLtt+\frac{c^2\,\chLtt}{8\lambda} +\frac{c^2\,\left(c^2\,\shLtt-32\chLtt\right)}{128\lambda^2}+O\left(\lambda^{-3}\right),

}

\\ \noalign{\medskip}

\displaystyle{\cosh\left(L_1 \right)=\chLtt+\frac{c^2\,\shLtt}{8\lambda} +\frac{c^2\,\left(c^2\,\chLtt-32\shLtt\right)}{128\lambda^2}+O\left(\lambda^{-3}\right),}
\\ \noalign{\medskip}

\displaystyle{\sinh\left(L_2 \right)=\shLt-\frac{c^2\,\chLt}{4\lambda^2}+O\left(\lambda^{-4}\right),
}

\\ \noalign{\medskip}

\displaystyle{\cosh\left(L_2 \right)=\chLt-\frac{c^2\,\shLt}{4\lambda^2}+O\left(\lambda^{-4}\right),
}

\\ \noalign{\medskip}

\displaystyle{\sinh\left(L_3 \right)=i\st-\frac{i\, c^3\,\ct}{16\lambda^2}+O\left(\lambda^{-3}\right),
}

\\ \noalign{\medskip}

\displaystyle{\cosh\left(L_3 \right)=\ct+\frac{c^3\,\ct}{16\lambda^2}+O\left(\lambda^{-3}\right),
}

\\ \noalign{\medskip}

\displaystyle{\sinh\left(L_4 \right)=\shLt+O\left(\lambda^{-1}\right),
}\quad

\displaystyle{\cosh\left(L_4 \right)=\chLt+O\left(\lambda^{-1}\right).
}
\end{array}\right.
\end{equation}
On the other hand, from \eqref{E--(3.31)} and \eqref{E--(3.34)}, we obtain
%%%%%%%%%%%%%%%%%%%%Equation%%%%%%%%%%%%%%%%%%%%%%
\begin{equation}\label{E--(3.35)}
\left(\sinh\left(L_1\right)+\sinh\left(L_2\right)\cosh\left(L_3\right)+O\left(\lambda^{-1/2}\right)\right)e^{-s_2}=-\left(\shLtt+\shLt \ct\right)e^{-\sqrt{\lambda}}.
\end{equation}
 Since real part of $\sqrt{\lambda}$ is positive, then  
%%%%%%%%%%%%%%%%%%%%Equation%%%%%%%%%%%%%%%%%%%%%%
\begin{equation*}
\lim_{|\lambda|\to\infty}\frac{e^{-\sqrt{\lambda}}}{\lambda^{3}}=0,
\end{equation*}
hence
%%%%%%%%%%%%%%%%%%%%Equation%%%%%%%%%%%%%%%%%%%%%%
\begin{equation}\label{E--(3.36)}
e^{-\sqrt{\lambda}}=o\left(\lambda^{-3}\right).
\end{equation}
Therefore, from \eqref{E--(3.35)} and \eqref{E--(3.36)}, we get
%%%%%%%%%%%%%%%%%%%%Equation%%%%%%%%%%%%%%%%%%%%%%
\begin{equation}\label{E--(3.37)}
\left(\sinh\left(L_1\right)+\sinh\left(L_2\right)\cosh\left(L_3\right)+O\left(\lambda^{-1/2}\right)\right)e^{-s_2}=o\left(\lambda^{-3}\right).
\end{equation}
Finally, inserting \eqref{E--(3.34)} and \eqref{E--(3.37)} in \eqref{E--(3.33)}, we get $\lambda$ is large root of $F$, where $F$ defined in \eqref{E--(3.29)}.
\end{proof}
%%%%%%%%%%%%%%%%%%%%%%%%%%%%%%%%%%%%%%%%%%%%%%%%%%
                % Lemma %
%%%%%%%%%%%%%%%%%%%%%%%%%%%%%%%%%%%%%%%%%%%%%%%%%%
 \begin{lem}\label{Lemma--3.5}
 \rm{Under hypothesis \eqref{E--(3.5)}, there exist $n_0\in \mathbb{N}$ sufficiently large  and two sequences $\left(\lambda_{1,n}\right)_{ |n|\geq n_0} $ and $\left(\lambda_{2,n}\right)_{ |n|\geq n_0} $ of simple roots of $F$ (that are also simple eigenvalues of $\mathcal{A}_2$) satisfying the following asymptotic behavior:
%%%%%%%%%%%%%%%%%%%%Equation%%%%%%%%%%%%%%%%%%%%%%
\begin{equation}\label{E--(3.38)}
\lambda_{1,n}=2i\pi n+i\pi +\epsilon_{1,n},\quad \text{such that }\lim_{|n|\to+\infty}\epsilon_{1,n}=0
\end{equation} 
 and
%%%%%%%%%%%%%%%%%%%%Equation%%%%%%%%%%%%%%%%%%%%%% 
 \begin{equation}\label{E--(3.39)}
\lambda_{2,n}=2 n\pi i+i\pi+i\, \arccos\left(\cos^2\left(\frac{c}{4}\right)\right)+\epsilon_{2,n},\quad \text{such that }\lim_{|n|\to+\infty}\epsilon_{2,n}=0.
\end{equation}
}\end{lem}
%%%%%%%%%%%%%%%%%%%%%%%%%%%%%%%%%%%%%%%%%%%%%%%%%%
                % Proof of  Lemma  %
%%%%%%%%%%%%%%%%%%%%%%%%%%%%%%%%%%%%%%%%%%%%%%%%%%
\begin{proof} First,  we look at the roots of $f_0$. From \eqref{E--(3.30)}, we deduce that $f_0$ can be written as
%%%%%%%%%%%%%%%%%%%%Equation%%%%%%%%%%%%%%%%%%%%%% 
\begin{equation}\label{E--(3.40)}
f_0(\lambda)=2\sinh\left(\frac{\lambda}{2}\right) \left( \cosh\left(\lambda\right) +\cos^2\left(\frac{c}{4}\right)\right).
\end{equation}
The roots of $f_0$ are given by
\begin{equation*}\left\{
\begin{array}{ll}
\displaystyle{
\mu_{1,n}=2 n\pi i,}& n\in\mathbb{Z},\\ \noalign{\medskip}
\displaystyle{ \mu_{2,n}=2 n\pi i+i\pi +i\, \arccos\left(\cos^2\left(\frac{c}{4}\right)\right),}& n\in\mathbb{Z}.
 \end{array}\right.
\end{equation*} 
Now with the help of Rouch\'{e}'s theorem, we will show that the roots of $F$ are close to $f_0$. \\[0.1in]
Let us start with the first family $\mu_{1,n}$. Let $B_n=B\left(2n\pi i, r_n\right)$ be the ball of centrum $2n\pi i$ and radius $r_n=\frac{1}{|n|^{\frac{1}{4}}}$ and $\lambda\in \partial B_n$; i.e., $\lambda=2n\pi i+r_n e^{i\theta},\ \theta\in[0,2\pi)$. Then 
%%%%%%%%%%%%%%%%%%%%Equation%%%%%%%%%%%%%%%%%%%%%%
\begin{equation}\label{E--(3.41)}
\sinh\left(\frac{\lambda}{2}\right)=\left(-1\right)^n\sinh\left(\frac{r_n e^{i\theta}}{2}\right)=\frac{ \left(-1\right)^n\, r_n e^{i\theta} }{2}+O(r_n^2),\ \cosh(\lambda)=\cosh\left(r_n e^{i\theta}\right)=1+O(r_n^2).
\end{equation}
Inserting  \eqref{E--(3.41)}  in \eqref{E--(3.40)}, we get
%%%%%%%%%%%%%%%%%%%%Equation%%%%%%%%%%%%%%%%%%%%%%
\begin{equation*}
f_0(\lambda)= \left(-1\right)^n\, r_n e^{i\theta}\left(1+\cos^2\left(\frac{c}{4}\right)\right)+O(r_n^3).
\end{equation*}
It follows that there exists a positive constant $C$ such that
\begin{equation*}
\forall\ \lambda\in \partial B_n, \quad \left|f_0\left(\lambda\right)\right|\geq C\, r_n=\frac{C}{|n|^{\frac{1}{4}}}.
\end{equation*}
On the other hand, from  \eqref{E--(3.29)}, we deduce that 
%%%%%%%%%%%%%%%%%%%%Equation%%%%%%%%%%%%%%%%%%%%%%
\begin{equation*}
\left|F(\lambda)-f_0(\lambda)\right|=O\left(\frac{1}{\sqrt{\lambda}}\right)=O\left(\frac{1}{\sqrt{|n|}}\right).
\end{equation*}
It follows that, for $|n|$ large enough
%%%%%%%%%%%%%%%%%%%%Equation%%%%%%%%%%%%%%%%%%%%%%
\begin{equation*}
\forall\ \lambda\in \partial B_n, \quad \left|F(\lambda)-f_0\left(\lambda\right)\right|<\left|f_0(\lambda)\right|.
\end{equation*}
Hence, with the help of Rouch\'{e}'s theorem,  there exists $n_0\in \mathbb{N}^*$ large enough, such that $\forall\ |n|\geq n_0\ \ \left( n\in \mathbb{Z}^*\right),$   the first branch of roots of $F$, denoted by $\lambda_{1,n}$ are close to  $\mu_{1,n}$, hence we get \eqref{E--(3.38)}. The same procedure yields \eqref{E--(3.39)}. Thus, the proof is complete.	
\end{proof}
%%%%%%%%%%%%%%%%%%%%%%%%%%%%%%%%%%%%%%%%%%%%%%%%%%
                % Remark %
%%%%%%%%%%%%%%%%%%%%%%%%%%%%%%%%%%%%%%%%%%%%%%%%%%
\begin{rk}\label{Remark--3.6}
\rm{ From Lemma \ref{Lemma--3.5}, we deduce that the real part of the  eigenvalues of $\mathcal{A}_2$ tends  to zero, and this is enough to get Theorem \ref{Theorem--3.2}. But we look forward to knowing the real part of $\lambda_{1,n}$ and $\lambda_{2,n}$. Since in the next section, we will use the real part of $\lambda_{1,n}$ and $\lambda _{2,n}$ for the optimality of polynomial stability.}\xqed{$\square$}
\end{rk}
%%%%%%%%%%%%%%%%%%%%%%%%%%%%%%%%%%%%%%%%%%%%%%%%%%
\noindent We are now in a position to conclude the proof of Proposition \ref{Proposition--3.3}.\\[0.1in]
%%%%%%%%%%%%%%%%%%%%%%%%%%%%%%%%%%%%%%%%%%%%%%%%%%
                % Proof of  Proposition  %
%%%%%%%%%%%%%%%%%%%%%%%%%%%%%%%%%%%%%%%%%%%%%%%%%%
\noindent \textbf{Proof of Proposition  \ref{Proposition--3.3}.} The proof is divided into two steps. \\[0.1in]
%%%%%%%%%%%%%%%%%%%%%%Step%%%%%%%%%%%%%%%%%%%%%%%% 
\textbf{Step 1.} Calculation of $\epsilon_{1,n}$. From \eqref{E--(3.38)}, we have
%%%%%%%%%%%%%%%%%%%%Equation%%%%%%%%%%%%%%%%%%%%%%
\begin{equation}\label{E--(3.42)}
\left\{
\begin{array}{ll}
\displaystyle{\cosh\left(\frac{3\lambda_{1,n}}{2}\right)=(-1)^n\, \cosh\left(\frac{3\epsilon_{1,n}}{2}\right),\ \sinh\left(\frac{3\lambda_{1,n}}{2}\right)=(-1)^n\, \sinh\left(\frac{3\epsilon_{1,n}}{2}\right) ,}

\\ \noalign{\medskip}

\displaystyle{\cosh\left(\frac{\lambda_{1,n}}{2}\right)=(-1)^n\, \cosh\left(\frac{\epsilon_{1,n}}{2}\right),\ \sinh\left(\frac{\lambda_{1,n}}{2}\right)=(-1)^n\, \sinh\left(\frac{\epsilon_{1,n}}{2}\right) ,}
\end{array}\right.
\end{equation}
and
%%%%%%%%%%%%%%%%%%%%Equation%%%%%%%%%%%%%%%%%%%%%%
\begin{equation}\label{E--(3.43)}
\left\{
\begin{array}{ll}
\displaystyle{\frac{1}{\lambda_{1,n}}= -\frac{i}{2\pi n}+O\left(\epsilon_{1,n}\, n^{-2}\right)+O\left(n^{-3}\right),\ \frac{1}{\lambda^2_{1,n}}= -\frac{1}{4\pi^2 n^2}+O\left(n^{-3}\right)},

 \\ \noalign{\medskip}
\displaystyle{\frac{1}{\sqrt{\lambda_{1,n}}}=\frac{1-i\sign(n)}{2\sqrt{\pi |n|}} +O\left(\epsilon_{1,n}\, |n|^{-3/2}\right)+O\left(|n|^{-5/2}\right)},

 \\ \noalign{\medskip}

\displaystyle{\frac{1}{\sqrt{\lambda^3_{1,n}}}=\frac{-1-i\sign(n)}{4\sqrt{\pi^3 |n|^3}}+O\left(|n|^{-5/2}\right),\ \frac{1}{\sqrt{\lambda^5_{1,n}}}=O\left(|n|^{-5/2}\right)}.

\end{array}\right.
\end{equation}
On the other hand, since $\lim_{|n|\to+\infty}\epsilon_{1,n}=0$, we have the asymptotic expansion
%%%%%%%%%%%%%%%%%%%%Equation%%%%%%%%%%%%%%%%%%%%%%
\begin{equation}\label{E--(3.44)}
\left\{
\begin{array}{ll}
\displaystyle{\cosh\left(\frac{3\epsilon_{1,n}}{2}\right)=1+\frac{9\epsilon_{1,n}^2}{8}+O(\epsilon_{1,n}^4)},\ 
\displaystyle{\sinh\left(\frac{3\epsilon_{1,n}}{2}\right)=\frac{3\epsilon_{1,n}}{2}+O(\epsilon_{1,n}^3) },\\ \noalign{\medskip}
\displaystyle{\cosh\left(\frac{\epsilon_{1,n}}{2}\right)=1+\frac{\epsilon_{1,n}^2}{8}+O(\epsilon_{1,n}^4)},\
\displaystyle{\sinh\left(\frac{\epsilon_{1,n}}{2}\right)=\frac{\epsilon_{1,n}}{2}+O(\epsilon_{1,n}^3) }.
\end{array}
\right.
\end{equation}
Inserting \eqref{E--(3.44)} in \eqref{E--(3.42)}, we get
%%%%%%%%%%%%%%%%%%%%Equation%%%%%%%%%%%%%%%%%%%%%%
\begin{equation}\label{E--(3.45)}
\left\{\begin{array}{ll}

\displaystyle{\cosh\left(\frac{3\lambda_{1,n}}{2}\right)=(-1)^n+\frac{9(-1)^n\, \epsilon_{1,n}}{8}+O(\epsilon_{1,n}^4),\  \sinh\left(\frac{3\lambda_{1,n}}{2}\right)=\frac{3(-1)^n\, \epsilon_{1,n}}{2}+O(\epsilon_{1,n}^3) ,}

\\ \noalign{\medskip}

\displaystyle{\cosh\left(\frac{\lambda_{1,n}}{2}\right)=(-1)^n+\frac{(-1)^n\, \epsilon_{1,n}}{8}+O(\epsilon_{1,n}^4),\  \sinh\left(\frac{\lambda_{1,n}}{2}\right)=\frac{(-1)^n\, \epsilon_{1,n}}{2}+O(\epsilon_{1,n}^3) .}

\end{array}\right.
\end{equation}
Inserting \eqref{E--(3.43)} and \eqref{E--(3.45)} in \eqref{E--(3.29)}, we get
\begin{equation}\label{E--(3.46)}
\begin{array}{ll}

\displaystyle{\frac{\epsilon_{1,n }}{2}\left(3+\cos\left(\frac{c}{2}\right)\right)+\frac{\left(1-i\sign(n)\right)\left(1-\cos\left(\frac{c}{2}\right)\right)}{2\sqrt{\pi\, |n|}}+\frac{i\, c\left(4\sin\left(\frac{c}{2}\right)-c\right)}{16\pi n}}

\\ \noalign{\medskip}

\hspace{1cm}\displaystyle{+\frac{\left(1+i\sign(n)\right)\left(1-\cos\left(\frac{c}{2}\right)\right)}{8\sqrt{\pi^3\, |n|^3}} +\frac{8 c \st+\left(1+\ct\right)c^2}{16\pi^2n^2}  }
\\ \noalign{\medskip}

\hspace{1cm}\displaystyle{
+O\left(|n|^{-5/2}\right)+O\left(\epsilon_{1,n }\, |n|^{-3/2}\right)+O\left(\epsilon_{1,n }^2\, |n|^{-1/2}\right)+O\left(\epsilon_{1,n }^3\right)=0}.

\end{array}
\end{equation}
We distinguish two cases:\\[0.1in]
%%%%%%%%%%%%%%%%%%%%%%Case%%%%%%%%%%%%%%%%%%%%%%%% 
\textbf{Case 1.} If $\sin\left(\frac{c}{4}\right)\neq0,$ then
%%%%%%%%%%%%%%%%%%%%Equation%%%%%%%%%%%%%%%%%%%%%%
\begin{equation*}
1-\cos\left(\frac{c}{2}\right)=2\sin^2\left(\frac{c}{4}\right)\neq0,
\end{equation*}
therefore, from \eqref{E--(3.46)}, we get
%%%%%%%%%%%%%%%%%%%%Equation%%%%%%%%%%%%%%%%%%%%%%
\begin{equation*}
\frac{\epsilon_{1,n}}{2}\left(3+ \cos\left(\frac{c}{2}\right)\right)+ \frac{\sin^2\left(\frac{c}{4}\right)\left(1-i \sign\left(n\right)\right)}{\sqrt{|n|\pi}}+O\left(\epsilon_{1,n}^3\right)+O\left(|n|^{-1/2}\,\epsilon_{1,n}^2\right)+O\left(n^{-1}\right)=0,
\end{equation*}
hence, we get
%%%%%%%%%%%%%%%%%%%%Equation%%%%%%%%%%%%%%%%%%%%%%
\begin{equation}\label{E--(3.47)}
\epsilon_{1,n}=- \frac{2\sin^2\left(\frac{c}{4}\right)\left(1-i \sign\left(n\right)\right)}{\left(3+ \cos\left(\frac{c}{2}\right)\right)\sqrt{|n|\pi}}+O\left(n^{-1}\right).
\end{equation}
Inserting \eqref{E--(3.47)} in \eqref{E--(3.38)}, we get \eqref{E--(3.23)} and \eqref{E--(3.25)}.\\[0.1in]
%%%%%%%%%%%%%%%%%%%%%%Case%%%%%%%%%%%%%%%%%%%%%%%%
\textbf{Case 2.} If $\sin\left(\frac{c}{4}\right)=0,$ then
%%%%%%%%%%%%%%%%%%%%Equation%%%%%%%%%%%%%%%%%%%%%%
\begin{equation*}
1-\cos\left(\frac{c}{2}\right)=2\sin^2\left(\frac{c}{4}\right)=0,\ \sin\left(\frac{c}{2}\right)=2\sin\left(\frac{c}{4}\right)\cos\left(\frac{c}{4}\right)=0,
\end{equation*}
therefore, from \eqref{E--(3.46)}, we get
%%%%%%%%%%%%%%%%%%%%Equation%%%%%%%%%%%%%%%%%%%%%%
\begin{equation}\label{E--(3.48)}
2{\epsilon_{1,n }}-\frac{i\, c^2}{16\pi n} +\frac{c^2}{8\pi^2n^2} 
+O\left(|n|^{-5/2}\right)+O\left(\epsilon_{1,n }\, |n|^{-3/2}\right)+O\left(\epsilon_{1,n }^2\, |n|^{-1/2}\right)+O\left(\epsilon_{1,n }^3\right)=0.
\end{equation} 
Solving Equation  \eqref{E--(3.48)}, we get
%%%%%%%%%%%%%%%%%%%%Equation%%%%%%%%%%%%%%%%%%%%%%
\begin{equation}\label{E--(3.49)}
\epsilon_{1,n}=\frac{i\, c^2}{32\pi n}-\frac{c^2}{16\pi^2 n^2}+O\left(|n|^{-5/2}\right).
\end{equation}
Inserting \eqref{E--(3.49)} in \eqref{E--(3.38)}, we get \eqref{E--(3.27)}.\\[0.1in]
%%%%%%%%%%%%%%%%%%%%%%Step%%%%%%%%%%%%%%%%%%%%%%%%
\textbf{Step 2.} Calculation of $\epsilon_{2,n}$. We distinguish three cases:\\[0.1in]
%%%%%%%%%%%%%%%%%%%%%%Case%%%%%%%%%%%%%%%%%%%%%%%%
\textbf{Case 1.} If $\sf\neq 0$ and $\cf\neq 0$, then 
$0<\cos^2\left(\frac{c}{4}\right)< 1.$ Therefore
%%%%%%%%%%%%%%%%%%%%Equation%%%%%%%%%%%%%%%%%%%%%%
\begin{equation*}
\zeta:=\arccos\left(\cos^2\left(\frac{c}{4}\right)\right)\in \left(0,\frac{\pi}{2}\right).
\end{equation*}
From \eqref{E--(3.39)}, we have
%%%%%%%%%%%%%%%%%%%%Equation%%%%%%%%%%%%%%%%%%%%%%
\begin{equation}\label{E--(3.50)}
\frac{1}{\sqrt{\lambda_{2,n}}}=\frac{1-i\sign(n)}{2\sqrt{\pi |n|}}+O\left(|n|^{-3/2}\right)\ \ \ \text{and} \ \ \ \frac{1}{{\lambda_{2,n}}}=O(n^{-1}).
\end{equation}
Inserting \eqref{E--(3.39)} and \eqref{E--(3.50)} in \eqref{E--(3.29)}, we get
%%%%%%%%%%%%%%%%%%%%Equation%%%%%%%%%%%%%%%%%%%%%%
\begin{equation}\label{E--(3.51)}
\begin{array}{ll}
\displaystyle{
2\sinh\left(\frac{\lambda_{2,n}}{2}\right) \left( \cosh\left(\lambda_{2,n}\right) +\cos^2\left(\frac{c}{4}\right)\right)}\\ \noalign{\medskip}

\hspace{1.2cm}\displaystyle{+\frac{\cosh\left(\frac{\lambda_{2,n}}{2}\right) \left( \cosh\left(\lambda_{2,n}\right) -\cos^2\left(\frac{c}{4}\right)\right)\left(1-i\sign(n)\right)}{\sqrt{\pi |n|}}+O(n^{-1})=0.}
\end{array}
\end{equation}
From  \eqref{E--(3.39)}, we obtain 
%%%%%%%%%%%%%%%%%%%%Equation%%%%%%%%%%%%%%%%%%%%%%
\begin{equation}\label{E--(3.52)}
\left\{
\begin{array}{ll}
\displaystyle{\cosh(\lambda_{2,n}) =-\cos^2\left(\frac{c}{4}\right)\cosh\left(\epsilon_{2,n}\right)-i\sin\left( \zeta \right)\sinh\left(\epsilon_{2,n}\right),}\\ \noalign{\medskip}
\displaystyle{\cosh\left(\frac{\lambda_{2,n}}{2}\right) =(-1)^n\, \left(
-\sin\left(\frac{\zeta }{2}\right)\cosh\left(\frac{\epsilon_{2,n}}{2}\right)+i\, \cos\left(\frac{\zeta }{2}\right)\sinh\left(\frac{\epsilon_{2,n}}{2}\right)\right),}\\ \noalign{\medskip}

\\ \noalign{\medskip}
\displaystyle{\sinh\left(\frac{\lambda_{2,n}}{2}\right) =(-1)^n\, \left(
-\sin\left(\frac{\zeta }{2}\right)\sinh\left(\frac{\epsilon_{2,n}}{2}\right)+i\, \cos\left(\frac{\zeta }{2}\right)\cosh\left(\frac{\epsilon_{2,n}}{2}\right)\right).}

\end{array}\right.
\end{equation}
Since $\zeta=\arccos\left(\cos^2\left(\frac{c}{4}\right)\right)\in \left(0,\frac{\pi}{2}\right)$, we have
%%%%%%%%%%%%%%%%%%%%Equation%%%%%%%%%%%%%%%%%%%%%%
\begin{equation}\label{E--(3.53)}
\sin\left( \zeta \right)=\left|\sin\left(\frac{c}{4}\right)\right|\sqrt{1+\cos^2\left(\frac{c}{4}\right)},\ \cos\left( \frac{\zeta}{2} \right)=\frac{\sqrt{1+\cos^2\left(\frac{c}{4}\right)}}{\sqrt{2}}
,\ \sin\left( \frac{\zeta}{2} \right)=\frac{\left|\sin\left(\frac{c}{4}\right)\right|}{\sqrt{2}}.
\end{equation}
On the other hand, since $\lim_{|n|\to+\infty}\epsilon_{2,n}=0$, we have the asymptotic expansion
%%%%%%%%%%%%%%%%%%%%Equation%%%%%%%%%%%%%%%%%%%%%%
\begin{equation}\label{E--(3.54)}
\left\{
\begin{array}{ll}
\displaystyle{
\cosh\left(\epsilon_{2,n}\right)=1+O(\epsilon_{2,n}^2),\
\sinh\left(\epsilon_{2,n}\right)=\epsilon_{2,n}+O(\epsilon_{2,n}^3),}
\\ \noalign{\medskip}
\displaystyle{
\cosh\left(\frac{\epsilon_{2,n}}{2}\right)=1+O(\epsilon_{2,n}^2),\
\sinh\left(\frac{\epsilon_{2,n}}{2}\right)=\frac{\epsilon_{2,n}}{2}+O(\epsilon_{2,n}^3).}
\end{array}\right.
\end{equation}
Inserting \eqref{E--(3.53)} and \eqref{E--(3.54)} in \eqref{E--(3.52)}, we get
%%%%%%%%%%%%%%%%%%%%Equation%%%%%%%%%%%%%%%%%%%%%%
\begin{equation}\label{E--(3.55)}
\left\{
\begin{array}{ll}
\displaystyle{\cosh(\lambda_{2,n}) =-\cos^2\left(\frac{c}{4}\right)-i\, \epsilon_{2,n} \,\left|\sin\left(\frac{c}{4}\right)\right|\sqrt{1+\cos^2\left(\frac{c}{4}\right)}+O(\epsilon_{2,n}^2),}\\ \noalign{\medskip}
\displaystyle{\cosh\left(\frac{\lambda_{2,n}}{2}\right) =\frac{(-1)^n}{\sqrt{2}}\, \left(
\frac{i\, \epsilon_{2,n}\, \sqrt{1+\cos^2\left(\frac{c}{4}\right)}}{2}-\left|\sin\left(\frac{c}{4}\right)\right|\right) +O(\epsilon_{2,n}^2),}\\ \noalign{\medskip}

\\ \noalign{\medskip}
\displaystyle{\sinh\left(\frac{\lambda_{2,n}}{2}\right) =-\frac{(-1)^n}{2\sqrt{2}}\, \left(
\left|\sin\left(\frac{c}{4}\right)\right|\, \epsilon_{2,n}-2i\, \sqrt{1+\cos^2\left(\frac{c}{4}\right)}\right) +O(\epsilon_{2,n}^2).}\\ \noalign{\medskip}

\end{array}\right.
\end{equation}
Inserting \eqref{E--(3.55)} in \eqref{E--(3.51)}, we get
%%%%%%%%%%%%%%%%%%%%Equation%%%%%%%%%%%%%%%%%%%%%%
\begin{equation*}\begin{array}{ll}
\displaystyle{
\sqrt{2}\, (-1)^n \,  \left|\sin\left(\frac{c}{4}\right)\right|\,  \left(1+\cos^2\left(\frac{c}{4}\right)\right)\,  \left(\epsilon_{2,n}+\frac{ \cos^2\left(\frac{c}{4}\right)\left(1-i\sign(n)\right)}{ \left(1+\cos^2\left(\frac{c}{4}\right)\right)\, \sqrt{\pi |n|}}
 \right)

  }\\ \noalign{\medskip}

\hspace{1.2cm}\displaystyle{+O(n^{-1})+O(\epsilon_{2,n}^2)+O\left(|n|^{-1/2}\,\epsilon_{2,n}\right)=0,}
\end{array}
\end{equation*}
since in this case $\sf\neq 0$, then we get
%%%%%%%%%%%%%%%%%%%%Equation%%%%%%%%%%%%%%%%%%%%%%
\begin{equation}\label{E--(3.56)}
\epsilon_{2,n}=-\frac{ \cos^2\left(\frac{c}{4}\right)\left(1-i\sign(n)\right)}{ \left(1+\cos^2\left(\frac{c}{4}\right)\right)\, \sqrt{\pi |n|}}+O(n^{-1}).
\end{equation}
Inserting \eqref{E--(3.56)} in \eqref{E--(3.46)},  we get \eqref{E--(3.24)}.\\[0.1in]
%%%%%%%%%%%%%%%%%%%%%%Case%%%%%%%%%%%%%%%%%%%%%%%%
\textbf{Case 2.} If  $\cf= 0$,  then
%%%%%%%%%%%%%%%%%%%%Equation%%%%%%%%%%%%%%%%%%%%%%
\begin{equation}\label{E--(3.57)}
\cos\left(\frac{c}{2}\right)=-1,\ \sin\left(\frac{c}{2}\right)=0.
\end{equation}
In this case $\lambda_{2,n}$ becomes
%%%%%%%%%%%%%%%%%%%%Equation%%%%%%%%%%%%%%%%%%%%%%
\begin{equation}\label{E--(3.58)}
\lambda_{2,n}=2i n\pi+\frac{3\pi\, i}{2}+\epsilon_{2,n}.
\end{equation}
Therefore, we have
%%%%%%%%%%%%%%%%%%%%Equation%%%%%%%%%%%%%%%%%%%%%%
\begin{equation}\label{E--(3.59)}
\left\{
\begin{array}{ll}
\displaystyle{\cosh\left(\frac{3\lambda_{2,n}}{2}\right)=\frac{(-1)^n}{\sqrt{2}}\, \left(\cosh\left(\frac{3\epsilon_{2,n}}{2}\right)+i\sinh\left(\frac{3\epsilon_{2,n}}{2}\right)\right),}

\\ \noalign{\medskip}

\displaystyle{

\sinh\left(\frac{3\lambda_{2,n}}{2}\right)=\frac{(-1)^n}{\sqrt{2}}\, \left(i\cosh\left(\frac{3\epsilon_{2,n}}{2}\right)+\sinh\left(\frac{3\epsilon_{2,n}}{2}\right)\right),}

\\ \noalign{\medskip}
\displaystyle{\cosh\left(\frac{\lambda_{2,n}}{2}\right)=\frac{(-1)^n}{\sqrt{2}}\, \left(-\cosh\left(\frac{\epsilon_{2,n}}{2}\right)+i\sinh\left(\frac{\epsilon_{2,n}}{2}\right)\right),}

\\ \noalign{\medskip}

\displaystyle{

\sinh\left(\frac{\lambda_{2,n}}{2}\right)=\frac{(-1)^n}{\sqrt{2}}\, \left(i\cosh\left(\frac{\epsilon_{2,n}}{2}\right)-\sinh\left(\frac{\epsilon_{2,n}}{2}\right)\right).}
\end{array}\right.
\end{equation}
On the other hand, since $\lim_{|n|\to+\infty}\epsilon_{2,n}=0$, we have the asymptotic expansion
%%%%%%%%%%%%%%%%%%%%Equation%%%%%%%%%%%%%%%%%%%%%%
\begin{equation}\label{E--(3.60)}
\left\{
\begin{array}{ll}
\displaystyle{\cosh\left(\frac{3\epsilon_{2,n}}{2}\right)=1+\frac{9\epsilon_{2,n}^2}{8}+O(\epsilon_{2,n}^4)},\ 
\displaystyle{\sinh\left(\frac{3\epsilon_{2,n}}{2}\right)=\frac{3\epsilon_{2,n}}{2}+O(\epsilon_{2,n}^3) },\\ \noalign{\medskip}
\displaystyle{\cosh\left(\frac{\epsilon_{2,n}}{2}\right)=1+\frac{\epsilon_{2,n}^2}{8}+O(\epsilon_{2,n}^4)},\
\displaystyle{\sinh\left(\frac{\epsilon_{2,n}}{2}\right)=\frac{\epsilon_{2,n}}{2}+O(\epsilon_{2,n}^3) }.
\end{array}
\right.
\end{equation}
Inserting \eqref{E--(3.60)} in \eqref{E--(3.59)}, we get
%%%%%%%%%%%%%%%%%%%%Equation%%%%%%%%%%%%%%%%%%%%%%
\begin{equation}\label{E--(3.61)}
\left\{
\begin{array}{ll}
\displaystyle{\cosh\left(\frac{3\lambda_{2,n}}{2}\right)=\frac{(-1)^n}{\sqrt{2}}\, \left(1+ \frac{3i\, \epsilon_{2,n}}{2}+\frac{9\epsilon_{2,n}^2}{8}+O(\epsilon_{2,n}^3)\right),}

\\ \noalign{\medskip}

\displaystyle{

\sinh\left(\frac{3\lambda_{2,n}}{2}\right)=\frac{(-1)^n}{\sqrt{2}}\, \left(i+ \frac{3\, \epsilon_{2,n}}{2}+\frac{9i\, \epsilon_{2,n}^2}{8}+O(\epsilon_{2,n}^3)\right),}

\\ \noalign{\medskip}
\displaystyle{\cosh\left(\frac{\lambda_{2,n}}{2}\right)=\frac{(-1)^n}{\sqrt{2}}\, \left(-1+ \frac{i\, \epsilon_{2,n}}{2}-\frac{\epsilon_{2,n}^2}{8}+O(\epsilon_{2,n}^3)\right),}

\\ \noalign{\medskip}

\displaystyle{

\sinh\left(\frac{\lambda_{2,n}}{2}\right)=\frac{(-1)^n}{\sqrt{2}}\, \left(i- \frac{ \epsilon_{2,n}}{2}+\frac{i\, \epsilon_{2,n}^2}{8}+O(\epsilon_{2,n}^3)\right).}
\end{array}\right.
\end{equation}
Moreover, from \eqref{E--(3.58)}, we get
%%%%%%%%%%%%%%%%%%%%Equation%%%%%%%%%%%%%%%%%%%%%%
\begin{equation}\label{E--(3.62)}
\left\{
\begin{array}{ll}
\displaystyle{\frac{1}{\lambda_{2,n}}= -\frac{i}{2\pi n}+\frac{3 i \pi}{8\pi^2 n^2}+O\left(\epsilon_{2,n}\, n^{-2}\right)+O\left(n^{-3}\right),\ \frac{1}{\lambda^2_{2,n}}= -\frac{1}{4\pi^2 n^2}+O\left(n^{-3}\right)},

 \\ \noalign{\medskip}
\displaystyle{\frac{1}{\sqrt{\lambda_{2,n}}}=\frac{1-i\sign(n)}{2\sqrt{\pi |n|}} +\frac{3\, (-\sign(n)+i)}{16\sqrt{\pi |n|^3}} +O\left(\epsilon_{2,n}\, |n|^{-3/2}\right)+O\left(|n|^{-5/2}\right)},

 \\ \noalign{\medskip}

\displaystyle{\frac{1}{\sqrt{\lambda^3_{2,n}}}=\frac{-1-i\sign(n)}{4\sqrt{\pi^3 |n|^3}}+O\left(|n|^{-5/2}\right),\ \frac{1}{\sqrt{\lambda^5_{2,n}}}=O\left(|n|^{-5/2}\right)}.

\end{array}\right.
\end{equation}
Inserting \eqref{E--(3.57)},  \eqref{E--(3.61)},  and \eqref{E--(3.62)} in \eqref{E--(3.29)}, we get
%%%%%%%%%%%%%%%%%%%%Equation%%%%%%%%%%%%%%%%%%%%%%
\begin{equation}\label{E--(3.63)}
\begin{array}{ll}

\displaystyle{\frac{i\, \epsilon_{2,n}^2}{2}+\left(1+\frac{\sign(n)+i}{2\sqrt{\pi\, |n|}}+\frac{3 c^2}{64\pi n}\right)\, \epsilon_{2,n}-\frac{i\, c^2}{32\pi n}+\frac{(\sign(n)-i)\, c^2}{64\sqrt{\pi^3|n|^3}}+\frac{\left(64-i\left(c^2-24\pi+16\right)\right)\, c^2}{1024 \pi^2 n^2}}
\\ \noalign{\medskip}

\hspace{1cm}\displaystyle{
+O\left(|n|^{-5/2}\right)+O\left(\epsilon_{2,n }\, |n|^{-3/2}\right)+O\left(\epsilon_{2,n }^2\, |n|^{-1/2}\right)+O\left(\epsilon_{2,n }^3\right)=0}.

\end{array}
\end{equation}
From \eqref{E--(3.63)}, we get
%%%%%%%%%%%%%%%%%%%%Equation%%%%%%%%%%%%%%%%%%%%%%
\begin{equation*}
\epsilon_{2,n}-\frac{i\, c^2}{32\pi n}+O\left(\epsilon_{2,n }\, |n|^{-1/2}\right)+O\left(\epsilon_{2,n }^2\right)=0,
\end{equation*}
hence 
%%%%%%%%%%%%%%%%%%%%Equation%%%%%%%%%%%%%%%%%%%%%%
\begin{equation}\label{E--(3.64)}
\epsilon_{2,n}=\frac{i\, c^2}{32\pi n}+\frac{\xi_{n}}{n}, \quad \text{such that }\lim_{|n|\to+\infty}\xi_{n}=0.
\end{equation}
Inserting \eqref{E--(3.64)} in \eqref{E--(3.63)}, we get
%%%%%%%%%%%%%%%%%%%%Equation%%%%%%%%%%%%%%%%%%%%%%
\begin{equation*}
\frac{\xi_{n}}{n}+\frac{\left(8+i\, (3\pi-2)\right)\, c^2 }{128\pi^2 n^2}+O\left(\xi_n\, |n|^{-3/2}\right)+O\left(|n|^{-5/2}\right)=0,
\end{equation*}
therefore
%%%%%%%%%%%%%%%%%%%%Equation%%%%%%%%%%%%%%%%%%%%%%
\begin{equation}\label{E--(3.65)}
\xi_{n}=-\frac{\left(8+i\, (3\pi-2)\right)\, c^2 }{128\pi^2 n}+O(n^{-3/2}).
\end{equation}
Inserting \eqref{E--(3.64)} in \eqref{E--(3.65)}, we get
%%%%%%%%%%%%%%%%%%%%Equation%%%%%%%%%%%%%%%%%%%%%%
\begin{equation}\label{E--(3.66)}
\epsilon_{2,n}=\frac{i\, c^2}{32\pi n}-\frac{\left(8+i\, (3\pi-2)\right)\, c^2 }{128\pi^2 n^2}+O(n^{-5/2}).
\end{equation}
Finally, inserting \eqref{E--(3.66)} in \eqref{E--(3.58)}, we get \eqref{E--(3.26)}.\\[0.1in]
%%%%%%%%%%%%%%%%%%%%%%Case%%%%%%%%%%%%%%%%%%%%%%%%
\textbf{Case 3.} If  $\sf= 0$,  then
%%%%%%%%%%%%%%%%%%%%Equation%%%%%%%%%%%%%%%%%%%%%%
\begin{equation}\label{E--(3.67)}
\cos\left(\frac{c}{2}\right)=1,\ \sin\left(\frac{c}{2}\right)=0.
\end{equation}
In this case $\lambda_{2,n}$ becomes 
%%%%%%%%%%%%%%%%%%%%Equation%%%%%%%%%%%%%%%%%%%%%%
\begin{equation}\label{E--(3.68)}
\lambda_{2,n}=2i n\pi+i\,\pi +\epsilon_{2,n}.
\end{equation}
Similar to case 2, from \eqref{E--(3.68)} and using the fact that $\lim_{|n|\to+\infty}\epsilon_{2,n}=0$, we have the asymptotic expansion
%%%%%%%%%%%%%%%%%%%%Equation%%%%%%%%%%%%%%%%%%%%%%
\begin{equation}\label{E--(3.69)}
\left\{
\begin{array}{ll}
\displaystyle{\cosh\left(\frac{3\lambda_{2,n}}{2}\right)=-\frac{3i\, (-1)^n\, \epsilon_{2,n}}{2}+O\left(\epsilon_{2,n}^3\right),}\

\displaystyle{

\sinh\left(\frac{3\lambda_{2,n}}{2}\right)=-i\, (-1)^n\, \left(1+\frac{9 \epsilon_{2,n}^2}{8}\right)+O(\epsilon_{2,n}^4),}

\\ \noalign{\medskip}
\displaystyle{\cosh\left(\frac{3\lambda_{2,n}}{2}\right)=\frac{i\, (-1)^n\, \epsilon_{2,n}}{2}+O\left(\epsilon_{2,n}^3\right),}\

\displaystyle{

\sinh\left(\frac{3\lambda_{2,n}}{2}\right)=i\, (-1)^n\, \left(1+\frac{ \epsilon_{2,n}^2}{8}\right)+O(\epsilon_{2,n}^4).}
\end{array}\right.
\end{equation}
Moreover, from \eqref{E--(3.68)}, we get
%%%%%%%%%%%%%%%%%%%%Equation%%%%%%%%%%%%%%%%%%%%%%
\begin{equation}\label{E--(3.70)}
\left\{
\begin{array}{ll}
\displaystyle{\frac{1}{\lambda_{2,n}}= -\frac{i}{2\pi n}+\frac{ i\, \pi}{4\pi^2 n^2}+O\left(\epsilon_{2,n}\, n^{-2}\right)+O\left(n^{-3}\right),\ \frac{1}{\lambda^2_{2,n}}= -\frac{1}{4\pi^2 n^2}+O\left(n^{-3}\right),}

 \\ \noalign{\medskip}
\displaystyle{\frac{1}{\sqrt{\lambda_{2,n}}}=\frac{1-i\sign(n)}{2\sqrt{\pi |n|}} +\frac{ (1+i\,\sign(n))\, \epsilon_{2,n}+ (-\sign(n)+i)\,\pi }{8\sqrt{\pi |n|^3}}}

 \\ \noalign{\medskip}
 
 \hspace{1.5cm} \displaystyle{+\frac{3\, (1-i\sign(n))}{64\sqrt{\pi |n|^5}}+O\left(\epsilon_{2,n}\, |n|^{-5/2}\right)+O\left(|n|^{-7/2}\right)},

 \\ \noalign{\medskip}

\displaystyle{\frac{1}{\sqrt{\lambda^3_{2,n}}}=\frac{-1-i\sign(n)}{4\sqrt{\pi^3 |n|^3}}+\frac{3\, (\sign(n)+i)}{16\sqrt{\pi^3 |n|^5}}+O\left(\epsilon_{2,n}\, |n|^{-5/2}\right)+O\left(|n|^{-7/2}\right)},

 \\ \noalign{\medskip}

\displaystyle{\frac{1}{\sqrt{\lambda^5_{2,n}}}=\frac{-1+i\sign(n)}{8\sqrt{\pi^5 |n|^5}}+O\left(|n|^{-7/2}\right),\ \ \frac{1}{\lambda_{2,n}^3}=O\left(n^{-3}\right)}.

\end{array}\right.
\end{equation}
Inserting \eqref{E--(3.67)},  \eqref{E--(3.69)},  and \eqref{E--(3.70)} in \eqref{E--(3.29)}, we get
%%%%%%%%%%%%%%%%%%%%Equation%%%%%%%%%%%%%%%%%%%%%%
\begin{equation}\label{E--(3.71)}
\begin{array}{ll}

\displaystyle{-i\, \epsilon_{2,n}^2+\left(-\frac{ \sign(n)+i}{\sqrt{\pi|n|}}-\frac{3 c^2}{32\pi n}+\frac{\sign(n)-i+(1+i\sign(n))\, \pi}{4\sqrt{\pi^3|n|^3}}\right)\, \epsilon_{2,n}-\frac{(\sign(n)-i)\, c^2}{32\sqrt{\pi^3|n|^3}}}
\\ \noalign{\medskip}

\displaystyle{+\frac{i\, c^4}{512\pi^2 n^2}-\frac{3\left(3 (\sign(n)+i)-(1-i \sign(n))\, \pi\right) \, c^2}{128\sqrt{\pi^5|n|^5}}}
\\ \noalign{\medskip}

\displaystyle{
+O\left(n^{-3}\right)+O\left(\epsilon_{2,n }\, n^{-2}\right)+O\left(\epsilon_{2,n }^2\, n^{-1}\right)+O\left(\epsilon_{2,n }^3\right)=0}.

\end{array}
\end{equation}
Similar to case 2, solving Equation \eqref{E--(3.71)}, we get
%%%%%%%%%%%%%%%%%%%%Equation%%%%%%%%%%%%%%%%%%%%%%
\begin{equation}\label{E--(3.72)}
\epsilon_{2,n}=\frac{i \, c^2}{32\pi n}-\frac{\left(4+i\pi\right)\, c^2}{64\pi^2 n^2}+O\left(|n|^{-5/2}\right).
\end{equation}
Finally, inserting \eqref{E--(3.72)} in \eqref{E--(3.68)}, we get \eqref{E--(3.28)}. Thus, the proof is complete.	\xqed{$\square$}\\[0.1in]
%%%%%%%%%%%%%%%%%%%%%%%%%%%%%%%%%%%%%%%%%%%%%%%%%%
                % Proof of  Theorem  %
%%%%%%%%%%%%%%%%%%%%%%%%%%%%%%%%%%%%%%%%%%%%%%%%%%
 \noindent\textbf{Proof of  Theorem \ref{Theorem--3.2}.} From Proposition \ref{Proposition--3.3} the operator $\mathcal{A}_2$ has two branches of eigenvalues with eigenvalues admitting real parts tending to zero. Hence, the energy corresponding to the first and second branch of eigenvalues  has no exponential decaying. Therefore the total energy of the  Timoshenko System \eqref{E--(1.1)}-\eqref{E--(1.2)} with  local Kelvin–Voigt damping, and with Dirichlet-Neumann boundary conditions \eqref{E--(1.4)},  has no exponential decaying in the  equal speed case.\xqed{$\square$}
%%%%%%%%%%%%%%%%%%%%%%%%%%%%%%%%%%%%%%%%%%%%%%%%%%
%%%%%%%%%%%%%%%%%%%%%%%%%%%%%%%%%%%%%%%%%%%%%%%%%%
%%%%%%%%%%%%%%%%%%%%%%%%%%%%%%%%%%%%%%%%%%%%%%%%%%
 % Section 4: Polynomial stability %
%%%%%%%%%%%%%%%%%%%%%%%%%%%%%%%%%%%%%%%%%%%%%%%%%%
%%%%%%%%%%%%%%%%%%%%%%%%%%%%%%%%%%%%%%%%%%%%%%%%%%	
%%%%%%%%%%%%%%%%%%%%%%%%%%%%%%%%%%%%%%%%%%%%%%%%%%
\section{Polynomial stability}\label{Section--4}
\noindent  In this section,  we use the  frequency domain approach method to show the polynomial  stability of $\left(e^{t\mathcal{A}_j}\right)_{t\geq0}$ associated with the Timoshenko System  \eqref{E--(2.1)}.  We prove the following theorem.
%%%%%%%%%%%%%%%%%%%%%%%%%%%%%%%%%%%%%%%%%%%%%%%%%%
                % Theorem %
%%%%%%%%%%%%%%%%%%%%%%%%%%%%%%%%%%%%%%%%%%%%%%%%%%
\begin{thm}\label{Theorem--4.1}
\rm{ Under hypothesis {\rm (H)}, for $j=1,2,$  there exists  $C>0$ such that for every $U_0\in D\left(\mathcal{A}_j\right)$, we have
%%%%%%%%%%%%%%%%%%%%Equation%%%%%%%%%%%%%%%%%%%%%%
\begin{equation}\label{E--(4.1)}
	E\left(t\right)\leq \frac{C}{t}\left\|U_0\right\|^2_{D\left(\mathcal{A}_j\right)},\quad\ t>0.
\end{equation}}
\end{thm}
%%%%%%%%%%%%%%%%%%%%%%%%%%%%%%%%%%%%%%%%%%%%%%%%%%
\noindent Since $ \ i\mathbb{R}\subseteq\rho\left(\mathcal{A}_j\right),$ then for the proof of Theorem \ref{Theorem--4.1}, according to Theorem \ref{Theorem--2.5}, we need to prove that
%%%%%%%%%%%%%%%%%%%%Equation%%%%%%%%%%%%%%%%%%%%%%
\begin{equation*}\tag{\rm H3}
\sup_{\lambda\in\mathbb{R}}\left\|\left(i\lambda I-\mathcal{A}_j\right)^{-1}\right\|_{\mathcal{L}\left(\mathcal{H}_j\right)}=O\left(\lambda^{2}\right).
\end{equation*}
We will argue by contradiction. Therefore suppose  there exists $\left\{(\lambda_n,U_n=\left(u_n,v_n,y_n,z_n\right))\right\}_{n\geq 1}\subset \mathbb{R}\times D\left(\mathcal{A}_j\right)$, with $\lambda_n>1$  and 
%%%%%%%%%%%%%%%%%%%%Equation%%%%%%%%%%%%%%%%%%%%%%
\begin{equation}\label{EQ--(4.2)}
\lambda_n\to+\infty,\ \ \|U_n\|_{\mathcal{H}_j}=1,
\end{equation}
such that
%%%%%%%%%%%%%%%%%%%%Equation%%%%%%%%%%%%%%%%%%%%%%
\begin{equation}\label{EQ--(4.3)}
\lambda_n^{2}\left(\ i\lambda_n U_n-\mathcal{A}_jU_n\right)=\left(f_{1,n},f_{2,n},f_{3,n},f_{4,n}\right)\to 0\ \text{ in } \mathcal{H}_j.
\end{equation}
Equivalently, we have 
%%%%%%%%%%%%%%%%%%%%Equation%%%%%%%%%%%%%%%%%%%%%%
\begin{eqnarray}
i{\lambda_n}u_n-  {v_n}&=&\lambda_n^{-2}f_{1,n}\to0\text{ in } H^1_0(0,L),\label{EQ--(4.4)}
\\ \noalign{\medskip}
 i \lambda_n v_n-\frac{k_1}{\rho_1 }((u_n)_x+y_n)_x
&=&\lambda_n^{-2}f_{2,n}\to0\text{ in } L^2(0,L),\label{EQ--(4.5)}
\\ \noalign{\medskip}
 i \lambda_ny_n-{z_n}&=&\lambda_n^{-2}f_{3,n}\to0\text{ in } \mathcal{W}_j(0,L) ,\label{EQ--(4.6)}
\\ \noalign{\medskip}
i{\lambda_n}  z_n- \frac{k_2}{\rho_2}\left((y_n)_{x}+\frac{D}{k_2}(z_n)_x\right)_x+\frac{k_1}{\rho_2 }((u_n)_x+y_n)&=&\lambda_n^{-2}f_{4,n}\to0\text{ in } L^2(0,L),\label{EQ--(4.7)}
\end{eqnarray}
where 
%%%%%%%%%%%%%%%%%%%%Equation%%%%%%%%%%%%%%%%%%%%%%
\begin{equation*}
\mathcal{W}_j(0,L)=\left\{ \begin{array}{ll} \displaystyle{H^1_0(0,L),\quad \text{if }j=1}, \nline \displaystyle{H^1_*(0,L) ,\quad \text{if }j=2}.\end{array}\right.
\end{equation*}
In the following, we will check the condition {\rm (H3)} by finding a contradiction with \eqref{EQ--(4.2)} such as $\left\| U_n\right\|_{\mathcal{H}_j} =o(1).$ For clarity, we divide the proof into several lemmas. From  now on, for simplicity, we drop the index $n$. Since $U$ is uniformly bounded in ${\mathcal{H}},$ we get from \eqref{EQ--(4.4)} and  \eqref{EQ--(4.6)} respectively that
%%%%%%%%%%%%%%%%%%%%Equation%%%%%%%%%%%%%%%%%%%%%%
\begin{equation}\label{EQ--(4.8)}
\int_{0}^{L}|u|^2\, dx=O\left(\lambda^{-2}\right)\ \ \ \text{and}\ \ \   \int_{0}^{L}|y|^2\, dx=O\left(\lambda^{-2}\right),\ 
\end{equation}
%%%%%%%%%%%%%%%%%%%%%%%%%%%%%%%%%%%%%%%%%%%%%%%%%%
                % Lemma %
%%%%%%%%%%%%%%%%%%%%%%%%%%%%%%%%%%%%%%%%%%%%%%%%%%
\begin{lem}\label{Lemma-4.5}
\rm{Under hypothesis {\rm (H)}, for $j=1,2,$   we have
%%%%%%%%%%%%%%%%%%%%Equation%%%%%%%%%%%%%%%%%%%%%%
\begin{eqnarray}
&\displaystyle{\int_{0}^LD(x)\left|z_{x}\right|^2 dx=o\left(\lambda^{-2}\right), \ \int_{\alpha}^\beta\left|z_x\right|^2 dx=o\left(\lambda^{-2}\right)\label{EQ--(4.9)},}\nline
&\displaystyle{\int_{\alpha}^\beta\left|y_{x}\right|^2 dx=o\left(\lambda^{-4}\right).}\label{EQ--(4.10)}
\end{eqnarray}}
\end{lem}
%%%%%%%%%%%%%%%%%%%%%%%%%%%%%%%%%%%%%%%%%%%%%%%%%%
                % Proof of  Lemma  %
%%%%%%%%%%%%%%%%%%%%%%%%%%%%%%%%%%%%%%%%%%%%%%%%%%
\begin{proof} First, taking the inner product of \eqref{EQ--(4.3)} with $U$ in $\mathcal{H}_j$, then using the fact that $U$ is uniformly bounded in $\mathcal{H}_j$, we get
%%%%%%%%%%%%%%%%%%%%Equation%%%%%%%%%%%%%%%%%%%%%%
\begin{equation*}
	\int_{0}^L D(x)\left|z_{x}\right|^2 dx=-\lambda^{-2}\Re\left(\left<\lambda^{2}\mathcal{A}_jU,U\right>_{\mathcal{H}_j}\right)=\lambda^{-2}\Re\left(\left<\lambda^{2}\left( i\lambda U-\mathcal{A}_jU\right),U\right>_{\mathcal{H}_j}\right)=o\left(\lambda^{-2}\right),
\end{equation*}
hence,  we get the first asymptotic estimate of  \eqref{EQ--(4.9)}. Next, using  hypothesis {\rm (H)} and the first asymptotic estimate of  \eqref{EQ--(4.9)}, we get the second  asymptotic estimate of  \eqref{EQ--(4.9)}. Finally, from  \eqref{EQ--(4.3)}, \eqref{EQ--(4.6)}, and \eqref{EQ--(4.9)}, we get the  asymptotic estimate of \eqref{EQ--(4.10)}. 
\end{proof}$\\[0.1in]$
%%%%%%%%%%%%%%%%%%%%%%%%%%%%%%%%%%%%%%%%%%%%%%%%%%
 Let $g \in C^1\left([\alpha,\beta]\right)$ such that 
%%%%%%%%%%%%%%%%%%%%Equation%%%%%%%%%%%%%%%%%%%%%%
\begin{equation*}
g(\beta)=-g(\alpha) = 1,\quad \max \limits_ {x \in [\alpha, \beta]} |g(x)|= c_g\quad \text{and}\quad \max \limits_ {x \in [\alpha, \beta]} |g'(x)| = c_{g'},
\end{equation*}
where $c_g$ and $c_{g'}$ are strictly positive constant numbers.
%%%%%%%%%%%%%%%%%%%%%%%%%%%%%%%%%%%%%%%%%%%%%%%%%%
                % Remark %
%%%%%%%%%%%%%%%%%%%%%%%%%%%%%%%%%%%%%%%%%%%%%%%%%%
\begin{rk}\label{Remark-4.3}
\rm{It is easy to see the existence of $g(x)$. For example, we can take $g(x)=\cos\left(\frac{ (\beta-x)\pi}{\beta-\alpha}\right)$ to get $g(\beta)=-g(\alpha)=1$, $g\in   C^1([\alpha,\beta])$, $|g(x)| \leq 1$ and  $|g'(x)|\leq \frac{\pi}{\beta-\alpha}$.  Also, we can take  
$$g(x)={x}^{2}- \left(\beta+ \alpha-2\, \left( \beta-\alpha \right) ^{-1}
 \right) x+\alpha\,\beta-\left(\beta+\alpha\right) \left(\beta-\alpha\right)^{-1}.$$
\xqed{$\square$}}
\end{rk}
%%%%%%%%%%%%%%%%%%%%%%%%%%%%%%%%%%%%%%%%%%%%%%%%%%
                % Lemma %
%%%%%%%%%%%%%%%%%%%%%%%%%%%%%%%%%%%%%%%%%%%%%%%%%%
\begin{lem}\label{Lemma-4.4}
\rm{Under hypothesis {\rm (H)}, for $j=1,2,$   we have
%%%%%%%%%%%%%%%%%%%%Equation%%%%%%%%%%%%%%%%%%%%%%
\begin{eqnarray}
 |z(\beta)|^2+|z(\alpha)|^2 \leq\left(\frac{\rho_2\lambda^{\frac{1}{2}}}{2k_2}+2 \, c_{g'} \right)
 \int_\alpha^\beta  |z|^2 \, dx+o\left(\lambda^{-\frac{5}{2}}\right), \label{EQ--(4.11)}\nline
 \left|\left({y}_x+\frac{D(x)}{k_2}{z}_x\right)(\alpha)\right|^2+ \left|\left({y}_x+\frac{D(x)}{k_2}{z}_x\right)(\beta)\right|^2\leq \frac{\rho_2 \lambda^{\frac{3}{2}}}{2k_2}  \int_\alpha^\beta  |z|^2 dx +o\left(\lambda^{-1}\right).\label{EQ--(4.12)}
\end{eqnarray}}
\end{lem}
%%%%%%%%%%%%%%%%%%%%%%%%%%%%%%%%%%%%%%%%%%%%%%%%%%
                % Proof of Lemma %
%%%%%%%%%%%%%%%%%%%%%%%%%%%%%%%%%%%%%%%%%%%%%%%%%%
\begin{proof}
The proof is divided into two  steps.\\[0.1in]
%%%%%%%%%%%%%%%%%%%%%%Step%%%%%%%%%%%%%%%%%%%%%%%%
\textbf{Step 1.} In this step, we prove the  asymptotic behavior estimate of \eqref{EQ--(4.11)}. For this aim, 
first, from \eqref{EQ--(4.6)}, we have
%%%%%%%%%%%%%%%%%%%%Equation%%%%%%%%%%%%%%%%%%%%%%
\begin{equation}\label{EQ--(4.13)}
 z_x  = i \lambda y_x- \lambda^{-2}\, (f_3)_x \quad \textrm{in} \quad L^2(\alpha,\beta).
\end{equation}
Multiplying \eqref{EQ--(4.13)} by $2\, g \overline{z}$ and integrating over $(\alpha,\beta),$ then taking the real part, we get
%%%%%%%%%%%%%%%%%%%%Equation%%%%%%%%%%%%%%%%%%%%%%
\begin{equation*}
\int_\alpha^\beta g(x)\, (|z|^2)_x \, dx = \Re\left\{2i \lambda \int_\alpha^\beta g(x)\,  y_x\overline{z}dx\right\}- \Re\left\{2\lambda^{-2}\,\int_\alpha^\beta  g(x)\, (f_4)_x\overline{z} dx\right\},
\end{equation*}
using by parts integration in the left hand side of above equation, we get
%%%%%%%%%%%%%%%%%%%%Equation%%%%%%%%%%%%%%%%%%%%%%
\begin{equation*}
\left[ g(x)\, |z|^2\right]^{\beta}_{\alpha} =\int_\alpha^\beta g'(x)\, |z|^2 \, dx+ \Re\left\{2i \lambda \int_\alpha^\beta g(x)\,  y_x\overline{z}dx\right\}- \Re\left\{2\lambda^{-2}\,\int_\alpha^\beta  g(x)\, (f_3)_x\overline{z} dx\right\},
\end{equation*}
consequently,
%%%%%%%%%%%%%%%%%%%%Equation%%%%%%%%%%%%%%%%%%%%%%
\begin{equation}\label{EQ--(4.14)}
 |z(\beta)|^2+|z(\alpha)|^2 \leq c_{g'}\int_\alpha^\beta  |z|^2 \, dx+ 2 \lambda \, c_g \int_\alpha^\beta   |y_x|\left|{z}\right|dx+2\lambda^{-2} \, c_g\,\int_\alpha^\beta   \left|(f_3)_x\right|\left|{z}\right|\, dx.
\end{equation}
On the other hand, we have 
%%%%%%%%%%%%%%%%%%%%Equation%%%%%%%%%%%%%%%%%%%%%%
\begin{equation*}
2 \lambda \, c_g |y_x| |z|\leq \frac{\rho_2\lambda^{\frac{1}{2}}|z|^2}{2k_2}+\frac{2k_2\lambda^{\frac{3}{2}} \, c_g^2}{\rho_2}|y_x|^2\ \ \ \text{and} \ \ \ 2 \lambda^{-2} |(f_3)_x| |z|\leq  c_{g'} \, |z|^2+\frac{c_g ^ 2 \, \lambda^{-4} }{c_{g'}}|(f_3)_x|^2.
\end{equation*}
Inserting the above equation in \eqref{EQ--(4.14)}, then using \eqref{EQ--(4.10)} and the fact that $(f_3)_x\to 0$ in $L^2(\alpha,\beta)$, we get
%%%%%%%%%%%%%%%%%%%%Equation%%%%%%%%%%%%%%%%%%%%%%
\begin{equation*}
 |z(\beta)|^2+|z(\alpha)|^2 \leq\left(\frac{\rho_2\lambda^{\frac{1}{2}}}{2k_2}+2 \, c_{g'} \right)
 \int_\alpha^\beta  |z|^2 \, dx+o\left(\lambda^{-\frac{5}{2}}\right),
\end{equation*}
hence, we get \eqref{EQ--(4.11)}.\\[0.1in]
%%%%%%%%%%%%%%%%%%%%%%Step%%%%%%%%%%%%%%%%%%%%%%%%
\textbf{Step 2.} In this step, we prove the following asymptotic behavior estimate of \eqref{EQ--(4.12)}.
For this aim, first, multiplying \eqref{EQ--(4.7)} by $-\frac{2\rho_2}{k_2}\, g \left( \overline{y}_x+\frac{D(x)}{k_2} \overline{z}_x\right)$ and integrating over $(\alpha,\beta),$ then taking the real part, we get
%%%%%%%%%%%%%%%%%%%%Equation%%%%%%%%%%%%%%%%%%%%%%
\begin{equation*}
\begin{array}{ll}
\displaystyle{
\int_\alpha^\beta g(x)\left(\left|{y}_x+\frac{D(x)}{k_2}{z}_x\right|^2\right)_x\, dx  =\frac{2\rho_2 \lambda}{k_2}\Re\left\{i  \int_\alpha^\beta g(x)  z\left( \overline{y}_x+\frac{D(x)}{k_2} \overline{z}_x\right)\, dx\right\} }\nline \hspace{1cm} \displaystyle{ +\frac{2k_1}{k_2 }\Re\left\{ \int_\alpha^\beta g(x) \left(u_x+y\right)\left( \overline{y}_x+\frac{D(x)}{k_2} \overline{z}_x\right) dx\right\}
-\frac{2\rho_2 \lambda^{-2}}{k_2}\,\Re\left\{ \int_\alpha^\beta g(x) f_4\left( \overline{y}_x+\frac{D(x)}{k_2} \overline{z}_x\right) dx\right\},}
\end{array}
\end{equation*}
using by parts integration in the left hand side of above equation, we get
%%%%%%%%%%%%%%%%%%%%Equation%%%%%%%%%%%%%%%%%%%%%%
\begin{equation*}
\begin{array}{ll}
\displaystyle{
\left[ g(x)\left|{y}_x+\frac{D(x)}{k_2}{z}_x\right|^2\right]_\alpha^\beta   =\int_\alpha^\beta g'(x)\left|{y}_x+\frac{D(x)}{k_2}{z}_x\right|^2\, dx +\frac{2\rho_2 \lambda}{k_2}\Re\left\{i  \int_\alpha^\beta g(x)  z\left( \overline{y}_x+\frac{D(x)}{k_2} \overline{z}_x\right)\, dx\right\} }\nline \hspace{1cm} \displaystyle{ +\frac{2k_1}{k_2 }\Re\left\{ \int_\alpha^\beta g(x) \left(u_x+y\right)\left( \overline{y}_x+\frac{D(x)}{k_2} \overline{z}_x\right) dx\right\}
-\frac{2\rho_2 \lambda^{-2}}{k_2}\,\Re\left\{ \int_\alpha^\beta g(x) f_4\left( \overline{y}_x+\frac{D(x)}{k_2} \overline{z}_x\right) dx\right\},}
\end{array}
\end{equation*}
consequently,
%%%%%%%%%%%%%%%%%%%%Equation%%%%%%%%%%%%%%%%%%%%%%
\begin{equation*}
\begin{array}{ll}

\displaystyle{
\left|\left({y}_x+\frac{D(x)}{k_2}{z}_x\right)(\alpha)\right|^2+ \left|\left({y}_x+\frac{D(x)}{k_2}{z}_x\right)(\beta)\right|^2    \leq 

\frac{2\rho_2\, c_g \lambda}{k_2}  \int_\alpha^\beta   \left|z\right|\left| {y}_x+\frac{D(x)}{k_2} {z}_x\right|\, dx }\nline

 \displaystyle{c_{g'}\int_\alpha^\beta \left|{y}_x+\frac{D(x)}{k_2}{z}_x\right|^2\, dx  +\frac{2k_1\, c_g}{k_2 } \int_\alpha^\beta  \left|u_x+y\right|\left| {y}_x+\frac{D(x)}{k_2} {z}_x\right| dx+\frac{2\rho_2\, c_g \lambda^{-2}}{k_2}\int_\alpha^\beta  \left|f_4\right|\left| {y}_x+\frac{D(x)}{k_2} {z}_x\right| dx.}
\end{array}
\end{equation*}
Now, using Cauchy Schwarz inequality, Equations \eqref{EQ--(4.9)}-\eqref{EQ--(4.10)}, the fact that $f_5\to 0$ in $L^2(\alpha,\beta)$ and the fact that $u_x+y$ is uniformly bounded in $L^2(\alpha,\beta)$ in the right hand side of above equation, we get 
%%%%%%%%%%%%%%%%%%%%Equation%%%%%%%%%%%%%%%%%%%%%%
\begin{equation}\label{EQ--(4.15)}
\left|\left({y}_x+\frac{D(x)}{k_2}{z}_x\right)(\alpha)\right|^2+ \left|\left({y}_x+\frac{D(x)}{k_2}{z}_x\right)(\beta)\right|^2    \leq 
\frac{2\rho_2\, c_g \lambda}{k_2}  \int_\alpha^\beta   \left|z\right|\left| {y}_x+\frac{D(x)}{k_2} {z}_x\right|\, dx +o\left(\lambda^{-1}\right) .
\end{equation}
On the other hand, we have
%%%%%%%%%%%%%%%%%%%%Equation%%%%%%%%%%%%%%%%%%%%%%
\begin{equation*}
\frac{2\rho_2\, c_g \lambda}{k_2} \, |z|\left|{y}_x+\frac{D(x)}{k_2} {z}_x\right|\leq \frac{\rho_2\lambda^{\frac{3}{2}}}{2k_2}|z|^2+\frac{2\rho_2\lambda^{\frac{1}{2}}\,  c_g^2}{k_2}    \left|{y}_x+\frac{D(x)}{k_2} {z}_x\right|^2.
\end{equation*}
Inserting the above equation in \eqref{EQ--(4.15)}, then using Equations \eqref{EQ--(4.9)}-\eqref{EQ--(4.10)}, we get
%%%%%%%%%%%%%%%%%%%%Equation%%%%%%%%%%%%%%%%%%%%%%
\begin{equation*}
\left|\left({y}_x+\frac{D(x)}{k_2}{z}_x\right)(\alpha)\right|^2+ \left|\left({y}_x+\frac{D(x)}{k_2}{z}_x\right)(\beta)\right|^2\leq \frac{\rho_2 \lambda^{\frac{3}{2}}}{2k_2}  \int_\alpha^\beta  |z|^2 dx +o\left(\lambda^{-1}\right) ,
\end{equation*}
hence, we get \eqref{EQ--(4.12)}.  Thus, the proof is complete.
\end{proof}
%%%%%%%%%%%%%%%%%%%%%%%%%%%%%%%%%%%%%%%%%%%%%%%%%%
                % Lemma %
%%%%%%%%%%%%%%%%%%%%%%%%%%%%%%%%%%%%%%%%%%%%%%%%%%
\begin{lem}\label{Lemma-4.5}
\rm{Under hypothesis {\rm (H)}, for $j=1,2,$   we have
%%%%%%%%%%%%%%%%%%%%Equation%%%%%%%%%%%%%%%%%%%%%%
\begin{eqnarray}
|u_x(\alpha)+y(\alpha)|^2=O\left(1\right),\ |u_x(\beta)+y(\beta)|^2=O\left(1\right).\label{EQ--(4.16)}\nline
|u(\alpha)|^2=O\left(\lambda^{-2}\right),\ |u(\beta)|^2=O\left(\lambda^{-2}\right),\label{EQ--(4.17)}\nline
|v(\alpha)|^2=O\left(1\right),\ |v(\beta)|^2=O\left(1\right).\label{EQ--(4.18)}
\end{eqnarray}}
\end{lem}
%%%%%%%%%%%%%%%%%%%%%%%%%%%%%%%%%%%%%%%%%%%%%%%%%%
                % Proof of  Lemma  %
%%%%%%%%%%%%%%%%%%%%%%%%%%%%%%%%%%%%%%%%%%%%%%%%%%
\begin{proof} Multiplying Equation \eqref{EQ--(4.5)} by $-\frac{2\rho_1}{k_1}g\left( \overline{u}_{x}+\overline{y}\right)$  and integrating over $(\alpha,\beta),$ then  taking the real part and using the fact that $u_x+y$ is uniformly bounded in $L^2(\alpha,\beta)$, $f_2\to0 $ in $L^2(\alpha,\beta)$, we get
%%%%%%%%%%%%%%%%%%%%Equation%%%%%%%%%%%%%%%%%%%%%%
\begin{equation}\label{EQ--(4.19)}
\int_{\alpha}^{\beta}g(x)\left(\left|u_x+y\right|^2\right)_x\, dx-\frac{2 \rho_1\lambda }{k_1} \Re\left\{i\int_{\alpha}^{\beta} g(x)\overline{u}_{x}\, v\, dx  \right\}=\frac{2 \rho_1\lambda }{k_1} \Re\left\{i\int_{\alpha}^{\beta} g(x) \overline{y}\, v\, dx  \right\}+o\left(\lambda^{-2}\right).
\end{equation}
 Now, we divided the proof into two steps.\\[0.1in]
%%%%%%%%%%%%%%%%%%%%%%Step%%%%%%%%%%%%%%%%%%%%%%%%
\textbf{Step 1.} In this step, we prove the   asymptotic behavior estimates of  \eqref{EQ--(4.16)}-\eqref{EQ--(4.17)}. First,  from  \eqref{EQ--(4.4)}, we have
%%%%%%%%%%%%%%%%%%%%Equation%%%%%%%%%%%%%%%%%%%%%%
\begin{equation*}
-i\lambda\, v=\lambda^2 u+  i\lambda^{-1}f_{1}.
\end{equation*}
Inserting the above equation in the second term in left of  \eqref{EQ--(4.19)}, then using the fact that $u_x$ is uniformly bounded in $L^2(\alpha,\beta)$ and $f_1\to0 $ in $L^2(\alpha,\beta)$, we get
%%%%%%%%%%%%%%%%%%%%Equation%%%%%%%%%%%%%%%%%%%%%%
\begin{equation*}
\int_{\alpha}^{\beta}g(x)\left(\left|u_x+y\right|^2\right)_x\, dx+\frac{ \rho_1 \lambda^2}{k_1} \int_{\alpha}^{\beta} g(x) \left(\left|u\right|^2\right)_x dx =-\frac{2 \rho_1 \lambda^2}{k_1} \Re\left\{\int_{\alpha}^{\beta} g(x)\, u\, \overline{y}  dx  \right\}+o\left(\lambda^{-1}\right).
\end{equation*}
Using  by parts integration and the fact that  $g(\beta)=-g(\alpha) = 1$ in the above equation, we get 
%%%%%%%%%%%%%%%%%%%%Equation%%%%%%%%%%%%%%%%%%%%%%
\begin{equation*}
\begin{array}{ll}

\displaystyle{
\left|u_x(\beta)+y(\beta)\right|^2+\frac{ \rho_1 \lambda^2}{k_1} \left|u(\beta)\right|^2 
+\left|u_x(\alpha)+y(\alpha)\right|^2+\frac{ \rho_1 \lambda^2}{k_1} \left|u(\alpha)\right|^2 
=\int_{\alpha}^{\beta}g'(x)\left|u_x+y\right|^2 dx}\nline

 \displaystyle{

+\frac{ \rho_1 \lambda^2}{k_1} \int_{\alpha}^{\beta} g'(x) \left|u\right|^2 dx

-\frac{2 \rho_1 \lambda^2}{k_1} \Re\left\{\int_{\alpha}^{\beta} g(x)\, u\, \overline{y}  dx  \right\}+o\left(\lambda^{-1}\right),
}

\end{array}
\end{equation*}
consequently, 
%%%%%%%%%%%%%%%%%%%%Equation%%%%%%%%%%%%%%%%%%%%%%
\begin{equation*}
\begin{array}{ll}

\displaystyle{
\left|u_x(\beta)+y(\beta)\right|^2+\frac{ \rho_1 \lambda^2}{k_1} \left|u(\beta)\right|^2 
+\left|u_x(\alpha)+y(\alpha)\right|^2+\frac{ \rho_1 \lambda^2}{k_1} \left|u(\alpha)\right|^2 
\leq c_{g'} \int_{\alpha}^{\beta}\left|u_x+y\right|^2 dx}\nline

 \displaystyle{

+\frac{ \rho_1\, c_{g'} \lambda^2}{k_1} \int_{\alpha}^{\beta}  \left|u\right|^2 dx

+\frac{2 \rho_1\, c_g \lambda^2}{k_1} \int_{\alpha}^{\beta}  \left|u\right| \left|{y}\right|  dx  +o\left(\lambda^{-1}\right).
}

\end{array}
\end{equation*}
Next, since $\lambda\, u,\ \lambda\, y$ and $u_x+y$ are uniformly bounded, then from the above equation, we get \eqref{EQ--(4.16)}-\eqref{EQ--(4.17)}.\\[0.1in]
%%%%%%%%%%%%%%%%%%%%%%Step%%%%%%%%%%%%%%%%%%%%%%%%
\textbf{Step 2.} In this step, we prove the   asymptotic behavior estimates of \eqref{EQ--(4.18)}. First, from  \eqref{EQ--(4.4)}, we have
%%%%%%%%%%%%%%%%%%%%Equation%%%%%%%%%%%%%%%%%%%%%%
\begin{equation*}
-i\lambda \overline{u}_x= \overline{v}_x-\lambda^{-2}(\overline{f}_{1})_x.
\end{equation*}
Inserting the above equation in the second term in left of  \eqref{EQ--(4.19)}, then using the fact that $v$ is uniformly bounded in $L^2(\alpha,\beta)$ and $(f_1)_x\to0 $ in $L^2(\alpha,\beta)$, we get
%%%%%%%%%%%%%%%%%%%%Equation%%%%%%%%%%%%%%%%%%%%%%
\begin{equation*}
\int_{\alpha}^{\beta}g(x)\left(\left|u_x+y\right|^2\right)_x\, dx+\frac{ \rho_1 }{k_1} \int_{\alpha}^{\beta} g(x) \left(\left|v\right|^2\right)_x dx =-\frac{2 \rho_1 \lambda^2}{k_1} \Re\left\{\int_{\alpha}^{\beta} g(x)\, u\, \overline{y}  dx  \right\}+o\left(\lambda^{-1}\right).
\end{equation*}
Similar  to step 1, by using by parts integration and the fact that  $g(\beta)=-g(\alpha) = 1$ in the above equation, then using the fact that $v,\ \lambda\, u,\ \lambda\, y$ and $u_x+y$ are uniformly bounded in $L^2(\alpha,\beta)$,  we get \eqref{EQ--(4.18)}. Thus, the proof is complete.
\end{proof}
%%%%%%%%%%%%%%%%%%%%%%%%%%%%%%%%%%%%%%%%%%%%%%%%%%
                % Lemma %
%%%%%%%%%%%%%%%%%%%%%%%%%%%%%%%%%%%%%%%%%%%%%%%%%%
\begin{lem}\label{Lemma-4.6}
\rm{Under hypothesis {\rm (H)}, for $j=1,2,$   for $\lambda$ large enough,  we have
%%%%%%%%%%%%%%%%%%%%Equation%%%%%%%%%%%%%%%%%%%%%%
\begin{eqnarray}
\int_\alpha^\beta  |z|^2 \, dx= o\left(\lambda^{-\frac{5}{2}}\right),\qquad \int_\alpha^\beta  |y|^2 \, dx= o\left(\lambda^{-\frac{9}{2}}\right),\label{EQ--(4.20)}\nline
 \left|\left({y}_x+\frac{D(x)}{k_2}{z}_x\right)(\alpha)\right|^2=o\left(\lambda^{-1}\right),\qquad  \left|\left({y}_x+\frac{D(x)}{k_2}{z}_x\right)(\beta)\right|^2=o\left(\lambda^{-1}\right).\label{EQ--(4.21)}
\end{eqnarray}}
\end{lem}
%%%%%%%%%%%%%%%%%%%%%%%%%%%%%%%%%%%%%%%%%%%%%%%%%%
                % Proof of  Lemma  %
%%%%%%%%%%%%%%%%%%%%%%%%%%%%%%%%%%%%%%%%%%%%%%%%%%
\begin{proof} The proof is divided into two  steps.\\[0.1in]
%%%%%%%%%%%%%%%%%%%%%%Step%%%%%%%%%%%%%%%%%%%%%%%%
\textbf{Step 1.} In this step, we prove the  following asymptotic behavior estimate 
%%%%%%%%%%%%%%%%%%%%Equation%%%%%%%%%%%%%%%%%%%%%%
\begin{equation}\label{EQ--(4.22)}
\left|\frac{i k_1}{\rho_2 \lambda }   \int_\alpha^\beta \left(u_x+y\right) \overline{z} \, dx\right|\leq \left(\frac{1}{4}+\frac{k_2 c_{g'}}{\rho_2\lambda^{\frac{1}{2}}} +\frac{ k_1}{\rho_2 \lambda^2 }\right)\int_\alpha^\beta  |z|^2 \, dx+o\left(\lambda^{-3}\right).
\end{equation}
For this aim, first, we have
%%%%%%%%%%%%%%%%%%%%Equation%%%%%%%%%%%%%%%%%%%%%%
\begin{equation}\label{EQ--(4.23)}
\left|\frac{i k_1}{\rho_2 \lambda }   \int_\alpha^\beta \left(u_x+y\right) \overline{z} \, dx\right|\leq  \left|\frac{i k_1}{\rho_2 \lambda }   \int_\alpha^\beta y\overline{z} \, dx\right|+\left|\frac{i k_1}{\rho_2 \lambda }  \int_\alpha^\beta u_x\overline{z} \, dx\right|.
\end{equation}
Now, from \eqref{EQ--(4.6)} and using the fact that $f_3\to0$ in $L^2(\alpha,\beta)$ and $z$ is uniformly bounded in $L^2(\alpha,\beta)$, we get
%%%%%%%%%%%%%%%%%%%%Equation%%%%%%%%%%%%%%%%%%%%%%
\begin{equation}\label{EQ--(4.24)}
\left|\frac{i k_1}{\rho_2 \lambda }   \int_\alpha^\beta y\overline{z} \, dx\right|\leq\frac{ k_1}{\rho_2 \lambda^2 }\int_\alpha^\beta |{z}|^2 dx+o\left(\lambda^{-4}\right).
\end{equation}
Next, by using by parts integration, we get
%%%%%%%%%%%%%%%%%%%%Equation%%%%%%%%%%%%%%%%%%%%%%
\begin{equation*}
\left|\frac{i k_1}{\rho_2 \lambda }  \int_\alpha^\beta u_x\overline{z} \, dx\right|=\left|-\frac{i k_1}{\rho_2 \lambda }\int_\alpha^\beta u \overline{z}_x \, dx+\frac{i k_1}{\rho_2 \lambda } u(\beta) \overline{z}(\beta)-\frac{i k_1}{\rho_2 \lambda } u(\alpha) \overline{z}(\alpha)\right|,
\end{equation*}
consequently, 
%%%%%%%%%%%%%%%%%%%%Equation%%%%%%%%%%%%%%%%%%%%%%
\begin{equation}\label{EQ--(4.25)}
\left|\frac{i k_1}{\rho_2 \lambda }  \int_\alpha^\beta u_x\overline{z} \, dx\right|\leq \frac{ k_1}{\rho_2 \lambda }\int_\alpha^\beta \left|u \right|\left|{z}_x\right| dx+\frac{k_1}{\rho_2 \lambda }\left( \left|u(\beta)\right| \left| {z}(\beta)\right|+\left|u(\alpha)\right| \left| {z}(\alpha)\right|\right),
\end{equation}
On the other hand, we have
%%%%%%%%%%%%%%%%%%%%Equation%%%%%%%%%%%%%%%%%%%%%%
\begin{equation*}
\frac{k_1}{\rho_2 \lambda }\left( \left|u(\beta)\right| \left| {z}(\beta)\right|+\left|u(\alpha)\right| \left| {z}(\alpha)\right|\right)\leq \frac{k_1^2}{2k_2\rho_2\lambda^{\frac{3}{2}}}\left(\left|u(\alpha)\right| ^2+\left|u(\beta)\right| ^2 \right)+\frac{k_2}{2\rho_2\lambda^{\frac{1}{2}}}\left(\left|z(\alpha)\right| ^2+\left|z(\beta)\right| ^2 \right).
\end{equation*}
Inserting \eqref{EQ--(4.11)} and \eqref{EQ--(4.17)} in the above equation, we get
%%%%%%%%%%%%%%%%%%%%Equation%%%%%%%%%%%%%%%%%%%%%%
\begin{equation*}
\frac{k_1}{\rho_2 \lambda }\left( \left|u(\beta)\right| \left| {z}(\beta)\right|+\left|u(\alpha)\right| \left| {z}(\alpha)\right|\right)\leq \left(\frac{1}{4}+\frac{k_2 c_{g'}}{\rho_2\lambda^{\frac{1}{2}}} \right)
 \int_\alpha^\beta  |z|^2 \, dx+o\left(\lambda^{-3}\right).
\end{equation*}
Inserting the above equation in \eqref{EQ--(4.25)}, then using \eqref{EQ--(4.9)} and the fact that $\lambda u$ is bounded in $L^2(\alpha,\beta)$, we get
%%%%%%%%%%%%%%%%%%%%Equation%%%%%%%%%%%%%%%%%%%%%%
\begin{equation*}
\left|\frac{i k_1}{\rho_2 \lambda }  \int_\alpha^\beta u_x\overline{z} \, dx\right|\leq  \left(\frac{1}{4}+\frac{k_2 c_{g'}}{\rho_2\lambda^{\frac{1}{2}}} \right)\int_\alpha^\beta  |z|^2 \, dx+o\left(\lambda^{-3}\right).
\end{equation*}
Finally, inserting the above equation and Equation \eqref{EQ--(4.24)}  in \eqref{EQ--(4.23)}, we get \eqref{EQ--(4.22)}.\\[0.1in] 
%%%%%%%%%%%%%%%%%%%%%%Step%%%%%%%%%%%%%%%%%%%%%%%%
\textbf{Step 2.} In this step, we prove the   asymptotic behavior estimates of \eqref{EQ--(4.20)}-\eqref{EQ--(4.21)}. For this aim,   first, multiplying  \eqref{EQ--(4.7)} by $-i  \lambda^{-1}\rho_2^{-1} \overline{z}$ and integrating over $(\alpha,\beta),$ then taking the real part, we get
%%%%%%%%%%%%%%%%%%%%Equation%%%%%%%%%%%%%%%%%%%%%%
\begin{equation*}
\int_\alpha^\beta | z|^2\, dx  =-\frac{k_2 }{\rho_2\lambda} \Re\left\{i \int_\alpha^\beta\left(y_{x}+\frac{D}{k_2}z_x\right)_x\,\overline{z} \, dx\right\}+\frac{ k_1}{\rho_2 \lambda }\Re\left\{ i   \int_\alpha^\beta \left(u_x+y\right) \overline{z} \, dx\right\}
- \lambda^{-3}\Re\left\{ i   \int_\alpha^\beta f_4\overline{z} \, dx\right\},
\end{equation*}
consequently,
%%%%%%%%%%%%%%%%%%%%Equation%%%%%%%%%%%%%%%%%%%%%%
\begin{equation}\label{EQ--(4.26)}
\int_\alpha^\beta | z|^2\, dx  \leq  \frac{ k_2}{\rho_2\lambda}\left|\int_\alpha^\beta\left(y_{x}+\frac{D}{k_2}z_x\right)_x\,\overline{z}  \, dx\right|+\left|\frac{i k_1}{\rho_2\lambda }\int_\alpha^\beta\left(u_x+y\right)\overline{z} \, dx\right|
+    \lambda^{-3}\int_\alpha^\beta \left|f_4\right|\left|z \right| dx.
\end{equation}
From the fact that $z$  is uniformly bounded in   $L^2(\alpha,\beta)$ and $f_5\to 0$ in $L^2(\alpha,\beta)$, we get
%%%%%%%%%%%%%%%%%%%%Equation%%%%%%%%%%%%%%%%%%%%%%
\begin{equation}\label{EQ--(4.27)}
 \lambda^{-3}\int_\alpha^\beta \left|f_4\right|\left|z \right| dx=o\left(\lambda^{-3}\right).
\end{equation}
Inserting \eqref{EQ--(4.22)} and \eqref{EQ--(4.27)} in \eqref{EQ--(4.26)}, we get
\begin{equation}\label{EQ--(4.28)}
\int_\alpha^\beta | z|^2\, dx  \leq  \frac{ k_2}{\rho_2\lambda}\left|\int_\alpha^\beta\left(y_{x}+\frac{D}{k_2}z_x\right)_x\,\overline{z}  \, dx\right|+\left(\frac{1}{4}+\frac{k_2 c_{g'}}{\rho_2\lambda^{\frac{1}{2}}} +\frac{ k_1}{\rho_2 \lambda^2 }\right)\int_\alpha^\beta  |z|^2 \, dx+o\left(\lambda^{-3}\right).
\end{equation}
Now, using by parts integration and \eqref{EQ--(4.9)}-\eqref{EQ--(4.10)}, we get
%%%%%%%%%%%%%%%%%%%%Equation%%%%%%%%%%%%%%%%%%%%%%
\begin{equation}\label{EQ--(4.29)}
\begin{array}{ll}

\displaystyle{
\left|\int_\alpha^\beta\left(y_{x}+\frac{D}{k_2}z_x\right)_x\,\overline{z}  \, dx\right|=\left|\left[\left(y_{x}+\frac{D}{k_2}z_x\right)\,\overline{z} \right]^\beta_\alpha-\int_\alpha^\beta\left(y_{x}+\frac{D}{k_2}z_x\right)  \overline{z}_x \, dx\right|}\nline

\displaystyle{
\leq \left|\left(y_{x}+\frac{D}{k_2}z_x\right)(\beta)\right| \left|{z}(\beta)\right|+\left|\left(y_{x}+\frac{D}{k_2}z_x\right)(\alpha)\right| \left|{z}(\alpha)\right|+\int_\alpha^\beta\left|y_{x}+\frac{D}{k_2}z_x\right| \left|{z}_x\right| \, dx
}
\nline
\displaystyle{
\leq \left|\left(y_{x}+\frac{D}{k_2}z_x\right)(\beta)\right| \left|{z}(\beta)\right|+\left|\left(y_{x}+\frac{D}{k_2}z_x\right)(\alpha)\right| \left|{z}(\alpha)\right|+o\left(\lambda^{-2}\right).
}
\end{array}
\end{equation}
Inserting \eqref{EQ--(4.29)} in \eqref{EQ--(4.28)}, we get
%%%%%%%%%%%%%%%%%%%%Equation%%%%%%%%%%%%%%%%%%%%%%
\begin{equation}\label{EQ--(4.30)}
\begin{array}{ll}

\displaystyle{
\left(\frac{3}{4}-\frac{k_2 c_{g'}}{\rho_2\lambda^{\frac{1}{2}}} -\frac{ k_1}{\rho_2 \lambda^2 }\right)\int_\alpha^\beta | z|^2\, dx }
\nline
\displaystyle{ \leq \frac{k_2 }{\rho_2\lambda}\left( \left|\left(y_{x}+\frac{D}{k_2}z_x\right)(\beta)\right| \left|{z}(\beta)\right|+ \left|\left(y_{x}+\frac{D}{k_2}z_x\right)(\alpha)\right| \left|{z}(\alpha)\right|\right) +o\left(\lambda^{-3}\right).
}
\end{array}
\end{equation}
Now, for $\zeta=\beta$ or $\zeta=\alpha$,  we have
%%%%%%%%%%%%%%%%%%%%Equation%%%%%%%%%%%%%%%%%%%%%%
\begin{equation*}
\frac{ k_2}{ \rho_2\lambda} \left|\left(y_{x}+\frac{D}{k_2}z_x\right)(\zeta)\right| \left|{z}(\zeta)\right|\leq 
\frac{ k_2\, \lambda^{-\frac{1}{2}}}{2\rho_2}|z(\zeta)|^2+\frac{ k_2\, \lambda^{-\frac{3}{2}}}{2\rho_2}\left|\left(y_{x}+\frac{D}{k_2}z_x\right)(\zeta) \right|^2.
\end{equation*}
Inserting the above equation in \eqref{EQ--(4.30)}, we get
%%%%%%%%%%%%%%%%%%%%Equation%%%%%%%%%%%%%%%%%%%%%%
\begin{equation*}
\begin{array}{ll}

\displaystyle{
\left(\frac{3}{4}-\frac{k_2 c_{g'}}{\rho_2\lambda^{\frac{1}{2}}} -\frac{ k_1}{\rho_2 \lambda^2 }\right)\int_\alpha^\beta | z|^2\, dx  \leq 

\frac{ k_2\, \lambda^{-\frac{3}{2}}}{2\rho_2}\left(\left|\left(y_{x}+\frac{D}{k_2}z_x\right)(\alpha) \right|^2+\left|\left(y_{x}+\frac{D}{k_2}z_x\right)(\beta) \right|^2\right)

}\nline\hspace{2cm}

\displaystyle{+\frac{ k_2\, \lambda^{-\frac{1}{2}}}{2\rho_2}\left(|z(\alpha)|^2+|z(\beta)|^2\right)+o\left(\lambda^{-3}\right).
}
\end{array}
\end{equation*}
Inserting Equations \eqref{EQ--(4.11)} and \eqref{EQ--(4.12)} in the above inequality, we obtain 
%%%%%%%%%%%%%%%%%%%%Equation%%%%%%%%%%%%%%%%%%%%%%
\begin{equation*}
\left(\frac{3}{4}-\frac{k_2 c_{g'}}{\rho_2\lambda^{\frac{1}{2}}} -\frac{ k_1}{\rho_2 \lambda^2 }\right)\int_\alpha^\beta | z|^2\, dx \leq  \left(\frac{1}{2}+\frac{k_2\, c_{g'}}{\rho_2 \lambda^{\frac{1}{2}}  }\right)
 \int_\alpha^\beta  |z|^2 \, dx+o\left(\lambda^{-\frac{5}{2}}\right),
\end{equation*}
consequently,
%%%%%%%%%%%%%%%%%%%%Equation%%%%%%%%%%%%%%%%%%%%%%
\begin{equation*}
\left(\frac{1}{4}-\frac{2k_2 c_{g'}}{\rho_2\lambda^{\frac{1}{2}}} -\frac{ k_1}{\rho_2 \lambda^2 }\right)
 \int_\alpha^\beta  |z|^2 \, dx\leq  o\left(\lambda^{-\frac{5}{2}}\right),
\end{equation*}
since $\lambda\to+\infty$, for $\lambda$ large enough, we get 
%%%%%%%%%%%%%%%%%%%%Equation%%%%%%%%%%%%%%%%%%%%%%
\begin{equation*}
 0 < \left(\frac{1}{4}-\frac{2k_2 c_{g'}}{\rho_2\lambda^{\frac{1}{2}}} -\frac{ k_1}{\rho_2 \lambda^2 }\right)  \int_\alpha^\beta  |z|^2 \, dx\leq o\left(\lambda^{-\frac{5}{2}}\right), 
\end{equation*}
hence, we get the first asymptotic estimate of  \eqref{EQ--(4.20)}. Then, inserting the first asymptotic estimate of  \eqref{EQ--(4.20)} in \eqref{EQ--(4.6)}, we get the second asymptotic estimate of  \eqref{EQ--(4.20)}. Finally, inserting \eqref{EQ--(4.20)} in \eqref{EQ--(4.12)}, we get  \eqref{EQ--(4.21)}. Thus, the proof is complete.
\end{proof}
%%%%%%%%%%%%%%%%%%%%%%%%%%%%%%%%%%%%%%%%%%%%%%%%%%
                % Lemma %
%%%%%%%%%%%%%%%%%%%%%%%%%%%%%%%%%%%%%%%%%%%%%%%%%%
\begin{lem}\label{Lemma-4.7} 
\rm{Under hypothesis {\rm (H)}, for $j=1,2,$   for $\lambda$ large enough,  we have
%%%%%%%%%%%%%%%%%%%%Equation%%%%%%%%%%%%%%%%%%%%%%
\begin{equation}\label{EQ--(4.31)}
\int_\alpha^\beta  |u_x|^2 \, dx= o\left(1\right)\ \ \ \text{and} \ \ \ \int_\alpha^\beta  |v|^2 \, dx= o\left(1\right).
\end{equation}}
\end{lem}
%%%%%%%%%%%%%%%%%%%%%%%%%%%%%%%%%%%%%%%%%%%%%%%%%%
                % Proof of  Lemma  %
%%%%%%%%%%%%%%%%%%%%%%%%%%%%%%%%%%%%%%%%%%%%%%%%%%
\begin{proof} The proof is divided into two  steps.\\[0.1in]
%%%%%%%%%%%%%%%%%%%%%%Step%%%%%%%%%%%%%%%%%%%%%%%%
\textbf{Step 1.} In this step, we prove the  first  asymptotic behavior estimate of \eqref{EQ--(4.31)}. First, multiplying Equation \eqref{EQ--(4.7)} by $ \frac{\rho_2}{k_1}\left(\overline{u}_x+\overline{y}\right)$   and integrating over   $(\alpha,\beta)$,  we get
%%%%%%%%%%%%%%%%%%%%Equation%%%%%%%%%%%%%%%%%%%%%%
\begin{equation*}
\int_{\alpha}^\beta\left|u_x+y\right|^2dx- \frac{k_2}{k_1}\int_{\alpha}^\beta\left(y_{x}+\frac{D}{k_2}z_x\right)_x\left(\overline{u}_x+\overline{y}\right)dx=-\frac{i\rho_2\lambda}{k_1} \int_{\alpha}^\beta z\left(\overline{u}_x+\overline{y}\right)dx+\frac{\rho_2}{k_1\lambda^{2}}\int_{\alpha}^\beta f_{4}\left(\overline{u}_x+\overline{y}\right)dx,
\end{equation*}
using by parts integration in the second term in the left hand side of above equation, we get
%%%%%%%%%%%%%%%%%%%%Equation%%%%%%%%%%%%%%%%%%%%%%
\begin{equation}\label{EQ--(4.32)}
\begin{array}{ll}
\displaystyle{
\int_{\alpha}^\beta\left|u_x+y\right|^2dx+ \frac{k_2}{k_1}\int_{\alpha}^\beta\left(y_{x}+\frac{D}{k_2}z_x\right)\left(\overline{u}_x+\overline{y}\right)_xdx

=\frac{k_2}{k_1}\left[\left(y_{x}+\frac{D}{k_2}z_x\right)\left(\overline{u}_x+\overline{y}\right)\right]_{\alpha}^\beta}\nline
\displaystyle{
-\frac{i\rho_2\lambda}{k_1} \int_{\alpha}^\beta z\left(\overline{u}_x+\overline{y}\right)dx+\frac{\rho_2}{k_1\lambda^{2}}\int_{\alpha}^\beta f_{4}\left(\overline{u}_x+\overline{y}\right)dx}.
\end{array}
\end{equation}
Next, multiplying Equation \eqref{EQ--(4.5)} by $ \frac{\rho_1k_2}{k_1^2} \left(\overline{y}_x+\frac{D}{k_2}\overline{z}_x\right)$   and integrating over  $(\alpha,\beta)$, then  using  the fact that $f_2\to 0$ in $L^2(0,L)$ and  Equations  \eqref{EQ--(4.9)}-\eqref{EQ--(4.10)},  we get
%%%%%%%%%%%%%%%%%%%%Equation%%%%%%%%%%%%%%%%%%%%%%
\begin{equation*}
 -\frac{k_2}{k_1 }\int_{\alpha}^\beta \left(\overline{y}_x+\frac{D}{k_2}\overline{z}_x\right)\left(u_x+y\right)_xdx
=-\frac{i\rho_1k_2\, \lambda }{k_1^2} \int_{\alpha}^\beta v \left(\overline{y}_x+\frac{D}{k_2}\overline{z}_x\right)dx+\frac{\rho_1k_2}{k_1^2\lambda^{2}}\int_{\alpha}^\beta f_{2}\left(\overline{y}_x+\frac{D}{k_2}\overline{z}_x\right)dx,
\end{equation*}
consequently, 
%%%%%%%%%%%%%%%%%%%%Equation%%%%%%%%%%%%%%%%%%%%%%
\begin{equation}\label{EQ--(4.33)}
 -\frac{k_2}{k_1 }\int_{\alpha}^\beta \left({y}_x+\frac{D}{k_2}{z}_x\right)\left(\overline{u}_x+\overline{y}\right)_xdx
=\frac{i\rho_1k_2\, \lambda }{k_1^2} \int_{\alpha}^\beta \overline{v} \left({y}_x+\frac{D}{k_2}{z}_x\right)dx+\frac{\rho_1k_2}{k_1^2 \lambda^{2}}\int_{\alpha}^\beta \overline{f}_{2}\left({y}_x+\frac{D}{k_2}{z}_x\right)dx.
\end{equation}
Adding \eqref{EQ--(4.32)} and \eqref{EQ--(4.33)}, we obtain
%%%%%%%%%%%%%%%%%%%%Equation%%%%%%%%%%%%%%%%%%%%%%
\begin{equation*}
\begin{array}{ll}
\displaystyle{
\int_{\alpha}^\beta\left|u_x+y\right|^2dx

=-\frac{i\rho_2\lambda}{k_1} \int_{\alpha}^\beta z\left(\overline{u}_x+\overline{y}\right)dx+\frac{k_2}{k_1}\left[\left(y_{x}+\frac{D}{k_2}z_x\right)\left(\overline{u}_x+\overline{y}\right)\right]_{\alpha}^\beta}\nline
\displaystyle{
+\frac{i\rho_1k_2\, \lambda }{k_1^2} \int_{\alpha}^\beta \overline{v} \left({y}_x+\frac{D}{k_2}{z}_x\right)dx+\frac{\rho_2}{k_1\lambda^{2}}\int_{\alpha}^\beta f_{4}\left(\overline{u}_x+\overline{y}\right)dx+\frac{\rho_1k_2}{k_1^2 \lambda^{2}}\int_{\alpha}^\beta \overline{f}_{2}\left({y}_x+\frac{D}{k_2}{z}_x\right)dx
},
\end{array}
\end{equation*}
therefore
%%%%%%%%%%%%%%%%%%%%Equation%%%%%%%%%%%%%%%%%%%%%%
\begin{equation}\label{EQ--(4.34)}
\begin{array}{ll}
\displaystyle{
\int_{\alpha}^\beta\left|u_x+y\right|^2dx

\leq\frac{\rho_2\lambda}{k_1} \int_{\alpha}^\beta \left|z\right|\left|{u}_x+{y}\right|dx+\frac{k_2}{k_1}
\left|\left(y_{x}+\frac{D}{k_2}z_x\right)(\beta)\right|\left|{u}_x(\beta)+{y}(\beta)\right|

}\nline
\displaystyle{+\frac{k_2}{k_1}
\left|\left(y_{x}+\frac{D}{k_2}z_x\right)(\alpha)\right|\left|{u}_x(\alpha)+{y}(\alpha)\right|
+\frac{\rho_1k_2\, \lambda }{k_1^2} \int_{\alpha}^\beta \left|{v}\right| \left|{y}_x+\frac{D}{k_2}{z}_x\right|dx}\nline
\displaystyle{+\frac{\rho_2}{k_1\lambda^{2}}\int_{\alpha}^\beta \left|f_{4}\right|\left|{u}_x+{y}\right|dx+\frac{\rho_1k_2}{k_1^2 \lambda^{2}}\int_{\alpha}^\beta \left|{f}_{2}\right|\left|{y}_x+\frac{D}{k_2}{z}_x\right|dx
}.
\end{array}
\end{equation}
From \eqref{EQ--(4.3)}, \eqref{EQ--(4.9)}, \eqref{EQ--(4.10)}, \eqref{EQ--(4.16)},  \eqref{EQ--(4.20)}, \eqref{EQ--(4.21)} and  the fact that $v,\ u_x+y$ are uniformly bounded in $L^2(\alpha,\beta)$, we obtain 
%%%%%%%%%%%%%%%%%%%%Equation%%%%%%%%%%%%%%%%%%%%%%
\begin{equation*}
\left\{
\begin{array}{ll}
\displaystyle{\left|\left(y_{x}+\frac{D}{k_2}z_x\right)(\beta)\right|\left|{u}_x(\beta)+{y}(\beta)\right|=o\left(\lambda^{-\frac{1}{2}}\right)},\ \displaystyle{\left|\left(y_{x}+\frac{D}{k_2}z_x\right)(\alpha)\right|\left|{u}_x(\alpha)+{y}(\alpha)\right|=o\left(\lambda^{-\frac{1}{2}}\right)},\nline

\displaystyle{\lambda \int_{\alpha}^\beta \left|z\right|\left|{u}_x+{y}\right|dx=o\left(\lambda^{-\frac{1}{4}}\right)
,\ \lambda\int_{\alpha}^\beta \left|{v}\right| \left|{y}_x+\frac{D}{k_2}{z}_x\right|dx=o(1)},\nline
\displaystyle{\lambda^{-2}\int_{\alpha}^\beta \left|f_{4}\right|\left|{u}_x+{y}\right|dx=o\left(\lambda^{-2}\right),\ 
\lambda^{-2}\int_{\alpha}^\beta \left|{f}_{2}\right|\left|{y}_x+\frac{D}{k_2}{z}_x\right|dx=o\left(\lambda^{-3}\right)}.
\end{array}
\right.
\end{equation*}
Inserting  the above equation in \eqref{EQ--(4.34)}, we get
%%%%%%%%%%%%%%%%%%%%Equation%%%%%%%%%%%%%%%%%%%%%%
\begin{equation*}
\int_{\alpha}^\beta\left|u_x+y\right|^2dx=o(1).
\end{equation*}
From the above equation and \eqref{EQ--(4.20)}, we get  the first asymptotic estimate of  \eqref{EQ--(4.31)}.\\[0.1in]
%%%%%%%%%%%%%%%%%%%%%%Step%%%%%%%%%%%%%%%%%%%%%%%%
\textbf{Step 2.} In this step, we prove the  second asymptotic behavior estimate of \eqref{EQ--(4.31)}. Multiplying  \eqref{EQ--(4.5)} by $-i\lambda^{-1} \overline{v}$ and integrating over $(\alpha,\beta),$ then taking the real part, we get
%%%%%%%%%%%%%%%%%%%%Equation%%%%%%%%%%%%%%%%%%%%%%
\begin{equation*}
\int_\alpha^\beta\left| v\right|^2dx=-\frac{ k_1  }{\rho_1\lambda }\Re\left\{i\int_\alpha^\beta(u_x+y)_x\overline{v}\, dx\right\} -\lambda^{-3}\Re\left\{i\int_\alpha^\beta f_{2}\overline{v}\, dx\right\},
\end{equation*}
 using by parts integration in the second term in the right hand side of above equation, we get
%%%%%%%%%%%%%%%%%%%%Equation%%%%%%%%%%%%%%%%%%%%%%
\begin{equation*}
\int_\alpha^\beta\left| v\right|^2dx=\frac{ k_1  }{\rho_1\lambda }\Re\left\{i\int_\alpha^\beta \left(u_x+y\right){\overline{v}_x}dx\right\} -\frac{k_1}{\rho_1 {\lambda} }\Re\left\{i \left[(u_x+y)\overline{v}\right]_\alpha^\beta\right\}-\lambda^{-3}\Re\left\{i\int_\alpha^\beta f_{2}\overline{v}\, dx\right\}.
\end{equation*}
consequently, 
%%%%%%%%%%%%%%%%%%%%Equation%%%%%%%%%%%%%%%%%%%%%%
\begin{equation}\label{EQ--(4.35)}
\begin{array}{ll}
\displaystyle{
\int_\alpha^\beta\left| v\right|^2dx\leq \frac{ k_1  }{\rho_1\lambda }\int_\alpha^\beta \left|u_x+y\right|\left|{v}_x\right| dx +\frac{k_1}{\rho_1 {\lambda} }\left( |u_x(\beta)+y(\beta)||{v}(\beta)|+|u_x(\alpha)+y(\alpha)|{v}(\alpha)|\right)}\nline \displaystyle{+\lambda^{-3}\int_\alpha^\beta \left|f_{2}\right|\left|{v}\right| dx.}
\end{array}
\end{equation}
Finally, from \eqref{EQ--(4.16)}, \eqref{EQ--(4.18)}, \eqref{EQ--(4.20)}, the first asymptotic behavior estimate of   \eqref{EQ--(4.31)}, the fact that $\lambda^{-1} v_x,\ v$ are uniformly bounded in $L^2(\alpha,\beta)$ and the fact that $f_2\to0$ in $L^2(\alpha,\beta)$, we get the second asymptotic behavior estimate of \eqref{EQ--(4.20)}.   Thus, the proof is complete.
\end{proof}$\\[0.1in]$
%%%%%%%%%%%%%%%%%%%%%%%%%%%%%%%%%%%%%%%%%%%%%%%%%%
From what precedes, under hypothesis {\rm (H)}, for $j=1,2,$  from Lemmas \ref{Lemma-4.5}, \ref{Lemma-4.6} and \ref{Lemma-4.7},     we deduce that
%%%%%%%%%%%%%%%%%%%%Equation%%%%%%%%%%%%%%%%%%%%%%
\begin{equation}\label{EQ--(4.36)}
\|U\|_{\mathcal{H}_j}=o\left(1\right),\quad\text{over } \left(\alpha,\beta\right).
\end{equation}
%%%%%%%%%%%%%%%%%%%%%%%%%%%%%%%%%%%%%%%%%%%%%%%%%%
                % Lemma %
%%%%%%%%%%%%%%%%%%%%%%%%%%%%%%%%%%%%%%%%%%%%%%%%%%
\begin{lem}\label{Lemma-4.8}
\rm{Under hypothesis {\rm (H)}, for $j=1,2,$    we have
%%%%%%%%%%%%%%%%%%%%Equation%%%%%%%%%%%%%%%%%%%%%%
 \begin{equation*}
\|U\|_{\mathcal{H}_j}=o\left(1\right),\quad\text{over } \left(0,L\right).
\end{equation*}}
\end{lem}
%%%%%%%%%%%%%%%%%%%%%%%%%%%%%%%%%%%%%%%%%%%%%%%%%%
                % Proof of  Lemma  %
%%%%%%%%%%%%%%%%%%%%%%%%%%%%%%%%%%%%%%%%%%%%%%%%%%
\begin{proof}  Let $\phi\in H^1_0\left(0,L\right)$ be a given function. We proceed the proof in two steps.\\[0.1in]
%%%%%%%%%%%%%%%%%%%%%%Step%%%%%%%%%%%%%%%%%%%%%%%%
\textbf{Step 1.} Multiplying Equation \eqref{EQ--(4.5)} by $2{\rho_1 } \phi  \overline{u}_x$ and integrating over $(\alpha,\beta),$  then using the fact that $u_x$ is bounded in $L^2(0,L)$, $f_2\to 0$ in $ L^2(0,L)$,  and use Dirichlet boundary conditions to get
%%%%%%%%%%%%%%%%%%%%Equation%%%%%%%%%%%%%%%%%%%%%%
\begin{equation}\label{EQ--(4.37)}
\Re\left\{{2i\rho_1 } \lambda   \int_0^L \phi v\overline{u}_x dx\right\}+ k_1\int_0^L \phi' |u_x|^2dx -\Re\left\{2k_1\int_0^L \phi \overline{u}_x y_x  dx\right\} =o(\lambda^{-2}) .
\end{equation}
From \eqref{EQ--(4.4)}, we have
%%%%%%%%%%%%%%%%%%%%Equation%%%%%%%%%%%%%%%%%%%%%%
\begin{equation*}
i{\lambda}\overline{u}_x=  -\overline{v}_x-\lambda^{-2}(\overline{f}_{1})_x.
\end{equation*}
 Inserting the above equation in \eqref{EQ--(4.37)}, then using the fact that $(f_1)_x\to 0 $ in $L^2(0,L)$ and the fact that $v$ is bounded  in $L^2(0,L)$, we get
%%%%%%%%%%%%%%%%%%%%Equation%%%%%%%%%%%%%%%%%%%%%%
 \begin{equation}\label{EQ--(4.38)}
{\rho_1 }   \int_0^L \phi'|v|^2 dx+k_1 \int_0^L \phi' |u_x|^2dx -\Re\left\{2k_1\int_0^L \phi \overline{u}_x y_x  dx\right\} =o(\lambda^{-2}) .
\end{equation}
 Similarly, multiplying Equation \eqref{EQ--(4.7)} by $2\rho_2  \phi \left( \overline{y}_x+\frac{D}{k_1} \overline{z}_x\right)$ and integrating over $(\alpha,\beta),$  then using by parts integration and Dirichlet boundary conditions to get
%%%%%%%%%%%%%%%%%%%%Equation%%%%%%%%%%%%%%%%%%%%%%
 \begin{equation}\label{EQ--(4.39)}
 \begin{array}{ll}
 \displaystyle{\Re\left\{{2i\rho_2} {\lambda}\int_0^L   \phi  z  \overline{y}_x dx\right\} + k_2\int_0^L   \phi'   \left|y_x+\frac{D}{k_2}z_x\right|^2   dx+\Re\left\{2k_1 \int_0^L   \phi  \overline{y}_x u_xdx\right\}}\nline
  \displaystyle{=-\lambda^{-1}\Re\left\{2k_1\int_0^L   \phi  \lambda y \overline{y}_x dx \right\}-\Re\left\{\frac{2i\rho_2 }{k_1} {\lambda}\int_0^L  D(x) \phi  z  \overline{z}_x dx\right\}-\Re\left\{2 \int_0^L  D(x)  \phi  \overline{z}_x u_xdx\right\}}\nline
   \displaystyle{ -\Re\left\{2\int_0^L D(x)  \phi  \overline{z}_x ydx \right\}+\Re\left\{{2\rho_2 } \lambda^{-2}\int_0^L \phi f_3\overline{y}_xdx\right\} +\Re\left\{\frac{2\rho_2 }{k_1} \lambda^{-2}\int_0^LD(x) \phi f_3\overline{z}_xdx\right\}}.
 \end{array}
 \end{equation}
For all bounded  $h\in L^2(0,L)$, using Cauchy-Schwarz inequality, the first estimation of \eqref{EQ--(4.9)},  and the fact that $D\in L^{\infty}(0,L)$, to obtain 
%%%%%%%%%%%%%%%%%%%%Equation%%%%%%%%%%%%%%%%%%%%%%
\begin{equation}\label{EQ--(4.40)}
\Re\left\{\int_0^L D(x) h \overline{z}_x dx\right\}\leq \left(\sup_{x\in (0,L)}D^{1/2}(x)\right)\left(\int_0^L D(x)  |{z}_x|^2 dx\right)^{1/2} \left(\int_0^L |h|^2 dx\right)^{1/2}=o(\lambda^{-1}).
\end{equation}
From \eqref{EQ--(4.39)} and  using \eqref{EQ--(4.40)}, the fact that $z,\ \lambda y,\ y_x$ are bounded in $L^2(0,L)$, the fact that $f_3\to 0$ in $L^2(0,L)$, we get 
%%%%%%%%%%%%%%%%%%%%Equation%%%%%%%%%%%%%%%%%%%%%%
  \begin{equation}\label{EQ--(4.41)}
\Re\left\{2i\rho_2  {\lambda}\int_0^L   \phi  z  \overline{y}_x dx\right\} +  k_2\int_0^L   \phi'   \left|y_x+\frac{D}{k_2}z_x\right|^2   dx+\Re\left\{2k_1 \int_0^L   \phi  \overline{y}_x u_xdx\right\}=o(1).
 \end{equation}
On the other hand, from \eqref{EQ--(4.6)}, we have
%%%%%%%%%%%%%%%%%%%%Equation%%%%%%%%%%%%%%%%%%%%%%
\begin{equation*}
i{\lambda}\overline{y}_x=  -\overline{z}_x-\lambda^{-2}(\overline{f}_{3})_x.
\end{equation*}
Inserting  the above equation  in \eqref{EQ--(4.41)}, then using the fact that $(f_3)_x\to 0 $ in $L^2(0,L)$ and the fact that $z$ is bounded  in $L^2(0,L)$, we get
%%%%%%%%%%%%%%%%%%%%Equation%%%%%%%%%%%%%%%%%%%%%%
  \begin{equation}\label{EQ--(4.42)}
\rho_2\int_0^L   \phi'  |z|^2dx+k_2\int_0^L   \phi'   \left|y_x+\frac{D}{k_2}z_x\right|^2   dx+\Re\left\{2k_1 \int_0^L   \phi  \overline{y}_x u_xdx\right\}=o(1).
 \end{equation}
Adding \eqref{EQ--(4.38)} and \eqref{EQ--(4.42)}, we get
%%%%%%%%%%%%%%%%%%%%Equation%%%%%%%%%%%%%%%%%%%%%%
\begin{equation}\label{EQ--(4.43)}
\int_0^L \phi'\left(\rho_1 |v|^2+\rho_2 |z|^2+k_1 |u_x|^2+k_2 \left|y_x+\frac{D}{k_2}z_x\right|^2 \right)dx=o(1).
\end{equation}
%%%%%%%%%%%%%%%%%%%%%%Step%%%%%%%%%%%%%%%%%%%%%%%%
\textbf{Step 2.}  Let $\epsilon>0$ such that $\alpha+\epsilon<\beta$ and define the cut-off function $\varsigma_1 \text{ in } C^1\left(\left[0,L\right]\right)$ by
%%%%%%%%%%%%%%%%%%%%Equation%%%%%%%%%%%%%%%%%%%%%%
\begin{equation*}
0\leq \varsigma_1 \leq 1,\ \varsigma_1=1 \text{ on } \left[0,\alpha\right] \text{ and } \varsigma_1=0 \text{ on } \left[\alpha+\epsilon,L\right].
\end{equation*}
Take $\phi=x  \varsigma_1$ in \eqref{EQ--(4.43)}, then use the fact that $\left\|U\right\|_{\mathcal{H}_j}=o\left(1\right)$ on $\left(\alpha,\beta\right)$ (i.e., \eqref{EQ--(4.36)}), the fact that $\alpha<\alpha+\epsilon<\beta$, and \eqref{EQ--(4.9)}-\eqref{EQ--(4.10)}, we get
%%%%%%%%%%%%%%%%%%%%Equation%%%%%%%%%%%%%%%%%%%%%%
\begin{equation}\label{EQ--(4.44)}
\int_0^\alpha \left(\rho_1 |v|^2+\rho_2 |z|^2+k_1 |u_x|^2+k_2 \left|y_x+\frac{D}{k_2}z_x\right|^2 \right)dx=o(1).
\end{equation}
Moreover, using Cauchy-Schwarz inequality, the first estimation of \eqref{EQ--(4.9)}, the fact that $D\in L^{\infty}(0,L)$, and \eqref{EQ--(4.44)}, we get
%%%%%%%%%%%%%%%%%%%%Equation%%%%%%%%%%%%%%%%%%%%%%
\begin{equation}\label{EQ--(4.45)}
\begin{array}{rl}
\displaystyle{\int_0^\alpha|y_x|^2dx}&\leq\displaystyle{ 2\int_0^\alpha\left|y_x+\frac{D}{k_2}z_x\right|^2dx+\frac{2}{k_2^2}\int_0^\alpha D(x)^2\left|z_x\right|^2dx},\nline
&\leq\displaystyle{2\int_0^\alpha\left|y_x+\frac{D}{k_2}z_x\right|^2dx+\frac{2 \left(\sup_{x\in (0,\alpha)}D(x)\right) }{k_2^2}\int_0^\alpha D(x)\left|z_x\right|^2dx},\nline
&=\displaystyle{o(1)}.
\end{array}
\end{equation}
Using \eqref{EQ--(4.44)} and \eqref{EQ--(4.45)}, we get
%%%%%%%%%%%%%%%%%%%%Equation%%%%%%%%%%%%%%%%%%%%%%
\begin{equation*}
\left\|U\right\|_{\mathcal{H}_j}=o\left(1\right)\text{ on }(0,\alpha).
\end{equation*}
Similarly, by symmetry, we can prove that $\left\|U\right\|_{\mathcal{H}_j}=o\left(1\right)\text{ on }(\beta,L)$ and therefore
%%%%%%%%%%%%%%%%%%%%Equation%%%%%%%%%%%%%%%%%%%%%%
\begin{equation*}
\left\|U\right\|_{\mathcal{H}_j}=o\left(1\right)\text{ on }(0,L).
\end{equation*}
Thus, the proof is complete.\end{proof}$\\[0.1in]$
%%%%%%%%%%%%%%%%%%%%%%%%%%%%%%%%%%%%%%%%%%%%%%%%%%
                % Proof of  Theorem  %
%%%%%%%%%%%%%%%%%%%%%%%%%%%%%%%%%%%%%%%%%%%%%%%%%%
\noindent \textbf{Proof of Theorem \ref{Theorem--4.1}.}  Under hypothesis {\rm (H)}, for $j=1,2,$ from Lemma \ref{Lemma-4.8}, we have $\|U\|_{\mathcal{H}_j}=o\left(1\right),$ over $(0,L)$, which contradicts \eqref{EQ--(4.2)}. This implies that $$\displaystyle{\sup_{\lambda\in\mathbb{R}}\left\|\left(i\lambda Id-\mathcal{A}_j\right)^{-1}\right\|_{\mathcal{L}\left(\mathcal{H}_j\right)}=O\left(\lambda^{2}\right)}.$$ The result follows from  Theorem \ref{Theorem--2.5} part (i). \xqed{$\square$}\\[0.1in]
%%%%%%%%%%%%%%%%%%%%%%%%%%%%%%%%%%%%%%%%%%%%%%%%%%	
It is very important to ask the question about the optimality of \eqref{E--(4.1)}. For the optimality of \eqref{E--(4.1)},  we first recall Theorem 3.4.1 stated  in \cite{Najdi-Thesis}.
%%%%%%%%%%%%%%%%%%%%%%%%%%%%%%%%%%%%%%%%%%%%%%%%%%
                % Theorem %
%%%%%%%%%%%%%%%%%%%%%%%%%%%%%%%%%%%%%%%%%%%%%%%%%% 
\begin{thm}\label{Theorem-4.9} 
\rm{Let $A:D(A)\subset H\to H $  generate a C$_0-$semigroup of contractions $\left(e^{t A}\right)_{t\geq0}$  on $H$.  Assume that $i\mathbb{R}\in  \rho(A)$. Let $\left(\lambda_{k,n}\right)_{1\leq k\leq k_0,\ n\geq 1}$ denote the k-th branch of eigenvalues of $A$ and $\left(e_{k,n}\right)_{1\leq k\leq k_0,\ n\geq 1}$ the system of normalized associated eigenvectors. Assume that for each $1\leq k\leq k_0$ there exist a positive sequence $\mu_{k,n}\to \infty$ as $n\to \infty$ and two positive constant $\alpha_k>0, \beta_k>0$ such that 
%%%%%%%%%%%%%%%%%%%%Equation%%%%%%%%%%%%%%%%%%%%%%
\begin{equation}\label{EQ--(4.46)}
\Re(\lambda_{k,n})\sim - \frac{\beta_k}{\mu_{k,n}^{\alpha_k}} \ \ \ \text{and}\ \ \ \Im(\lambda_{k,n})\sim \mu_{k,n}\ \ \ \text{as } n\to \infty.
\end{equation}
Here $\Im$ is used to denote the imaginary part of a complex number. Furthermore, assume that  for  $u_0\in D(A)$, there exists constant $M>0$ independent of $u_0$ such that
%%%%%%%%%%%%%%%%%%%%Equation%%%%%%%%%%%%%%%%%%%%%%
  \begin{equation}\label{EQ--(4.47)}
  \left\|e^{t A}u_0\right\|_{H}^2\leq \frac{M}{t^{\frac{2}{\ell_k}}}\left\|u_0\right\|_{D(A)}^2,\ \ \ell_k=\max_{1\leq k\leq k_0} \alpha_k,\ \ \forall \ t>0.
  \end{equation}
Then the decay rate \eqref{EQ--(4.47)} is optimal in the sense that for any $\epsilon>0$  we cannot expect the energy decay rate $t^{-\frac{2}{\ell_k}-\epsilon}.$\xqed{$\square$}
}\end{thm}
%%%%%%%%%%%%%%%%%%%%%%%%%%%%%%%%%%%%%%%%%%%%%%%%%%
\noindent In the next corollary, we show that the optimality of \eqref{E--(4.1)} in some cases.
%%%%%%%%%%%%%%%%%%%%%%%%%%%%%%%%%%%%%%%%%%%%%%%%%%
                % Corollary%
%%%%%%%%%%%%%%%%%%%%%%%%%%%%%%%%%%%%%%%%%%%%%%%%%%
\begin{cor}\label{Corollary-4.10} 
\rm{   For every $U_0\in D\left(\mathcal{A}_2\right)$, we have the following two cases:
%%%%%%%%%%%%%%%%%%Enumerate%%%%%%%%%%%%%%%%%%%%%%%
\begin{enumerate}
%%%%%%%%%%%%%%%%%%%%%%item%%%%%%%%%%%%%%%%%%%%%%%%
\item[1.] If condition \eqref{E--(3.1)} holds, then the energy decay rate in \eqref{E--(4.1)} is optimal. 
%%%%%%%%%%%%%%%%%%%%%%item%%%%%%%%%%%%%%%%%%%%%%%%
\item[2.] If condition \eqref{E--(3.4)} holds and if there exists  $\kappa_1\in\mathbb{N}$ such that  $c=\sqrt{\frac{k_1}{k_2}}=2\kappa_1 \pi$, then the energy decay rate in \eqref{E--(4.1)} is optimal. 
\end{enumerate}}
\end{cor}
%%%%%%%%%%%%%%%%%%%%%%%%%%%%%%%%%%%%%%%%%%%%%%%%%%
                % Proof of  Corollary  %
%%%%%%%%%%%%%%%%%%%%%%%%%%%%%%%%%%%%%%%%%%%%%%%%%%
\begin{proof} We distinguish two cases:
%%%%%%%%%%%%%%%%%%Enumerate%%%%%%%%%%%%%%%%%%%%%%%
\begin{enumerate}
%%%%%%%%%%%%%%%%%%%%%%item%%%%%%%%%%%%%%%%%%%%%%%%
\item[1.] If condition \eqref{E--(3.1)} holds, then from Theorem  \ref{Theorem--3.1}, for $\epsilon>0\left(\text{small enough}\right)$, we cannot expect the energy decay rate $t^{-\frac{2}{{2-\epsilon}}}$ for all initial data $U_0\in D\left(\mathcal{A}_2\right)$ and for all $t>0.$ Hence  the energy decay rate in \eqref{E--(4.1)} is optimal. 
%%%%%%%%%%%%%%%%%%%%%%item%%%%%%%%%%%%%%%%%%%%%%%%
\item[2.] If condition \eqref{E--(3.4)} holds, first following Theorem \ref{Theorem--4.1}, for all initial data $U_0\in D\left(\mathcal{A}_2\right)$ and for all $t>0,$ we get \eqref{EQ--(4.47)} with $\ell_k=2$. Furthermore, from Proposition \ref{Proposition--3.3} (case 2 and case 3), we remark that: \\[0.1in]
%%%%%%%%%%%%%%%%%%%%%%Case%%%%%%%%%%%%%%%%%%%%%%%%
\noindent \textbf{Case 1.} If there exists  $\kappa_0\in\mathbb{N}$ such that  $c=2\left(2\kappa_0+1\right) \pi$, we have
%%%%%%%%%%%%%%%%%%%%Equation%%%%%%%%%%%%%%%%%%%%%%
\begin{equation*}
\left\{
\begin{array}{ll}
\displaystyle{\Re\left(\lambda_{1,n}\right)\sim  -\frac{1}{\pi^{1/2}  |n|^{1/2}},  \ \ \ \Im\left(\lambda_{1,n}\right)\sim 2 n\pi},\nline
\displaystyle{\Re\left(\lambda_{2,n}\right)\sim  -\frac{ c^2}{16\pi^2 n^2},  \ \ \ \Im\left(\lambda_{2,n}\right)\sim \left(2 n+\frac{3}{2}\right) \pi},
\end{array}
\right.
\end{equation*}
then  \eqref{EQ--(4.46)} holds with   $\alpha_1=\frac{1}{2}$ and $\alpha_2=2$. Therefore, $\ell_{k}=2=\max(\alpha_1,\alpha_2).$ Then, applying Theorem \ref{Theorem-4.9}, we get that the energy decay rate in \eqref{E--(4.1)} is optimal. \\[0.1in]
%%%%%%%%%%%%%%%%%%%%%%Case%%%%%%%%%%%%%%%%%%%%%%%%
\textbf{Case 2.} If there exists  $\kappa_1\in\mathbb{N}$ such that  $c=4\kappa_1 \pi$, we have 
%%%%%%%%%%%%%%%%%%%%Equation%%%%%%%%%%%%%%%%%%%%%%
\begin{equation*}
\left\{
\begin{array}{ll}
\displaystyle{\Re\left(\lambda_{1,n}\right)\sim  -\frac{c^2}{16\pi^2 n^2},  \ \ \ \Im\left(\lambda_{1,n}\right)\sim 2 n\pi},\nline
\displaystyle{\Re\left(\lambda_{2,n}\right)\sim  -\frac{ c^2}{16\pi^2 n^2},  \ \ \ \Im\left(\lambda_{2,n}\right)\sim \left(2 n+1\right) \pi},
\end{array}
\right.
\end{equation*}
then  \eqref{EQ--(4.46)} holds with   $\alpha_1=2$ and $\alpha_2=2$. Therefore, $\ell_{k}=2=\max(\alpha_1,\alpha_2).$ Then, applying Theorem \ref{Theorem-4.9} , we get that the energy decay rate in \eqref{E--(4.1)} is optimal.
\end{enumerate}
\end{proof}
%%%%%%%%%%%%%%%%%%%%%%%%%%%%%%%%%%%%%%%%%%%%%%%%%%
%%%%%%%%%%%%%%%%%%%%%%%%%%%%%%%%%%%%%%%%%%%%%%%%%%
%%%%%%%%%%%%%%%%%%%%%%%%%%%%%%%%%%%%%%%%%%%%%%%%%%
 %  References %
%%%%%%%%%%%%%%%%%%%%%%%%%%%%%%%%%%%%%%%%%%%%%%%%%%
%%%%%%%%%%%%%%%%%%%%%%%%%%%%%%%%%%%%%%%%%%%%%%%%%%	
%%%%%%%%%%%%%%%%%%%%%%%%%%%%%%%%%%%%%%%%%%%%%%%%%%
%\protect\bibliographystyle{abbrv}
%\protect\bibliographystyle{alpha}
%\bibliography{References}

\end{document}